\documentclass[letterpaper]{amsart}
\usepackage{amssymb}
\usepackage{mathtools}      
\usepackage{mathrsfs}
\usepackage{enumitem}

\usepackage[normalem]{ulem}

\usepackage[giveninits=true,style=alphabetic,isbn=false,url=false]{biblatex}
\usepackage{hyperref, cleveref}
\usepackage[dvipsnames]{xcolor}
\hypersetup{colorlinks=true,citecolor=NavyBlue,linkcolor=BrickRed,urlcolor=Black}

\usepackage{tikz-cd}
\tikzset{commutative diagrams/arrow style=Latin Modern}

\usepackage{amsthm}

\usepackage{thmtools}
\declaretheorem[sibling=equation,style=definition]{definition}
\declaretheorem[sibling=equation]{theorem}
\declaretheorem[sibling=equation]{corollary}
\declaretheorem[sibling=equation]{proposition}
\declaretheorem[sibling=equation]{lemma}
\declaretheorem[sibling=equation, style=definition]{remark}
\declaretheorem[sibling=equation, style=definition]{example}

\Crefformat{proposition}{#2Proposition~#1#3}
\Crefformat{theorem}{#2Theorem~#1#3}
\Crefformat{lemma}{#2Lemma~#1#3}
\Crefformat{definition}{#2Definition~#1#3}
\Crefformat{corollary}{#2Corollary~#1#3}
\Crefformat{remark}{#2Remark~#1#3}
\Crefformat{example}{#2Example~#1#3}
\Crefformat{subsection}{#2Section~#1#3}
\Crefformat{section}{#2Chapter~#1#3}

\let\emptyset\varnothing

\DeclareMathOperator{\Hom}{Hom}

\newcommand{\shT}{\mathscr{T}}
\newcommand{\PP}{\mathbb{P}}
\DeclareMathOperator{\Pic}{Pic}
\newcommand{\argbl}{-}
\newcommand{\ft}{\tilde{f}}
\newcommand{\sigmat}{\tilde{\sigma}}
\newcommand{\xit}{\tilde{\xi}}
\newcommand{\mto}{\dashrightarrow}
\DeclareMathOperator{\disc}{disc}
\newcommand{\into}{\hookrightarrow}
\DeclareMathOperator{\id}{id}
\newcommand{\tensor}{\otimes}
\newcommand{\eps}{\varepsilon}

\DeclareMathOperator{\Aut}{Aut}

\DeclareMathOperator{\coker}{coker}
\newcommand{\Dbc}{D_c^b}
\newcommand{\derR}{\mathbf{R}}

\newcommand{\Lie}{\operatorname{Lie}}
\newcommand{\alphat}{\tilde{\alpha}}
\newcommand{\lambdat}{\tilde{\lambda}}
\newcommand{\red}{\mathrm{red}}
\newcommand{\sing}{\mathrm{sing}}
\newcommand{\reg}{\mathrm{reg}}
\newcommand{\sm}{\mathrm{sm}}
\newcommand{\main}{\mathrm{main}}
\newcommand{\coinv}{\mathrm{coinv}}
\newcommand{\prim}{\mathrm{prim}}

\newcommand{\DD}{\mathbf{D}}

\newcommand{\fuj}{\mathscr{C}}
\newenvironment{tcd}{\[\begin{tikzcd}}{\end{tikzcd}\]\ignorespacesafterend}
\newcommand{\RR}{\mathbb{R}}
\newcommand{\ZZ}{\mathbb{Z}}
\newcommand{\QQ}{\mathbb{Q}}
\newcommand{\CC}{\mathbb{C}}
\newcommand{\define}[1]{\emph{#1}}
\newcommand{\menge}[2]{\bigl\{ \thinspace #1 \thinspace\thinspace \big\vert%
\thinspace\thinspace #2 \thinspace \bigr\}}
\newcommand{\MENGE}[2]{\left\{ \thinspace #1 \thinspace\thinspace \middle\vert%
\thinspace\thinspace #2 \thinspace \right\}}
\newcommand{\shO}{\mathscr{O}}
\newcommand{\OX}{\shO_X}
\newcommand{\OB}{\shO_B}
\newcommand{\OY}{\shO_Y}
\DeclareMathOperator{\Supp}{Supp}
\newcommand{\fl}{f_{\ast}}
\newcommand{\iu}{i^{\ast}}
\newcommand{\shI}{\mathcal{I}}
\newcommand{\pt}{\mathit{pt}}
\DeclareMathOperator{\im}{im}
\DeclareMathOperator{\sgn}{sgn}
\newcommand{\abs}[1]{\lvert #1 \rvert}
\DeclareMathOperator{\gr}{gr}
\DeclareMathOperator{\codim}{codim}
\newcommand{\NN}{\mathbb{N}}
\DeclareMathOperator{\Alb}{Alb}

\newcommand{\fu}{f^{\ast}}

\newcommand{\pu}{p^{\ast}}
\newcommand{\au}{a^{\ast}}
\newcommand{\al}{a_{\ast}}

\newcommand{\ql}{q_{\ast}}
\newcommand{\decal}[1]{\lbrack #1 \rbrack}
\newcommand{\sltwo}{\mathfrak{sl}_2(\QQ)}
\DeclareMathOperator{\ad}{ad}
\DeclareMathOperator{\End}{End}
\newcommand{\IC}{\mathit{IC}}
\newcommand{\locC}{\mathcal{C}}
\newcommand{\locHZ}{\mathcal{H}_{\ZZ}}
\newcommand{\locL}{\mathcal{L}}
\newcommand{\locT}{\mathcal{T}}
\newcommand{\SP}{P^+}
\newcommand{\contr}{\mathrel{\lrcorner}}

\newcommand{\shHom}{\mathscr{H}\hspace{-1.3pt}om}
\newcommand{\Xt}{\tilde{X}}
\newcommand{\Xreg}{X_{\mathrm{reg}}}
\newcommand{\Ureg}{U_{\mathrm{reg}}}

\definecolor{darkgreen}{rgb}{0.0, 0.7, 0.0}
\newenvironment{??}{\noindent \color{darkgreen}{\bf ???:} \footnotesize}{}
\definecolor{cyan}{cmyk}{1,0,0,0}

\newcommand{\bdg}{\begin{dg}}
\newcommand{\shH}{\mathcal{H}}

\newtheorem{introtheorem}{Theorem}

\Crefformat{introtheorem}{#2Theorem~#1#3}
\Crefmultiformat{introtheorem}{Theorems~#2#1#3}{ and~#2#1#3}{, #2#1#3}{, and~#2#1#3}

\begin{document}

\title[Group Actions and the Support Theorem]{Meromorphic Group Actions and the Support Theorem for Lagrangian Fibrations}

\author{Mark Andrea de Cataldo}
\address{Department of Mathematics, Stony Brook University, Stony Brook, NY 11794, USA}
\email{mark.decataldo@stonybrook.edu}

\author{Yoon-Joo Kim}
\address{Department of Mathematics, Columbia University, New York, NY 10027, USA}
\email{yk3029@columbia.edu}

\author{Christian Schnell}
\address{Department of Mathematics, Stony Brook University, Stony Brook, NY 11794, USA}
\email{christian.schnell@stonybrook.edu}

\date{\today}

\begin{abstract}
We prove several new results  about Lagrangian fibrations on holomorphic symplectic
complex spaces, under the assumption that the total space is K\"ahler
(but possibly non-compact or singular) and that the base is a complex manifold.
First, we construct a meromorphic action by a family of meromorphic groups.
Second, we use this structure, together with Hodge-theoretic methods, to prove
a version of Ng\^o's support theorem for Lagrangian fibrations. Along the way, we prove a freeness
theorem for the cohomology of compact K\"ahler spaces equipped with a meromorphic
group action.
\end{abstract}

\maketitle

\tableofcontents

\section{Introduction}\label{sec: intro}

In classical mechanics, the Liouville--Arnold Theorem describes completely integrable
Hamiltonian systems. Under suitable regularity and compactness assumptions, the
regular fibers are compact tori, and the dynamics is linear in action--angle
coordinates.

Proper holomorphic symplectic fibrations---also known as Lagrangian
fibrations---provide the complex-geometric analogue of this picture. If
$f \colon X \to B$ is a K\"ahler Lagrangian fibration, then over the locus of regular
values it is a torsor under a family of compact complex tori, mirroring the
Liouville--Arnold theorem. Unlike the classical situation, the fibration is
typically not holomorphically trivial, and the torus structure does not extend
 across the singular fibers.

This raises  the following question:  given a Lagrangian fibration $f \colon X \to B$, can one construct a family of
commutative Lie groups acting faithfully  on $X$ over the entire base $B$,
rather than only over the regular locus?

In this paper, we answer this question affirmatively
(\Cref{thm:intro-A}). We then establish a freeness theorem for the
action of these groups on the cohomology of the fibers
(\Cref{thm:intro-B}), which we use to prove a support theorem for
Lagrangian fibrations (\Cref{thm:intro-C}).

\begin{introtheorem}\label{thm:intro-A}
Let $f \colon X \to B$ be a Lagrangian fibration, with $X$ K\"ahler but possibly
singular, and $B$ smooth. Then there exists a smooth  family $G \to B$  of commutative Lie groups
 acting faithfully on $X$ over $B$.
\end{introtheorem}

This is \Cref{prop:family-groups} in the body of the text.
More precisely, the family $G$ carries a natural structure of a family of
meromorphic groups in the sense of Fujiki \cite{fujiki-auto}, and the action is meromorphic. This
additional structure is automatic in the algebraic category but requires
substantial new arguments in the K\"ahler setting.

Our second main result is the Freeness Theorem (\Cref{thm:freeness} in the body of the text),
which we state here in the context of Lagrangian fibrations and in terms of modules rather than comodules. All cohomology groups are taken with real coefficients. Let $b\in B$, and let $X_b$ denote the corresponding fiber of the Lagrangian fibration $X/B$. Let $T_b$ be the maximal compact torus quotient
of  the identity component $G^\circ_b$ of  the complex Lie group $G_b$.

\begin{introtheorem}[\textbf{The Freeness Theorem}]\label{thm:intro-B}
The singular cohomology $H^*(X_b,\mathbb{R})$ is  a free module over the cohomology algebra $H^*(T_b,\mathbb{R})$.
\end{introtheorem}
\Cref{pr: K weq} provides the natural generalization to $H^*(X_b,K)$, where $K$ is any weakly equivariant  constructible complex under the action of $G^\circ_b$. This includes, in particular, the cases in which $K$ is the constant sheaf or, when $X$ is singular,  the restriction of the intersection complex of $X$ to $X_b$; this special case plays a crucial role in the proof of the Support Theorem.

The group action, the Freeness Theorem, and substantial Hodge-theoretic input are then used to prove the third main result of this paper, the Support Theorem for a Lagrangian fibration $f:X\to B$.

Before stating the theorem, let us briefly explain its geometric meaning. The
support theorem relates the Decomposition Theorem for the direct image complex $\derR \fl IC_X$  (or $\derR \fl \RR_X[\dim X]$ when $X$ is smooth)
to the geometry of the family
of groups constructed in \Cref{thm:intro-A}. Roughly speaking, the possible
supports and the corresponding direct summands of the  direct image complex  are
governed by the compact torus quotients of this family of groups. We give the statement for Lagrangian fibrations;
a more general statement can be found in \Cref{thm:support}.

\begin{introtheorem}[\textbf{The Support Theorem}]\label{thm:intro-C}
Let $f \colon X \to B$ be a Lagrangian fibration with $X$ K\"ahler. If
$Z \subseteq B$ is a support appearing in the Decomposition Theorem, then the
dimension of $Z$ is equal to the dimension of the maximal compact torus quotient $T_b$
of $G_b^\circ$ at a general point $b\in Z$. Moreover, the direct sum of the simple perverse sheaves supported on $Z$
is built from the cohomology of the corresponding
compact torus quotients and a local system with finite monodromy.
\end{introtheorem}

In the body of the text, this is
\Cref{thm:LF-support} for $X$ smooth, and \Cref{thm:singLF-support} for $X$ singular.
We need the assumption that $B$ is a manifold in an essential way during the construction
of the family of meromorphic groups $G \to B$.  On the other hand, the smoothness of $X$ plays no role and  our results hold for $X$ singular.

The K\"ahler assumption, which is automatic  when $X$ is quasi-projective, is also essential. It is used in three distinct places. First, it is needed for the construction of the family of Lie groups $G/B$, where we rely on Fujiki's work on Douady spaces for proper K\"ahler fibrations. Second, it is needed in the proof of the Freeness Theorem, where splittings of mixed Hodge structures are required to set up the relevant linear-algebraic arguments. Third, it is needed in the proof of the Support Theorem, where we rely on Saito's Decomposition Theorem together with delicate Hodge-theoretic arguments showing that the perverse filtration is compatible with the module structure.

The reason for using $\RR$-coefficients is that, with such
coefficients, the $H^*(T_b)$-module structure in the Freeness Theorem
(\Cref{thm:intro-B}) is canonical and compatible with mixed Hodge theory, and
this compatibility in turn feeds into the proof of the Support Theorem
(\Cref{thm:intro-C}). On the other hand, the freeness statement also holds with
$\mathbb Q$-coefficients, but then the module structure is not canonical and is
not compatible with mixed Hodge theory. Since the Support Theorem is valid over
$\mathbb R$, it also follows over $\mathbb Q$, because all the relevant objects
and splittings can be defined over $\mathbb Q$. (See \Cref{sec:rational} for more on
this point.)

\subsection{Recent work on constructing the group scheme}
Let us briefly summarize the history of constructing the group object $G$ with an action on $X$. For it, we temporarily denote the set of non-critical points of $f$ by $X_{\mathrm{nc}/B} \subseteq X$ and consider the condition
\begin{equation} \label{eq:non-multiple fiber condition}
	f(X_{\mathrm{nc}/B}) = B .
\end{equation}
\cite[Prop~2.1--2.3]{mar96} first observed that integrating the vertical vector fields of $f$ yields a $T^*B$-action on its non-critical locus $X_{\mathrm{nc}/B}$. As a result, he constructed a family of complex Lie groups $G^\circ$ over $B$, together with its action on $X_{\mathrm{nc}/B}$. This idea has been revisited several times in \cite[\S\S 2--3]{hwang-ogu09}, \cite[\S 2.2]{aba-rogov25}, and \cite[Thm~5.1]{sacca24}, and culminated in the construction of the $G^\circ$-action on the full space $X$ under the condition \eqref{eq:non-multiple fiber condition}. On the other hand, \cite[Thm~2]{Arinkin+Fedorov} used a different approach of constructing $G^\circ$ as a subgroup of the relative automorphism group $\Aut(X/B)$ but assuming $f$ has only integral fibers. \cite[Thm~4.2]{kim24} extended this result and observed that a larger group $G$ exists in $\Aut(X/B)$, but still could not eliminate the same condition \eqref{eq:non-multiple fiber condition}. In summary, the construction of $G$ and its action on $X$ was known in the previous work but only under \eqref{eq:non-multiple fiber condition}, which has been used as an essential ingredient. Our result completely eliminates the dependency on \eqref{eq:non-multiple fiber condition} and moreover drops the algebraicity of $f$.

\subsection{Outline of the proof}\label{subs:fam act grp}

In \Cref{sec:Liouville-Arnold}, we use the standard technique of
integrating vertical vector fields to obtain their flows, which define
an action of the cotangent bundle $T^*B$ on the Lagrangian fibration.
This action yields a classifying $B$-morphism
\[
T^*B \to \Aut_{X/B},
\]
where $\Aut_{X/B}$ denotes the group object over $B$ whose fibers are
the automorphism groups of the fibers of $X/B$. It is known that $\Aut_{X/B}$ is represented by a Zariski-open
subspace of the relative Douady space of $X\times_B X$.
We  prove  the key fact used  in \Cref{sec:meromorphic}  that the induced morphism on
Lie algebras is everywhere injective.

\Cref{sec:meromorphic} is devoted to the construction of the family of
meromorphic groups $G$. Let $D$ be the reduction of the irreducible
component of the relative Douady space containing the image of the
classifying morphism. We define $G$ to be the locus of points
$d\in D$ corresponding to graphs of automorphisms of fibers of
$X/B$ at which the natural morphism $D\to B$ is smooth. The heart of
the argument is to show that this construction indeed produces a
family of meromorphic groups. The main difficulty is to prove that $D$
is locally irreducible along the identity section and that the latter
is contained in $G$. Once these facts are established, the geometry of
the relative Douady space of a proper K\"ahler morphism allows us to
endow $G$ with a canonical structure of family of meromorphic groups
acting meromorphically on $X/B$.
We conclude by proving that $G^\circ$ is open in $G$ and is precisely the subgroup generated by the
Liouville--Arnold flows, thereby identifying the infinitesimal construction with the identity component of the global group.

In \Cref{sec:freeness}, we prove the Freeness Theorem. The starting
point is the action of the family of meromorphic groups constructed in
the previous section. The main difficulty is that the spaces and groups
under consideration need not be smooth and that the action is only
meromorphic. Consequently, standard arguments from equivariant topology
do not apply directly. To overcome this problem, we combine the
structure theory of meromorphic groups with mixed Hodge theory and
establish a series of splitting results for the cohomology of the
fibers. The key observation is that the group action endows the
cohomology of the fibers with a large collection of canonical
cohomological operators. Using mixed Hodge theory, we show that these
operators force a sequence of splittings that ultimately imply that the
cohomology groups are free (co-)modules over the cohomology algebra of the
maximal compact quotient of the acting group.

The final step is the proof of the Support Theorem in
\Cref{sec:support-theorem}. The problem is to determine which subvarieties of
the base can occur as supports of summands in the decomposition theorem
for the direct image complex. The geometric information provided by the
canonical family of meromorphic groups, together with the Freeness
Theorem, imposes strong restrictions on the topology of the fibers.
These restrictions translate into constraints on the possible supports,
placing the Lagrangian fibration within the framework of $\delta$-regular weak abelian
fibrations, and allows us to adapt Ng\^o's strategy.
Combining this with the decomposition theorem and the Hodge-theoretic
properties established earlier, we obtain the support theorem and the
description of the corresponding summands.

\subsection{Acknowledgments} We thank Anna Abasehva, Andres Fernandez Herrero,
Eyal Markman, Mirko Mauri,  Giulia Sacc\`a, and Michael Thaddeus for useful conversations.
M.A.dC.\ has been supported by NSF Grants DMS-2200492 and DMS-2502310. Ch.S.\ is supported
by NSF grant DMS-2301526.

\section{The Liouville-Arnold theorem}
\label{sec:Liouville-Arnold}

Let $X$ be a holomorphic symplectic manifold of dimension $2n$; this means that there is
a holomorphic form $\sigma \in H^0(X, \Omega_X^2)$ that is closed ($d\sigma = 0$)
and nondegenerate (in the sense that $\sigma^n$ is a nowhere vanishing section of $\Omega_X^{2n}$).
We consider a Lagrangian fibration $f \colon X \to B$ over an $n$-dimensional complex
manifold $B$; here $f$ is proper and surjective, with connected fibers,
and  the smooth fibers $X_b = f^{-1}(X_b)$ are Lagrangian, in the sense that $\sigma \vert_{X_b} = 0$.
The purpose of this chapter is to prove that the cotangent bundle $T^{\ast} B$ of the
base acts holomorphically on the total space $X$ of the Lagrangian fibration. In
classical mechanics, this kind of result is called the \emph{Liouville-Arnold theorem}.

\subsection{Constructing the action}
\label{par:action}

We begin by constructing a holomorphic mapping
\[
	T^{\ast} B \times_B X \to X.
\]
Let $t_1, \dotsc, t_n$ be local holomorphic coordinates, defined on an open
neighborhood $U \subseteq B$ of a given point. The symplectic form gives
an isomorphism $\sigma \colon \shT_X \to \Omega_X^1$, and so there are $n$
holomorphic vector fields $\xi_1, \dotsc, \xi_n \in \Gamma(f^{-1}(U), \shT_X)$
such that
\[
	\xi_j \contr \sigma = \sigma(\xi_j, \argbl) = \fu dt_j.
\]
These vector fields are tangent to the fibers of the Lagrangian fibration.

\begin{lemma}
For $b \in B$, denote by $X_{b, \red}$ the reduction of the fiber $X_b = f^{-1}(b)$.
Then $\xi_j \in T_x X_{b, \red}$ at any smooth point $x \in X_{b, \red}$.
\end{lemma}

\begin{proof}
Since $f \colon X \to B$ is a Lagrangian fibration, $T_x X_{b, \red}$ is a maximal
isotropic subspace of $T_x X$, provided that $x \in X_{b, \red}$ is a smooth point.
If $X_b$ is smooth, this is the definition of a Lagrangian fibration; and for the singular
fibers, it follows from a theorem by Matsushita \cite[Thm.~1]{mat00}.
For $\xi \in T_x X_{b, \red}$, we have
\[
	\sigma(\xi_j, \xi) = \fu(dt_j)(\xi) = 0,
\]
and therefore $\xi_j \in T_x X_{b, \red}$ by maximality.
\end{proof}

The vector fields $\xi_1, \dotsc, \xi_n$ also commute. Here it matters that
$d \sigma = 0$.

\begin{lemma}
We have $[\xi_i, \xi_j] = 0$ for $1 \leq i,j \leq n$.
\end{lemma}

\begin{proof}
	We use the coordinate-free description of the exterior derivative, namely
	\begin{align*}
			(d \alpha)(\eta_0, \dotsc, \eta_k) &= \sum_{i=0}^k (-1)^i \eta_i \cdot
			\alpha(\eta_0, \dotsc, \widehat{\eta_i}, \dotsc, \eta_k) \\
			&+ \sum_{0 \leq i < j \leq k} (-1)^{i+j} \alpha \bigl( [\eta_i, \eta_j], \eta_0,
			\dotsc, \widehat{\eta_i}, \dotsc, \widehat{\eta_j}, \dotsc, \eta_k \bigr).
	\end{align*}
	Because the $1$-form $\xi_j \contr \sigma = \fu dt_j$ is closed, we have
	\begin{equation} \label{eq:d-one-form}
		\sigma \bigl( \xi_j, [\xi, \eta] \bigr)
		= \xi \cdot \sigma(\xi_j, \eta) - \eta \cdot \sigma(\xi_j, \xi)
	\end{equation}
	for any pair of local vector fields $\xi, \eta \in \shT_X$. From $d \sigma = 0$, we get
	\[
		\xi_i \cdot \sigma(\xi_j, \eta) + \xi_j \cdot \sigma(\eta, \xi_i)
		+ \eta \cdot \sigma(\xi_i, \xi_j) =
		\sigma \bigl( [\xi_i, \xi_j], \eta \bigr) +
		\sigma \bigl( [\xi_j, \eta], \xi_i \bigr) +
		\sigma \bigl( [\eta, \xi_i], \xi_j \bigr),
	\]
	for every $\eta \in \shT_X$. After rearranging the terms and remembering that
	$\sigma(\xi_i, \xi_j) = 0$, this turns into
	\begin{align*}
		\sigma \bigl( [\xi_i, \xi_j], \eta \bigr) &=
		\xi_i \cdot \sigma(\xi_j, \eta) - \xi_j \cdot \sigma(\xi_i, \eta) +
		\sigma \bigl( \xi_i, [\xi_j, \eta] \bigr) - \sigma \bigl( \xi_j, [\xi_i, \eta] \bigr) \\
		&= \xi_i \cdot \sigma(\xi_j, \eta) - \xi_j \cdot \sigma(\xi_i, \eta) +
		\xi_j \cdot \sigma(\xi_i, \eta) - \xi_i \cdot \sigma(\xi_j, \eta) = 0,
	\end{align*}
	where the second equality follows from  \eqref{eq:d-one-form}. Since $\sigma$ is nondegenerate,
	the result follows.
\end{proof}

We now integrate the vector fields into a family of holomorphic automorphisms.
Let $X_U = f^{-1}(U)$, which is an open subset of the complex manifold $X$. The main point
is of course the properness of $f \colon X \to B$.

\begin{proposition}
	For each $j =1, \dotsc, n$, the flow
	\[
		\Phi^j \colon \CC \times X_U \to X_U, \quad \Phi^j(t, x) = \Phi_t^j(x),
	\]
	of the holomorphic vector field $\xi_j$ exists for all times $t \in \CC$, and
	is compatible with the morphism $f \colon X_U \to U$,
	meaning that $f \circ \Phi_t^j = f \circ p_2$.
\end{proposition}

\begin{proof}
This is a pretty standard result: \cite[Lem.~2.4]{Milnor} deals with compact complex manifolds,
and \cite[\S4]{Kaup} with reduced compact complex spaces.
We provide some additional details for completeness.
According to a general result by Kaup \cite[Satz~3]{Kaup} (which is written
in German, but see for example \cite[\S4]{GGK} for an English version), the flow $\Phi_t^j$ of
the holomorphic vector field $\xi_j$ is always defined at least on an open neighborhood
of the zero section in $\CC \times X_U$.

Let $b \in U$ and let $X_b = f^{-1}(b)$ be the corresponding (possibly
nonreduced) fiber. Because $f \colon X \to B$ is proper, there is an open neighborhood
$V \subseteq U$ of the point $b$ and a positive real number $\eps > 0$ such that
the flow on $X_V = f^{-1}(V)$ is defined for all times $\abs{t} < \eps$, a priori
with values in the larger open set $X_U$. This gives us a holomorphic mapping
\[
	\Phi^j \colon \menge{t \in \CC}{\abs{t} < \eps} \times X_V \to X_U
\]
Since $\xi_j$ is tangent to the smooth fibers of $f$, we have
\[
	f \bigl( \Phi^j(t, x) \bigr) = f(x)
\]
as long as $x \in X_V$ is a regular point of $f \colon X \to B$ (and $\abs{t} < \eps$);
because both sides are holomorphic, we get $f \circ \Phi^j = f \circ p_2$ on
the entire domain of $\Phi^j$. Consequently, $\Phi^j$ actually maps $X_V$ into itself.
We can now apply $\Phi^j$ repeatedly; by uniqueness of the flow, it follows that
$\Phi^j(t,x)$ is defined for all $x \in X_V$ and all times $t \in \CC$. As the
point $b \in B$ was arbitrary, we obtain the desired holomorphic mapping
\[
	\Phi^j \colon \CC \times X_U \to X_U
\]
with the property that $f \circ \Phi^j = f \circ p_2$. Again by uniqueness, $\Phi_t^j$ is
a one-parameter group of biholomorphisms of $X_U$ (over $U$): the inverse to
$\Phi_t^j$ is of course $\Phi_{-t}^j$.
\end{proof}

The vector fields $\xi_1, \dotsc, \xi_n$ commute, and so the flows
$\Phi^1, \dotsc, \Phi^n \colon \CC \times X_U \to X_U$ also commute with each other.
We can therefore define
\[
	\Phi_U \colon \CC^n \times X_U \to X_U, \quad
	\Phi_U(z_1, \dotsc, z_n, x) = (\Phi_{z_1}^1 \circ \dotsb \circ \Phi_{z_n}^n)(x),
\]
which is a holomorphic mapping with $f \circ \Phi_U = f \circ p_2$.
By uniqueness, the flow
of any holomorphic vector field of the form $g_1 \xi_1 + \dotsb + g_n \xi_n$
with $g_1, \dotsc, g_n \in \Gamma(X_U, \OX)$ must be given by the formula
\[
	(t, x) \mapsto \Phi_U \bigl( g_1(x) t, \dotsc, g_n(x) t, x \bigr);
\]
this is easily seen by computing the derivative of the right-hand side.

\begin{proposition}
The construction above defines a holomorphic mapping
\[
	\Phi \colon T^{\ast} B \times_B X \to X
\]
that fits into a commutative diagram
\begin{tcd}
T^{\ast} B \times_B X \rar{\Phi} \dar{p_2} & X \dar{f} \\
X \rar{f} & B.
\end{tcd}
\end{proposition}

\begin{proof}
We only need to show that our construction transforms correctly under
coordinate changes. Let $t_1', \dotsc, t_n'$ be another set of
local coordinates on $U$; then
\[
	dt_i' = \sum_{j=1}^n \frac{\partial t_i'}{\partial t_j} dt_j,
\]
and so the corresponding holomorphic vector fields are
\[
	\xi_i' = \sum_{j=1}^n \fu \frac{\partial t_i'}{\partial t_j} \cdot \xi_j.
\]
According to the remark above, the flow of this vector field is
\[
	\CC \times X_U \to X_U, \quad
	(t, x) \mapsto \Phi_U \left( t \fu \frac{\partial t_i'}{\partial t_1},
	\dotsc, t \fu \frac{\partial t_i'}{\partial t_n}, x \right),
\]
and by composing these mappings and using the identity $\Phi^i \circ \Phi^j
= \Phi^j \circ \Phi^i$, we find that the new holomorphic mapping
\[
	\Phi_U' \colon \CC^n \times X_U \to X_U
\]
is related to $\Phi_U \colon \CC^n \times X_U \to X_U$ by the formula
\[
	\Phi_U'(z_1, \dotsc, z_n, x) = \Phi_U \left(
		\sum_{i=1}^n z_i \fu \frac{\partial t_i'}{\partial t_1}, \dotsc,
		\sum_{i=1}^n z_i \fu \frac{\partial t_i'}{\partial t_n}, x
	\right).
\]
It follows that the individual mappings $\Phi_U$ glue together
to a holomorphic mapping $\Phi \colon T^{\ast} B \times_B X \to X$ that of course
still satisfies $f \circ \Phi = f \circ p_2$.
\end{proof}

\subsection{Discreteness of the kernel}

In the next chapter, we are going to prove that the action of the cotangent bundle
$T^{\ast} B$ factors through a meromorphic action by a family of meromorphic groups
$p \colon G \to B$. The crucial condition that we need for this is contained
in the following lemma. Let $X_b = f^{-1}(b)$  be one of the fibers of the
Lagrangian fibration $f \colon X \to B$, and denote by $\Aut^\circ(X_b)$
the neutral component of its automorphism group. Note that $X_b$ may be nonreduced;
in fact, the result will usually be false for the reduction $X_{b, \red}$.

\begin{lemma} \label{lem:faithful-action}
	The morphism $T_b^{\ast} B \to \Lie \Aut^{\circ}(X_b)$ is injective for every $b \in B$.
\end{lemma}

\begin{proof}
After shrinking $B$, we may assume that we have coordinates $t_1, \dotsc, t_n$, and that
the point $b \in B$ is the origin of the coordinate system; then $T_b^{\ast} B \cong \CC^n$.
The Lagrangian fibration $f \colon X \to B$ is now just $n$ holomorphic functions $f_1, \dotsc, f_n$.
We have $n$ commuting holomorphic vector fields $\xi_1, \dotsc, \xi_n \in \Gamma(X, \shT_X)$,
defined by the rule $df_i = \sigma(\xi_i, \argbl)$, and the construction
in \Cref{par:action} integrates these vector fields to a holomorphic mapping
$\CC^n \times X \to X$. We are interested in the scheme-theoretic fiber $X_0 = f^{-1}(0)$,
whose structure sheaf is $\OX/(f_1, \dotsc, f_n)$.

We argue by contradiction. Suppose that $\CC^n \to \Lie \Aut^{\circ}(X_b)$ is not injective. After a linear
change of coordinates, we can arrange that the vector $(1, 0, \dotsc, 0)$ is in the kernel;
this means that the flow $\Phi_t^1$ of the vector field $\xi_1$ acts trivially
on the (possibly nonreduced) subspace $X_0$. This is saying that the flow acts trivially on
the reduction $X_{0, \red}$ of the fiber, and also that
that for every (locally defined) holomorphic function $g$ on $X$, one has
\[
	g - g \circ \Phi_t^1 \in (f_1, \dotsc, f_n).
\]
The infinitesimal version, which we get
by taking the derivative at $t=0$, is that the Lie derivative of $g$ in the direction of
$\xi_1$ lies in the ideal of $X_0$. In symbols,
\begin{equation} \label{eq:Lie-derivative}
	\mathcal{L}_{\xi_1}(g) = dg(\xi_1) \in (f_1, \dotsc, f_n).
\end{equation}
Choose local coordinates $x_1, \dotsc, x_{2n}$ on $X$, centered at an arbitrary point
$x \in X_0$.  Write the vector field as $\xi_1 = a_1 \partial_1 + \dotsb + a_{2n} \partial_{2n}$,
where $a_1, \dotsc, a_{2n}$ are holomorphic functions on a
neighborhood of the point $x$, and $\partial_j = \partial/\partial x_j$.  If we apply \eqref{eq:Lie-derivative} to the functions
$x_1, \dotsc, x_{2n}$, we find that all the coefficients $a_j = dx_j(\xi_1)$ belong to the ideal $(f_1,
\dotsc, f_n)$. By construction,
\[
	\frac{\partial f_1}{\partial x_k} = df_1(\partial_k) = \sigma(\xi_1, \partial_k)
	= \sum_{j=1}^{2n} a_j \sigma(\partial_j, \partial_k),
\]
and so it follows that all the partial
derivatives of $f_1$ in our local coordinate system $x_1, \dotsc, x_{2n}$ also
belong to the ideal $(f_1, \dotsc, f_n)$. In other words, the $1$-form $df_1$ vanishes
on the (possibly nonreduced) subspace $X_0$.

It is now a pleasant exercise to show that this cannot happen. For dimension reasons,
the equations $f_2 = \dotsb = f_n = 0$ define a closed analytic subspace $Y \subseteq X$
of dimension $n+1$; because $\dim X_0 = n$, the image of the holomorphic function
$f_1 \colon Y \to \CC$ must contain an open neighborhood of the origin.
We can therefore find a holomorphic curve $c \colon \Delta \to Y$ such that $c(0) = x$
and such that $h = f_1 \circ c \colon \Delta \to \CC$ is nonconstant.
Since $f_j \circ c = 0$ for $j=2, \dotsc, n$, we deduce that $h(0) = 0$ and that $dh/dt$ is
divisible by $h$. By looking at the power series for $h$, it follows that $h \equiv 0$,
which is a contradiction.
\end{proof}

\begin{example}
The automorphism group of a nonreduced space can be much larger than that of its reduction.
One can see this already in the case of the projective scheme $\operatorname{Proj}
\CC[x,y,z]/(z^2)$: the automorphism group of $\PP^1 = \operatorname{Proj} \CC[x,y]$ is
the $3$-dimensional group $\operatorname{PGL}_2(\CC)$, while the nonreduced scheme
has a $6$-dimensional automorphism group, consisting of all invertible $3 \times 3$-matrices
of the form
\[
	\begin{pmatrix}
	a_0 & a_1 & a_2 \\
	b_0 & b_1 & b_2 \\
	0 & 0 & 1
	\end{pmatrix}.
\]
\end{example}

\section{Families of meromorphic groups} \label{sec:meromorphic}

In the previous chapter, we showed that in any Lagrangian fibration over a smooth
base, the cotangent bundle of the base acts holomorphically on the fibers of the
Lagrangian fibration. We are now going to argue that this action factors through a
meromorphic action by family of meromorphic groups (in the sense of Fujiki); in
particular, this means that when the Lagrangian fibration is algebraic, we get an
algebraic action by a commutative group scheme.

\subsection{Introduction}

In order to make the result as applicable as possible, we are going to work in
somewhat greater generality. Let $X$ be a normal complex space that is Cohen-Macaulay
and K\"ahler (in the sense of \cite[Def.~1.2]{fujiki-douady}), and let $f \colon X \to B$
be a proper holomorphic mapping to a complex
manifold $B$. We assume that $f$ is surjective, that all fibers are connected
and have the same dimension $n$. Consequently, $\dim X = \dim B + n$, and $f$ is flat
of relative dimension $n$. We denote by $X_b = f^{-1}(b)$ the scheme-theoretic
fibers; these are $n$-dimensional compact complex spaces that are K\"ahler but
possibly nonreduced. We further assume that $f$ is generically a \define{torus
fibration}, in the sense that the singular locus $X_{\sing}$ does not map onto $B$,
and that the smooth fibers of $f$ are compact complex tori of dimension $n$. We
denote by $\disc(f) \subseteq B$ the maximal closed analytic subset over which the
fibers of $f$ are singular; $f$ is then a smooth morphism over the Zariski-open
subset $B \setminus \disc(f)$.

Next, let $p \colon E \to B$ be a holomorphic vector bundle of rank $n$, and suppose
that we have a holomorphic action
\begin{equation} \label{eq:action}
	a \colon E \times_B X \to X
\end{equation}
on $X$ over $B$; in the case of a Lagrangian fibration, $E$ is the
cotangent bundle of $B$. We are going to use the same letter $a \colon E_b \times X_b
\to X_b$ to denote the action of the $n$-dimensional complex vector space $E_b =
p^{-1}(b)$ on the fiber $X_b$. The fiberwise action is of course the same thing as a
holomorphic group homomorphism $E_b \to \Aut^{\circ}(X_b)$, where $\Aut^\circ(X_b)$ is the
neutral component of the automorphism group of the compact complex
space $X_b$.

We impose the following two additional conditions on the action in \eqref{eq:action},
suggested by what happens in the case of Lagrangian fibrations:
\begin{enumerate}
\item The homomorphism $E_b \to \Lie \Aut^{\circ}(X_b)$ is injective for every $b \in B$.
\item The homomorphism $E_b \to \Lie \Aut^{\circ}(X_b)$ is surjective when $X_b$ is smooth.
\end{enumerate}
Since a smooth fiber $X_b$ is a compact complex torus of dimension $n$, the group
$\Aut^{\circ}(X_b)$ is the quotient of the $n$-dimensional vector space $\Lie
\Aut^{\circ}(X_b)$ by a lattice; therefore (1) and (2) are saying that, on each smooth fiber, the
group homomorphism $E_b \to \Aut^{\circ}(X_b)$ is surjective with discrete kernel. When $X_b$ is
singular and maybe nonreduced, the dimension of $\Aut^{\circ}(X_b)$ can be much larger than $n$;
but the condition in (1) is saying that $E_b \to \Aut^{\circ}(X_b)$ still has a discrete kernel. We
are going to describe this kernel in more detail later on (see \Cref{rmk:Lambda} and
\Cref{subs: netr}).

\begin{remark}
For (2), the nonreduced structure of the fibers matters very much. Indeed,
$E_b$ can act nontrivially on $X_b$ even when its induced action on the
reduced fiber $X_{b,\mathrm{red}}$ is trivial.
\end{remark}

The following theorem is the main result of this chapter. See \Cref{subs: merogrp}
for relevant definitions. The proof uses three big
results -- existence of the relative Douady space, properness of its
components, and Hironaka's flattening theorem -- but is otherwise elementary.

\begin{theorem} \label{thm:meromorphic}
	With notation and assumptions as above, there is a family of
	meromorphic groups $p \colon G \to B$, such that the action in \eqref{eq:action}
	factors through a meromorphic action $a \colon G \times_B X \to X$ over $B$. For
	each $b \in B$, the fiber $G_b = p^{-1}(b)$ is an $n$-dimensional meromorphic
	subgroup of $\Aut(X_b)$, whose neutral component is the quotient of the
	complex vector space $E_b$ by a discrete subgroup.
\end{theorem}

The general idea is to associate to an automorphism of $X_b$ its graph in the Douady
space of $X_b \times X_b$. In this way, the action in
\eqref{eq:action} determines a holomorphic mapping from $E$ to the relative Douady
space of $X \times_B X$ over $B$, and we construct $G$ as a carefully chosen
Zariski-open subset of the irreducible component containing the image.
In the case when $f \colon X \to B$ is a projective morphism between algebraic varieties,
the relative Douady space is the analytification of the relative Hilbert scheme; here the
theorem implies that the action in \eqref{eq:action} factors through an algebraic
action by a commutative group scheme.

\begin{corollary}
	With notation and assumptions as above, suppose that $X$ and $B$ are algebraic
	varieties, and that $f \colon X \to B$ is a projective morphism.  Then $p \colon G
	\to B$ is a commutative group scheme that acts algebraically on $X$ over $B$.
\end{corollary}

\subsection{Meromorphic mappings}

We recommend \cite[\S5]{fujiki-douady} for basic definitions. Let $X$ be a
complex space, Hausdorff and second-countable, but possibly nonreduced. A \define{meromorphic mapping} from $X$ to another
complex space $Y$ is a holomorphic mapping $f \colon U \to Y$ that is defined on a
dense Zariski-open subset $U \subseteq X$, and whose graph
\[
	\Gamma_f = \menge{(x,f(x)) \in U \times Y}{x \in U}
\]
has an analytic closure inside $X \times Y$. We usually abbreviate this by saying
that $f \colon X \mto Y$ is a meromorphic mapping. We consider two
meromorphic mappings to be equivalent if they agree on a dense Zariski-open
subset of $X$. If we have holomorphic mappings $X \to S$ and $Y \to S$, then $f
\colon X \mto Y$ is called a \define{meromorphic mapping over $S$} if the
diagram
\[
	\begin{tikzcd}
		U \rar{f} \dar & Y \dlar \\
		S
	\end{tikzcd}
\]
commutes; equivalently, if the closure of the graph satisfies $\overline{\Gamma_f}
\subseteq X \times_S Y$.

Let $X \to S$ be a proper holomorphic mapping, and denote by $D_{X/S}$ the relative
Douady space. Fujiki \cite[Lem.~5.1]{fujiki-douady} provides the following result for
constructing memomorphic mappings into $D_{X/S}$ (using Hironaka's flattening theorem).

\begin{lemma} \label{lem:Fujiki}
	Let $T$ be a complex space with a holomorphic mapping $T \to S$, and let $Z$ be a
	closed subspace of the fiber product $X_T = T \times_S X$.
	\[
		\begin{tikzcd}
			Z \rar[hook] \drar & X_T \rar \dar & X \dar \\
			& T \rar & S
		\end{tikzcd}
	\]
	Suppose that $Z$ is flat over a dense Zariski-open subset $U \subseteq T$, and denote
	by $h \colon U \to D_{X/S}$ the resulting holomorphic mapping into the relative
	Douady space. If $T$ is reduced, then $h \colon T \mto D_{X/S}$ is
	a meromorphic mapping.
\end{lemma}

\subsection{Meromorphic groups and meromorphic actions}\label{subs: merogrp}

Let's briefly recall a few relevant definitions and results from Fujiki's work on
automorphism groups of compact K\"ahler manifolds \cite{fujiki-auto}, related to
meromorphic groups and meromorphic actions. The main point is that one can embed the
automorphism group of a compact K\"ahler space $X$ into the Douady space
of $X \times X$, and that the components of the Douady space are compact. (The Douady space
is the analogue of the Hilbert scheme for compact complex spaces.) See
\cite[\S1]{fujiki-auto} for the definition of (possibly singular and nonreduced)
compact K\"ahler spaces and of (reduced) compact complex spaces in Fujiki's class $\fuj$.

Let $G$ be a complex Lie group. In Fujiki's terminology \cite[Def.~2.1]{fujiki-auto},
a \define{meromorphic
structure} on $G$ is a compactification $G^{\ast}$, with $G$ Zariski-open and dense in
$G^{\ast}$, such that the group operation $G \times G \to G$ extends to a meromorphic
mapping $G^{\ast} \times G^{\ast} \mto G^{\ast}$ that is holomorphic on the dense
Zariski-open subset $G \times G^{\ast} \cup G^{\ast} \times G$. (In particular, the
compactification $G^{\ast}$ comes with a $G$-action.)
If $G$ acts holomorphically on a compact complex space $X$, then the action is said to be
\emph{meromorphic} if the holomorphic mapping $G \times X \to X$ extends to a
meromorphic mapping $G^{\ast} \times X \mto X$. We usually talk about ``meromorphic
groups'' and ``meromorphic actions'', with the understanding that this refers to a fixed
(equivalence class of) compactification $G^{\ast}$.

\begin{example}
In the case where $G = \Aut^\circ(X)$ is the neutral component of the automorphism
group of a compact K\"ahler space $X$, the \define{natural meromorphic structure} $G^{\ast}$
is obtained by taking the closure of $G$ inside the Douady space of $X \times X$.
Fujiki proves that this is a compact analytic space
of class $\fuj$ (= the image of a compact K\"ahler manifold under a proper
holomorphic mapping), and that this particular compactification gives $G$ the
structure of a meromorphic group that acts meromorphically on $X$, in the above sense.
More precisely, Fujiki \cite[Lem~4.6]{fujiki-auto} only states this for \emph{reduced}
compact complex
spaces of class $\fuj$, but the result actually holds as long as the reduction $X_{\red}$ is of class $\fuj$:
by \cite[Thm.~5.3]{fujiki-douady}, this condition is enough to make the reduction of
each irreducible component of $D_{X \times X}$ into a compact complex space of class $\fuj$,
and so \cite[Thm.~1.2]{fujiki-auto} and hence  \cite[Prop~2.2]{fujiki-auto} hold in that generality.
\end{example}

Next, let's discuss the analogue of the Chevalley structure theorem from the theory
of algebraic groups. Let $G$ be a connected meromorphic group.
By \cite[Prop.~3.1]{fujiki-auto}, there is always a universal morphism $q \colon G \to T$
to a compact complex torus, such that $\ker q$ is connected; concretely, $T$ is obtained
as the Albanese torus of a resolution of singularities of the compactification $G^{\ast}$.
For simplicity, we call $T$ the \define{maximal compact quotient} of $G$.
Fujiki \cite[Def.~3.1]{fujiki-auto} then says that $G$ is \define{of regular type} if there
is an exact sequence
\begin{equation} \label{eq:regular-type}
	\begin{tikzcd}[column sep=small]
		1 \rar & L \rar & G \rar{q} & T \rar & 1
	\end{tikzcd}
\end{equation}
in the category of meromorphic groups,with $L = \ker q$ isomorphic, as a meromorphic group,
to a linear algebraic group. If $G$ is of regular type, then both $L$ and $T$ are
functorial, in the category of meromorphic groups; the main point is that a meromorphic mapping
from a  compactification of a linear algebraic group to a compact complex torus is constant
\cite[Lem.~3.8]{fujiki-auto}. Fujiki proves in \cite[Cor.~5.7]{fujiki-auto} that the
following three conditions on a connected meromorphic group $G$ are equivalent:
\begin{enumerate}
\item $G$ is of regular type.
\item $G$ is \define{of type $\fuj$}, meaning that one can choose $G^{\ast}$ in Fujiki's class $\fuj$.
\item $G$ is \define{of K\"ahler type}, meaning that one can choose $G^{\ast}$ to be
a compact K\"ahler manifold.
\end{enumerate}
According to later work by Pillay and Scanlon \cite{Pillay+Scanlon}, every meromorphic
group is of regular type. The following weaker result by Fujiki is sufficient for our
purposes.

\begin{proposition}(Fujiki) \label{prop:regular-type}
Let $G$ be a connected meromorphic group. If $G$ is commutative, or if $G$ is
a meromorphic subgroup of $\Aut^{\circ}(X)$ for $X$ a compact complex space such
that $X_{\red}$  is of class $\fuj$, then $G$ is of regular type.
\end{proposition}

\begin{proof}
When $G$ is commutative, this is \cite[Prop.~3.9]{fujiki-auto}. Fujiki \cite[Thm.~5.5]{fujiki-auto}
states the second result only for reduced $X$, so we give some details.
Suppose that $X$ is a compact complex space whose reduction $X_{\red}$ is of class $\fuj$.
As we explained above, $\Aut^{\circ}(X)$ has a natural structure of meromorphic group,
which is of class $\fuj$ in the sense of \cite[Def.~3.1]{fujiki-auto}. By \cite[Cor.~5.7]{fujiki-auto},
this implies that $\Aut^{\circ}(X)$ is of regular type; the point is that $\Aut^{\circ}(X)$
acts not just on $X$, but also on the compactification giving the meromorphic structure,
which is reduced. As a meromorphic subgroup, $G$ is then also of regular type by \cite[Rmk.~4.2]{fujiki-auto}.
\end{proof}

As we are interested in families of compact K\"ahler spaces, we should  define
carefully what we mean by a holomorphic family of complex Lie groups. In the case
when $p \colon G \to B$ is algebraic, this is the definition of a (flat) group scheme.

\begin{definition} \label{def:family-groups}
	A \define{holomorphic family of complex Lie groups} is a flat holomorphic mapping $p
	\colon G \to B$ between two complex manifolds, together with a holomorphic section
	$e \colon B \to G$ and two holomorphic mappings
	\[
		m \colon G \times_B G \to G \quad \text{and} \quad i \colon G \to G
	\]
	over $B$, subject to the following conditions:
	\begin{enumerate}
		\item $m$ is fiberwise associative, meaning that this diagram commutes:
			\[
				\begin{tikzcd}
					G \times_B G \times_B G \rar{m \times \id} \dar{\id \times m}
						& G \times_B G \dar{m} \\
					G \times_B G \rar{m} & G
				\end{tikzcd}
			\]
		\item $e$ is the fiberwise unit, meaning that this diagram commutes:
			\[
				\begin{tikzcd}[row sep=small]
					G \drar[swap,sloped]{(\id, e)} \arrow[bend left=25]{drr}{\id} \\
					& G \times_B G \rar{m} & G \\
					G \urar[sloped]{(e, \id)} \arrow[bend right=25]{urr}{\id}
				\end{tikzcd}
			\]
		\item $i$ is the fiberwise inverse, meaning that this diagram commutes:
			\[
				\begin{tikzcd}
					& G \times_B G \drar[bend left=20]{m} \\
					G \urar[bend left=20]{(\id, i)} \drar[bend right=20,swap]{(i, \id)}
							\rar & B \rar{e} & G \\
				  & G \times_B G \urar[bend right=20,swap]{m}
				\end{tikzcd}
			\]
	\end{enumerate}
	Each fiber $G_b = p^{-1}(b)$ is a complex Lie group, and therefore reduced
	and nonsingular; because of flatness, this makes the holomorphic mapping $p \colon
	G \to B$ smooth.
\end{definition}

Here are relative versions of Fujiki's definitions.

\begin{definition}
	Let $p \colon G \to B$ be a holomorphic family of complex Lie groups. It is called
	a \define{family of meromorphic groups} if $p \colon G \to B$ admits an extension
	to a proper morphism $D \to B$ with $D$ reduced and $G$ Zariski-open and
	dense in $D$, in such a way that the
	group operation $G \times_B G \to G$ extends to a holomorphic mapping $G
	\times_B D \cup D \times_B G \to D$ whose graph has an analytic closure inside $D
	\times_B D \times_B D$.
\end{definition}

The somewhat complicated definition is due to the fact that the fiber product
$D \times_B D$ may have several irreducible components. Denote by
$(D \times_B D)_{\main}$ the main component, meaning the closure of $G
\times_B G$, with the reduced complex structure. Then we are asking that
\[
	(D \times_B D)_{\main} \mto D
\]
is a meromorphic mapping that is holomorphic on the dense Zariski-open subset $G
\times_B D \cup D \times_B G$ and extends the group operation $m \colon G \times_B G
\to G$. Note that the proper morphism $D \to B$ is part of the data, and when we speak of a
``family of meromorphic groups'', it will always be clear which $D$ we have in mind.
The definition implies that the inverse $i \colon D \mto D$ is  a
meromorphic mapping over $B$.

\begin{definition}
	Let $p \colon G \to B$ be a family of meromorphic groups that act holomorphically
	on the fibers of a proper holomorphic mapping $f \colon X \to B$. We say that the
	action $a \colon G \times_B X \to X$ is \define{meromorphic} if it extends to a
	meromorphic mapping $D \times_B X \mto X$.
\end{definition}

Note that we are not requiring $D \to B$ to be equidimensional, and so each
fiber $D_b$ is not necessarily a compactification of the complex Lie group
$G_b = p^{-1}(b)$. Nevertheless, our definitions do imply that each $G_b$ is a
meromorphic group that acts meromorphically on $X_b = f^{-1}(b)$. The reason is
that $G_b$ is a Zariski-open subset of the compact complex space $D_b$, and so
$G_b$ is contained in finitely many irreducible components of $D_b$. Their
union must then be the closure of $G_b$, and by restricting the meromorphic
mapping $(D \times_B D)_{\mathrm{main}} \mto D$ to this closure (which makes
sense because it is holomorphic on $G \times_B D \cup D \times_B G$), we obtain
on $G_b$ the structure of a meromorphic group. Similarly, by restricting the
meromorphic mapping $D \times_B X \mto X$, we see that $G_b$ acts
meromorphically on $X_b$. Thus, for every $b \in B$, the fiber $G_b$ is a
meromorphic group and acts meromorphically on $X_b$. This justifies the
terminology ``family of meromorphic groups'' and ``acts meromorphically.''

\subsection{The compactification}

Let's start working on the proof of \Cref{thm:meromorphic}.
Let $D_{X \times_B X/B}$ denote the
relative Douady space for the morphism $X \times_B X \to B$. This is a complex space
with countably many irreducible components; since $f \colon X \to B$ is proper and
flat, and $X$ is K\"ahler, the reduction of each irreducible component of
the Douady space is proper over $B$ \cite[Thm~5.2]{fujiki-douady}, and the fibers of
the projection to $B$ are in Fujiki's class $\fuj$ \cite{fujiki-catC-II}.
The (relative) Douady space has the same universal property as the (relative) Hilbert
scheme in algebraic geometry. In particular, there is a morphism $D_{X \times_B X/B}
\to B$, whose fiber over any point $b \in B$ is the Douady space of $X_b \times
X_b$, and there is a universal family
\[
	Z_{X \times_B X/B} \subseteq D_{X \times_B X/B} \times_B X \times_B X
\]
that is proper and flat over $D_{X \times_B X/B}$. When $f \colon X \to B$ is a projective
morphism between algebraic varieties,  $D_{X \times_B X/B}$ is naturally isomorphic to
the analytification of the relative Hilbert scheme.

The action in \eqref{eq:action} corresponds to a holomorphic mapping from the vector
bundle $E$ into the relative Douady space, in the following manner. The graph
\begin{tcd}
		E \times_B X \dar \rar[hook] & E \times_B X \times_B X \dar \\
		E \rar{p} & B
\end{tcd}
of the action embeds $E \times_B X$ into $E \times_B X \times_B X$ as a closed complex
subspace. It is proper and flat over $E$ (because $f \colon X \to B$ is proper and
flat); by the universal property of the relative Douady space, it therefore
determines a holomorphic mapping
\[
	\eps \colon E \to D_{X \times_B X/B}
\]
over $B$. Let $D$ denote the unique irreducible component containing the image
of $\eps$, taken with its \emph{reduced} structure; because $E$ is a complex manifold,
$\eps$ factors through the reduction of the component containing the image.
We also write $\eps \colon E \to D$ for the induced mapping, and we let
$\pi \colon D \to B$ be the projection to $B$.

\begin{lemma}
	The mapping $\pi \colon D \to B$ is proper, and its fibers are compact complex
	spaces in $\fuj$. We have $\dim D = \dim B + n = \dim E$, and $D$ is the closure of
	$\eps(E)$.
\end{lemma}

\begin{proof}
	Since $X$ is K\"ahler and $f \colon X \to B$ is proper and flat, the first two
	assertions are consequences of Fujiki's results, already quoted above. Each smooth
	fiber $X_b$ is a compact complex torus of dimension $n$, and because $E_b \to \Aut^{\circ}(X_b)$
	is surjective in that case (by our assumptions), the general fiber $D_b
	= \pi^{-1}(b)$ is just the embedding of $\Aut^{\circ}(X_b)$ into the Douady space of $X_b
	\times X_b$, and so it has dimension $n$. We conclude that $\dim D = \dim B + n =
	\dim E$. For dimension reasons, $D$ is therefore equal to the closure of the image
	of $\eps$.
\end{proof}

Note that $D_b$ is always a compact subspace of the Douady space of $X_b \times X_b$,
but it could in principle be of dimension greater than $n$, have multiple components,
etc. That said, we are going to prove quite soon that $D$ is actually a
complex manifold along the image of $\eps \colon E \to D$.

We can now construct $p \colon G \to B$ as a Zariski-open subset of $D$. Let $Z \to
D$ denote the pullback of the universal family, as in the following diagram:
\begin{tcd}
	Z \rar \dar[hook] & Z_{X \times_B X/B} \dar[hook] \\
	D \times_B X \times_B X \rar \dar & D_{X \times_B X/B} \times_B X \times_B X \dar \\
	D \rar & D_{X \times_B X/B}
\end{tcd}
Over each point $d \in D_b$, the fiber $Z_d$ is a closed subspace of $X_b
\times X_b$; when $d$ is the image of a point in $E_b$, this subspace
is exactly the graph of the automorphism coming from the action in \eqref{eq:action}.
We define $G \subseteq D$ as
\begin{equation} \label{eq:P-definition}
	G = \MENGE{d \in D}{\parbox{6.3cm}{$Z_d$ is the graph of an automorphism of
	$X_b$\\and $\pi \colon D \to B$ is smooth at the point $d$}};
\end{equation}
this makes sure that the fibers $G_b$ are complex manifolds of dimension $n$ and
consist entirely of automorphisms of $X_b$ (provided $G_b$ is nonempty). Of course,
$G$ contains $\Aut^{\circ}(X_b) = \Aut^\circ(X_b)$ for all the smooth fibers, but some work is
necessary to show that the fibers over other points $b \in B$ are also nonempty. In
particular, because of the smoothness condition in \eqref{eq:P-definition}, it is not
clear a priori that the image of $\eps \colon E \to D$ actually lands in $G$.

\begin{lemma} \label{lem:action}
	$G$ is a Zariski-open subset of $D$, and the restriction of the universal family
	$Z \to D$ to this subset is the graph of a holomorphic mapping
	\begin{equation} \label{eq:action-P}
		a \colon G \times_B X \to X.
	\end{equation}
	For $g \in G_b$ and $x \in X_b$, we usually write $g \cdot x$ in place of $a(g,x)$.
\end{lemma}

\begin{proof}
	There are two conditions in the definition of $G$. The condition on the second
	line is Zariski-open, because the smooth locus of a holomorphic mapping between
	complex spaces is Zariski-open. To deal with the
	condition on the first line, we use a similar
	argument as in \cite[Lem.~5.5]{fujiki-douady}. The general result we need is that
	if $f \colon X \to Y$ is a proper morphism, then the maximal open subset of $Y$
	over which $f$ is an isomorphism is actually Zariski-open. The reason is that its
	complement is the union
	\[
		\menge{y \in Y}{\dim f^{-1}(y) \geq 1} \cup \Supp \ker(\phi) \cup \Supp
		\coker(\phi)
	\]
	of three closed analytic sets; here $\phi \colon \OY \to \fl \OX$ is the comorphism.
	Consider the two projections
	\[
		p_{12} \colon Z \to D \times_B X \quad \text{and} \quad
		p_{13} \colon Z \to D \times_B X.
	\]
	The set of points in $D \times_B X$ over which either $p_{12}$ or $p_{13}$ fails to
	be an isomorphism is closed analytic; its image in $D$ is therefore closed
	analytic, because $X$ is proper over $B$. The complement of this set consists
	exactly of those points $d \in D$ where the projection of $Z_d \subseteq X_{\pi(d)} \times
	X_{\pi(d)}$ to both factors is an isomorphism, meaning where $Z_d$ is the graph of an
	automorphism of $X_{\pi(d)}$. This proves that $G$ is Zariski-open inside $D$.

	By construction, the restriction of the universal family $Z$ to the open subset $G$
is isomorphic to $G \times_B X$, as in the following diagram:
\begin{tcd}
	G \times_B X \rar \dar[hook] & Z \dar[hook] \\
	G \times_B X \times_B X \rar \dar & D \times_B X \times_B X \dar \\
	G \rar & D.
\end{tcd}
The vertical arrow on the top left is therefore the graph of a morphism
\[
	a \colon G \times_B X \to X,
\]
which can also be viewed as the composition of the projection $p_3 \colon G
\times_B X \times_B X \to X$ with the inverse of the isomorphism $p_{12} \colon G
\times_D Z \to G \times_B X$.
\end{proof}

\subsection{The family of complex Lie groups}

Let $p \colon G \to B$ denote the restriction of the proper mapping $\pi \colon D
\to B$ to the Zariski-open subset $G$. The next step is to show that this is a
holomorphic family of complex Lie groups.

\begin{theorem} \label{prop:family-groups}
	The mapping $p \colon G \to B$ is a holomorphic family of complex Lie groups.
	Moreover, the image $\eps(E)$ of the morphism $\eps \colon E \to D$ is an open
	subset of $G$, and the induced mapping $\eps \colon
	E \to G$ is a fiberwise group homomorphism.
\end{theorem}

We begin by constructing the group operations, because the same reasoning will be
needed again later on. Consider the following commutative diagram:
\begin{tcd}
	& B \dar{e} \dlar[bend right=25] \arrow[bend left=50]{dd}{\id} \\
	E \rar{\eps} \drar[bend right=25]{p} & D \dar{\pi} \\
													 & B
\end{tcd}
The image of the zero section in $E$ gives us a section $e \colon B \to D$ that
embeds $B$ into $D$ as a complex subspace; alternatively, one can get this from
the diagonal embedding $X \into X \times_B X$, using the universal property of $D_{X
\times_B X/B}$. For each $b \in B$, the point $e(b) \in D_b$ corresponds to the graph
of the identity automorphism of $X_b$. This is eventually going to be the identity
section of our family of complex Lie groups -- but because of the smoothness condition in
\eqref{eq:P-definition}, it will take us some time to prove that the image of the
section actually lies in $G$.

\begin{remark}
The general idea is to use the universal
family $Z \subseteq D \times_B X \times_B X$. The following heuristic may be helpful.
If we have a triple $(g,x,y) \in G \times_B X \times_B X$, then
$(g,x,y) \in Z$ if and only if $y = g \cdot x$ (or $y = a(g,x)$ in the notation of
\Cref{lem:action});
this makes sense because $Z_g$ is the graph of an automorphism. So we should think of
$(d,x,y) \in Z$ as saying that $y = d \cdot x$, even though $Z_d$ is of course not
the graph of an automorphism for arbitrary $d \in D$.
\end{remark}

We can easily construct a holomorphic involution $i \colon G \to G$ that sends each
automorphism to its inverse. Let
\[
	\sigma \colon D \times_B X \times_B X \to D \times_B X \times_B X
\]
be the involution defined by $\sigma(d,x,y) = (d,y,x)$. Then $\sigma(Z)$ is still
proper and flat over $D$, and so it corresponds to a morphism $i \colon D \to D_{X
\times_B X/B}$. The image of $i$ must be contained in $D$, because the image of $\eps
\colon E \to D$ obviously goes to itself under $i$, and $D$ is the unique reduced and
irreducible complex subspace containing $\eps(E)$. At a point $g \in G$, the subset
$Z_g$ is the graph of an automorphism, and $\sigma(\{g\} \times Z_g)$ is the graph of
the inverse automorphism; also, the involution $i$ preserves the locus where $\pi
\colon D \to B$ is smooth. Therefore $i(G) \subseteq G$.

\begin{lemma} \label{lem:automorphism}
	The holomorphic mapping
	\[
		G \times_B X \to G \times_B X, \quad (g,x) \mapsto \bigl( g, a(g,x) \bigr),
	\]
	defined using $a \colon G \times_B X \to X$ from \eqref{eq:action-P}, is
	biholomorphic.
\end{lemma}

\begin{proof}
	The involution $i \colon G \to G$ takes a point $g \in G_b$, which represents
	an automorphism of $X_b$, to the inverse automorphism. Moreover, for $x \in X_b$,
	the point $a(g,x) \in X_b$ is the result of applying the automorphism to $x$.
	Therefore the inverse of our holomorphic mapping is obtained by conjugating it
	by the involution
	\[
		G \times_B X \to G \times_B X, \quad (g,x) \mapsto \bigl( i(g), x \bigr).
	\]
	So both the mapping and its inverse are holomorphic.
\end{proof}

Finally, we construct the group operation $m \colon G \times_B G \to G$, which is of
course given by composition of automorphisms. Consider the automorphism
\[
	\phi \colon G \times_B D \times_B X \times_B X \to G \times_B D \times_B X
	\times_B X,
\]
given by the rule $\phi(g,d,x,y) = (g,d,x,g \cdot y)$; this is an automorphism by
\Cref{lem:automorphism}. The image of $G \times_B Z$ under $\phi$ is again
proper and flat over $G \times_B D$, and so it corresponds to a morphism
\begin{equation} \label{eq:P-action-D}
	G \times_B D \to D_{X \times_B X/B}.
\end{equation}
Now $p \colon G \to B$ is smooth and the generic fiber $G_b$ is irreducible, and so
$G \times_B D$ is smooth over $D$, irreducible and reduced; because we know what happens over the Zariski-open
subset where $X_b$ is a compact complex torus, we can conclude that
\eqref{eq:P-action-D} maps into $D$. Restricting to the Zariski-open subset $G$, we get the
desired morphism
\[
	m \colon G \times_B G \to G.
\]
To check that the image really lands in $G$, we need to verify both conditions in
\eqref{eq:P-definition}. The one on the first line holds because the composition of
two automorphisms is an automorphism. The one on the second line holds because, $p
\colon G \to B$ being smooth, it admits local holomorphic sections, which act as
automorphisms on $D$ and therefore preserve the smooth locus of $\pi \colon D \to B$.

\begin{remark}
	The conditions in \Cref{def:family-groups} are clearly satisfied over
	the smooth locus of $f \colon X \to B$ (where $G_b \cong \Aut^{\circ}(X_b)$ is a compact
	complex torus), and by
	the identity theorem, most of them therefore hold over all of $B$. The only
	nontrivial matter is to show that the image of the section $e \colon B \to D$ is
	contained in $G$. This is a bit more subtle than one might expect, and uses in a crucial
	way the existence of the action in \eqref{eq:action}. We shall deal with this
	issue in the next section.
\end{remark}

\subsection{Smoothness along the section}

We are now going to prove that $D$ is smooth over $B$ along the image of the
section $e \colon B \to D$, and hence that the image of the section is contained in
the Zariski-open subset $G$ defined in \eqref{eq:P-definition}. A small complication
is caused by the fact that $D$, which is globally irreducible, might be locally
reducible. This forces us to start with a weaker result.

\begin{lemma} \label{lem:local-component}
	There is an open neighborhood of the zero section in $E$, whose image under
	$\eps \colon E \to D$ is a complex manifold of dimension $\dim D$; the image is a local
	irreducible component of $D$ at every point of the section $e \colon B \to D$.
\end{lemma}

\begin{proof}
Let $b \in B$ be any point. As we have seen, the condition for $Z_d$ to be the graph
of an automorphism of $X_b$ is Zariski-open on $d \in D$. In a neighborhood of the
point $e(b)$, the scheme-theoretic fiber $D_b = \pi^{-1}(b)$ is therefore a (possibly
nonreduced) closed subspace of $\Aut^{\circ}(X_b)$, with the point $e(b)$
corresponding to the identity in $\Aut^{\circ}(X_b)$. This gives us an injection
\[
	T_{e(b)} D_b \into \Lie \Aut^{\circ}(X_b)
\]
on the level of Zariski tangent spaces. Because the $\CC$-vector space $E_b$ injects
into $\Lie \Aut^{\circ}(X_b)$ by assumption, both morphisms
\[
	E_b \into T_{e(b)} D_b \into \Lie \Aut^{\circ}(X_b)
\]
must be injective. After dualizing, it follows that
\[
	T_{e(b)}^{\ast} D_b \twoheadrightarrow E^{\ast}_b
\]
is surjective. We can turn this into a statement about the conormal sheaf of the
section $e$. Let $\shI \subseteq \shO_D$ denote the ideal sheaf of the section $e$, and $N^{\ast}
= \shI/\shI^2$ be the conormal sheaf. Because $\pi \colon D \to B$ has a section,
it is easy to see that
\[
	\iu_b N^{\ast} \twoheadrightarrow  T_{e(b)}^{\ast} D_b
\]
is surjective, where $i_b \colon \pt \into B$ is the inclusion of the point $b \in
B$.
Putting everything together, we find that the morphism of $\shO_B$-modules
\begin{equation} \label{eq:conormal}
	N^{\ast} \to E^{\ast}
\end{equation}
is surjective. Note that $E^{\ast}$ is locally free of rank $n$; but a priori, we
don't know anything about the conormal sheaf $N^{\ast}$.

We now use \eqref{eq:conormal} to prove the assertions in the lemma.
As $E$ is a vector bundle of rank $n$, the surjection in
\eqref{eq:conormal} is locally split; in a neighborhood of any point $b \in B$, we
can therefore find $n$ holomorphic functions $h_1, \dotsc, h_n \in \shI$ whose images
under \eqref{eq:conormal} give a local trivialization for $E^{\ast}$.  Together with
$\pi \colon D \to B$, these functions locally define a holomorphic mapping $D \to B
\times \CC^n$, and one checks that the composition $E \to D \to B \times \CC^n$ acts
as the identity on the tangent space at the origin in $E_b \subseteq E$, and is
therefore biholomorphic in a neighborhood. In particular, the image of $\eps$ is, at
least in a small neighborhood of the section $e$, a complex manifold of dimension
$\dim D$ and a local irreducible component of $D$.
\end{proof}

Next, we show that $D$ is in fact locally irreducible. The argument needs
flatness, which is easier to check over curves, and so we reduce the problem to a
carefully chosen curve in the base (with the details postponed to the following
section).

\begin{lemma} \label{lem:local-irreducibility}
	$D$ is locally irreducible at every point on the section $e \colon B \to D$.
\end{lemma}

\begin{proof}
Suppose that, for some $b \in B$, the complex space $D$ has two local irreducible
components $D'$ and $D''$ that both pass through the point $e(b)$. By
\Cref{lem:local-component}, we may assume that $D'$ is the image of an open
neighborhood of a point on the zero section of the bundle $E$. Let $\Delta \subseteq
\CC$ denote the open unit disk. We can choose a holomorphic curve
$\tilde{c} \colon \Delta \to D''$ such that $\tilde{c}(0) = e(b)$ and such that the
image of $\Delta \setminus \{0\}$ lies over the smooth locus $B \setminus \disc(f)$
and does not intersect $D'$. The reason is that $D$ is the closure of its
restriction to the smooth locus of $f$, and so the same is true for its local
irreducible components; therefore both $\pi^{-1} \bigl( \disc(f) \bigr)$
and $D'$ can only intersect $D''$ in a closed analytic subset of smaller dimension.
Define $c = \pi \circ \tilde{c} \colon \Delta \to B$, which is a holomorphic
curve on the complex manifold $B$, as in the following commutative diagram:
\begin{tcd}
	\Delta \rar{\tilde{c}} \drar[bend right=25]{c} & D' \dar{\pi} \\
						& B
\end{tcd}
The base change $X_{\Delta} = \Delta \times_B X \to \Delta$ is flat of relative
dimension $n$, and the fibers over $\Delta \setminus \{0\}$ are compact complex tori.
The vector bundle $E_{\Delta} = \Delta \times_B E \to \Delta$ acts fiberwise on
$X_{\Delta}$, and therefore determines as before a reduced and irreducible subspace
$D_{\Delta}$ of the relative Douady space, with $\dim D_{\Delta} = n+1$; note that
$D_{\Delta}$ is \emph{not} the base change of $\pi \colon D \to B$ (which may have
multiple irreducible components).
By construction, $D_{\Delta}$ contains both the image of $E_{\Delta}$ and the image
of the curve $\tilde{c}$, and so $D_{\Delta}$ still has at least two local
irreducible components $D_{\Delta}'$ and $D_{\Delta}''$. But this now contradicts
\Cref{lem:curves}, which asserts that $D_{\Delta}$ is nonsingular in a
neighborhood of the section.
\end{proof}

\subsection{Restriction to curves}

We need an auxiliary result that describes what happens under restriction to
holomorphic curves in $B$. We denote by $\Delta \subseteq \CC$ the open unit disk,
and by $\disc(f) \subseteq B$ the closed analytic subset over which the fibers of $f
\colon X \to B$ are singular. Consider a holomorphic curve
\[
	c \colon \Delta \to B
\]
with the property that $c^{-1} \bigl( \disc(f) \bigr) = \{0\}$.
Denote by $f_{\Delta} \colon X_{\Delta} \to \Delta$ the base change of the morphism
$f \colon X \to B$, and by $p_{\Delta} \colon E_{\Delta} \to \Delta$ the base change
of the vector bundle $p \colon E \to B$. Then $f_{\Delta}$ is flat of relative
dimension $n$, and its fibers over $\Delta \setminus \{0\}$ are compact complex tori.
Because of \eqref{eq:action}, the vector bundle $E_{\Delta}$ acts fiberwise on
$X_{\Delta} \to \Delta$. As before, this action determines a reduced and irreducible
subspace $D_{\Delta}$ inside the relative Douady space of $X_{\Delta} \times_{\Delta}
X_{\Delta}$ over $\Delta$; of course, $D_{\Delta}$ is usually \emph{not} the base
change of $\pi \colon D \to B$. Since we are now working over a curve, $D_{\Delta}$
is flat of relative dimension $n$ over $\Delta$. We still have a section $e_{\Delta}
\colon \Delta \to D_{\Delta}$.

\begin{lemma} \label{lem:curves}
	With assumptions and notation as above, the morphism $D_{\Delta} \to \Delta$ is
	smooth in a neighborhood of the point $e_{\Delta}(0)$.
\end{lemma}

\begin{proof}
	Consider the Zariski-open subset
	\[
		G_{\Delta} = \MENGE{d \in D_{\Delta}}%
		{\text{$Z_d$ is the graph of an automorphism of $X_b$}},
	\]
	keeping in mind that this is again \emph{not} the base change of $p \colon G \to
	B$. It clearly contains the image of the section $e_{\Delta} \colon \Delta \to
	D_{\Delta}$. We claim that $G_{\Delta} \to \Delta$ is a family of complex Lie
	groups. Indeed, by exactly the same construction as in the previous section, we
	get a holomorphic mapping
	\[
		G_{\Delta} \times_{\Delta} D_{\Delta} \to D_{\Delta},
	\]
	the point being that the fiber product on the left-hand side is reduced and
	irreducible because this holds over $\Delta \setminus \{0\}$ and the morphism
	$G_{\Delta} \to \Delta$ is flat. It is compatible with composition of
	automorphisms, and therefore restricts to a holomorphic mapping
	\[
		G_{\Delta} \times_{\Delta} G_{\Delta} \to G_{\Delta},
	\]
	which makes $G_{\Delta} \to \Delta$ into a family of complex Lie groups. It
	follows that $G_{\Delta} \to \Delta$ is a smooth morphism, and because
	$G_{\Delta}$ contains an open neighborhood of the section $e_{\Delta} \colon
	\Delta \to D_{\Delta}$, we get the desired result.
\end{proof}

\subsection{Proof of Theorem~\ref*{thm:intro-A}}

We now have everything in place to prove that $p \colon G \to B$ is a family of
complex Lie groups; that $\eps(E)$ is an open subset of $G$; and that $\eps \colon E
\to G$ is a fiberwise group homomorphism.

\begin{proof}
According to \Cref{lem:local-irreducibility}, $D$ is locally irreducible along
the section $e \colon B \to D$; consequently, \Cref{lem:local-component} tells
us that $D$ is nonsingular along the image of $e$, and that $\eps \colon E \to
D$ induces an isomorphism between an open neighborhood of the zero section in $E$ and
an open neighborhood of $e(B)$ in $D$. In particular, $\pi \colon D \to
B$ is smooth at every point of the section $e$, and in view of
\eqref{eq:P-definition}, it follows that $e(B) \subseteq G$. We may therefore
view the section as a holomorphic mapping $e \colon B \to G$.

This is enough to conclude that $p \colon G \to B$ is a holomorphic family of complex
Lie groups. Indeed, all the conditions in \Cref{def:family-groups} hold
over the Zariski-open subset $B \setminus \disc(f)$ where $X_b$ and $G_b \cong
\Aut^{\circ}(X_b)$ are compact complex tori. Because $p \colon G \to B$ is smooth, all the fiber
products that appear are reduced and irreducible, and so all the relevant diagrams
commute by the identity theorem.

We also know from \Cref{lem:local-component} that $\eps$ maps an open
neighborhood of the zero section of $E$ into $G$. Because each fiber $E_b$ is
generated, as a group, by any open neighborhood of the origin, it follows that
$\eps(E_b) \subseteq G_b$, and hence $\eps(E) \subseteq G$. Because the group
operation on $G$ is composition of automorphisms, it is then clear that
\[
	\eps \colon E \to G
\]
is a fiberwise group homomorphism. We may then conclude, by the same argument is in
the proof of \Cref{lem:local-component}, that $\eps(E)$ is an open subset of $G$ and that
$\eps$ is locally biholomorphic.
\end{proof}

\begin{remark} \label{rmk:Lambda}
We can even describe the kernel of the morphism
\[
	\eps \colon E \to G;
\]
the argument is inspired by the proof of \cite[Prop.~5.7]{sacca24}.
Consider again the section $e \colon B \to G$. Its preimage under $\eps$ is a closed
analytic subset of the vector bundle $E$ that we shall denote by the letter
\[
	\Lambda = \eps^{-1} \bigl( e(B) \bigr) \subseteq E.
\]
Over any point where $X_b$ is a compact complex torus, we have $G_b \cong \Aut^{\circ}(X_b)$,
and so the kernel of $E_b \to \Aut^{\circ}(X_b)$ is a lattice $\Lambda_b$ of rank $2n$ in the
$n$-dimensional complex vector space $E_b$. The restriction of $\Lambda$ to the
smooth locus of $f$ is therefore just the family of these lattices. Now what happens
is that the whole set $\Lambda$ is the closure of this Zariski-open subset.
Indeed, $e(B)$ is locally defined by $n$ holomorphic functions inside $G$,
due to $G$ being smooth along the image of the section; therefore every
irreducible component of its preimage $\Lambda$ has codimension $\leq n$.
Because $\Lambda$ intersects every fiber $E_b$ in a discrete subset (due to our
assumption that $E_b \to \Lie \Aut^{\circ}(X_b)$ is injective), it follows that every
irreducible component has codimension exactly $n$ and dominates $B$. This means that
$\Lambda$ is the closure of its restriction to the smooth locus of $f$.
\end{remark}

Each fiber $G_b$ is a complex Lie group of dimension $n$. By construction, it is a
subgroup of the automorphism group $\Aut(X_b)$, through its embedding into the Douady
space of $X_b \times X_b$; the group operation $G_b \times G_b \to G_b$ is just
composition of automorphisms. When $X_b$ is smooth, $G_b \cong \Aut^{\circ}(X_b)$ is a compact
complex torus; but it is possible that $G_b$ has multiple connected components over
the singular fibers $X_b$. The number of connected components is of course finite,
because $G_b$ is a Zariski-open subset of the compact complex space $D_b$.
We also have a group homomorphism $\eps \colon E_b \to G_b$ with discrete kernel
$\Lambda_b$; for dimension reasons, it follows that the neutral component of $G_b$ is
isomorphic to $E_b/\Lambda_b$, and therefore commutative. In fact, this is true for
the entire Lie group $G_b$.

\begin{lemma}
Each complex Lie group $G_b$ is commutative.
\end{lemma}

\begin{proof}
Consider the holomorphic mapping
\[
	G \times_B G \to G, \quad (g,h) \mapsto g h g^{-1} h^{-1}.
\]
By smoothness of $p \colon G \to B$, the fiber product is smooth over $B$, and
therefore irreducible. The set of points $g \in G$ such that $gh g^{-1} h^{-1} = \id$
is closed analytic (as the preimage of a section); because it contains $G_b \times
G_b$ whenever $X_b$ is a compact complex torus, and because $G \times_B G$ is
irreducible, it must be all of $G \times_B G$.
\end{proof}

\subsection{The family of meromorphic groups}

We also need to prove that $p \colon G \to B$ is a family of meromorphic groups; the
relative compactification is of course the proper holomorphic mapping $\pi \colon D
\to B$ that we constructed above.

\begin{proposition}
	The compactification $\pi \colon D \to B$ turns the family of complex Lie groups
	$p \colon G \to B$ into a family of meromorphic groups.
\end{proposition}

\begin{proof}
The proof is very similar to \cite[Prop.~2.2]{fujiki-auto}. In
\eqref{eq:P-action-D}, we had already constructed a holomorphic mapping
\[
	G \times_B D \to D
\]
whose restriction to $G \times_B G$ is the group operation on $G$. By symmetry, it is
sufficient to show that this extends to a meromorphic mapping $(D \times_B D)_{\main}
\mto D$.  To do that, we go back to the automorphism
\[
	\phi \colon G \times_B D \times_B X \times_B X \into G \times_B D \times_B X
	\times_B X, \quad \phi(g, d, x, y) = (g, d, x, g \cdot y).
\]
Let $W = \phi(G \times_B Z)$, which is proper and flat over $G \times_B D$.
We will first argue that the closure $\bar{W}$ inside $D \times_B D \times_B X
\times_B X$ remains analytic. With that goal in mind, consider the five-fold fiber product
\[
	D \times_B D \times_B X \times_B X \times_B X
\]
and denote a typical point by $(d_1, d_2, x_1, x_2, x_3)$. There are two closed
analytic subsets $Z_1$ and $Z_2$. The first set is
\[
	Z_1 = D \times_B Z \times_B X
\]
which is (morally) the set of points where $x_2 = d_2 \cdot x_1$. The second set is
\[
	Z_2 = \menge{(d_1, d_2, x_1, x_2, x_3)}{(d_1, x_2, x_3) \in Z}
\]
which is morally the set of points where $x_3 = d_1 \cdot x_2$. The intersection $Z_1
\cap Z_2$ is closed analytic (and morally consists of points with $x_3 = d_1 \cdot
x_2$ and $x_2 = d_2 \cdot x_1$), and because $f \colon X \to B$ is proper, its image
\[
	p_{1235}(Z_1 \cap Z_2) \subseteq D \times_B D \times_B X \times_B X
\]
is again closed analytic. It contains $W = \phi(G \times_B Z)$ as a Zariski-open
subset; this is enough to conclude that the closure $\bar{W}$ is an
analytic subset of $D \times_B D \times_B X \times_B X$.

Now look at the projection $p_{12} \colon \bar{W} \to D \times_B D$. Over the open
subset $G \times_B D$, this is just $p_{12} \colon G \times_B D \times_B X \to G
\times_B D$, which is proper and flat. We can now apply Fujiki's \Cref{lem:Fujiki};
in the notation used there, $T = (D \times_B D)_{\main}$ and $U = G \times_B D$. The
conclusion is that $\bar{W}$ determines a meromorphic extension of $G \times_B D \to
D$ to a meromorphic mapping $(D \times_B D)_{\main} \mto D$. Because $G$ is
Zariski-open in $D$ by construction, this completes the proof.
\end{proof}

Finally, let's convince ourselves that the action in \eqref{eq:action-P} is meromorphic.

\begin{proposition}
	The holomorphic mapping $a \colon G \times_B X \to X$ in \eqref{eq:action-P}
	extends to a meromorphic mapping $D \times_B X \mto X$.
\end{proposition}

\begin{proof}
Recall that $Z \subseteq D \times_B X \times_B X$ is the pullback of the universal subspace
that comes with the relative Douady space $D_{X \times_B X/B}$. By construction, we
have a commutative diagram
\begin{tcd}
		G \times_B X \rar \dar[hook] & Z \dar[hook] \\
		G \times_B X \times_B X \rar \dar{p_1} & D \times_B X \times_B X \dar{p_1} \\
		G \rar & D
\end{tcd}
in which both squares are Cartesian and the composition of two adjacent vertical arrows
is proper and flat. The projection $p_{12} \colon Z \to D \times_B X$ is proper, and
is an isomorphism over the Zariski-open subset $G \times_B X$ (which is
dense because $f \colon X \to B$ is flat). Therefore $p_{12}$ is a modification, and
therefore bimeromorphic; the composition of its (meromorphic) inverse with $p_3 \colon
Z \to X$ is the desired meromorphic extension of $a \colon G \times_B X \to X$.
\end{proof}

\begin{remark} \label{rmk:stabilizer}
	For each $b \in B$, the group $G_b$ is a meromorphic subgroup of $\Aut^{\circ}(X_b) = \Aut^\circ(X_b)$.
	According to \cite[Prop.~2.7]{fujiki-auto}, it follows that the stabilizer
	\[
		\menge{g \in G_b}{g \cdot x = x}
	\]
	of each point $x \in X_b$ is meromorphically isomorphic to a linear algebraic group
	(Fujiki states this for reduced irreducible spaces, but the same argument works in general.)
\end{remark}

\subsection{The family of neutral components}\label{subs:neutral-component}

As we have seen, the image of the morphism $\eps \colon E \to G$ gives, fiber by
fiber, the neutral component $G_b^{\circ}$ of each commutative complex Lie group $G_b$.
The kernel of $\epsilon$ is described in \Cref{rmk:Lambda}; see also \Cref{subs: netr} for
a Hodge-theoretic interpretation in the case of Lagrangian fibrations.
\Cref{prop:family-groups} shows that the image is open. But $\eps(E)$ is
in fact Zariski-open, and is therefore itself a family of meromorphic groups.
This is true for any family of meromorphic groups, and so we state and prove the
result in that generality.

\begin{proposition}\label{prop:neutral-component}
	Let $p \colon G \to B$ be a family of meromorphic groups, and let
	\[
	G^{\circ}  = \eps (E)
	\subseteq G
	\]
	 denote the union of the neutral components of all the fibers. Then
	$G^{\circ}$ is Zariski-open in $G$, and is therefore itself a family of
	meromorphic groups over $B$.
\end{proposition}

\begin{proof}
It seems that the shortest way to prove this is by using Whitney stratifications
\cite[\S1.2]{goresky-macpherson}. By
definition, there is a proper holomorphic mapping $\pi \colon D \to B$ such that $G$
is Zariski-open in $D$. By properness, one can find (locally finite) Whitney
stratifications of $D$ and $B$ that make $\pi$ into a stratified morphism, in such a
way that $D \setminus G$ (and hence $G$) is a union of strata
\cite[\S1.7]{goresky-macpherson}. Each stratum $S
\subseteq B$ is a locally closed complex submanifold of $B$, and $\pi^{-1}(S) \to S$
is a locally trivial fiber bundle, in a way that is compatible with the
stratification on $\pi^{-1}(S)$.

Our goal is to show that $G^{\circ}$ is Zariski-open in $D$, which is the same thing
as $D \setminus G^{\circ}$ being closed analytic. This is clearly local on $B$, and
so we can assume without loss of generality that our locally finite Whitney
stratifications are actually finite, and that each stratum is connected. Because each
stratum of $D$ is locally closed (in the Zariski topology), and $G^{\circ}$ is open (in
the usual topology), it then suffices to prove that $G^{\circ}$ is a union
of strata. Let $S \subseteq B$ be any stratum, and consider its preimage
$\pi^{-1}(S)$. By Thom's first isotopy lemma \cite[\S1.5]{goresky-macpherson}, this
is a stratified fiber bundle over $S$. The intersection $\pi^{-1}(S) \cap G =
p^{-1}(S)$ is Zariski-open in $\pi^{-1}(S)$, and because $G$ is a union of strata, it
follows that $p^{-1}(S) \to S$ is also a stratified fiber bundle; in particular, each
fiber has the same number of connected components. Now $p^{-1}(S)$ is a complex
manifold, and $p^{-1}(S) \cap G^{\circ}$ is one of its connected components; because
all the strata are connected, it follows that $p^{-1}(S) \cap G^{\circ}$ is a union of
strata. This is enough to conclude that $G^{\circ}$ is a union of strata, hence
Zariski-open in $D$.
\end{proof}

\section{Group actions and freeness}
\label{sec:freeness}

\subsection{Introduction}

In this chapter, we prove a freeness theorem for the cohomology of a compact K\"ahler
space with a faithful action by a meromorphic group. Ng\^o \cite[Prop.~7.5.5]{Ngo-lemme} first
obtained this kind of result in the special case of a commutative algebraic group
acting on a projective variety (by reduction to the case of finite fields). The
argument that we present below uses Hodge theory and the language of Hopf algebras
instead. We also use many results from Fujiki's work on automorphisms of compact K\"ahler
spaces \cite{fujiki-auto}, especially his notion of a ``meromorphic group'' (which is the
complex-analytic analogue of an algebraic group).

Let $X$ be a (possibly nonreduced) compact complex space that is K\"ahler, and denote by
$H^{\ast}(X) = H^{\ast}(X, \RR)$ its cohomology groups with real coefficients.
Suppose that $G$ is a connected meromorphic group (see \Cref{subs: merogrp}) that acts meromorphically and faithfully on
$X$ through biholomorphisms. Since the action is faithful, $G$ is a subgroup of the
neutral component $\Aut^\circ(X)$ of the automorphism group; the other assumptions amount
to saying that $G$ is a meromorphic subgroup of $\Aut^\circ(X)$. According to
\Cref{prop:regular-type}, $G$ is \define{of regular type}, meaning
that there is a short exact sequence
\begin{equation} \label{eq:Chevalley}
	1 \to L \to G \to T \to 1,
\end{equation}
in which $T$ is a compact complex torus, and $L$ is isomorphic (as a meromorphic
group) to a linear algebraic group. This is the analogue of Chevalley's structure
theorem for connected algebraic groups.

Since $G$ is a Lie group, its cohomology $H^{\ast}(G)$ is a \define{Hopf algebra}, which
means that it is both an $\RR$-algebra and an $\RR$-coalgebra, and that the two
structures are compatible. We show (in \Cref{lem:HG}) that there is
a unique isomorphism of Hopf algebras
\[
	H^{\ast}(G) \cong H^{\ast}(L) \tensor_{\RR} H^{\ast}(T)
\]
that also respects the mixed Hodge structures on both sides (which come from the fact that
$G$ and $L$ are meromorphic groups, see \Cref{sec:MHS}).

Now $G$ acts on $X$ through biholomorphisms, and pulling back along the holomorphic
mapping $a \colon G \times X \to X$ defines a morphism
\[
	\delta = a^{\ast} \colon H^{\ast}(X) \to H^{\ast}(G) \otimes H^{\ast}(X),
\]
that makes $H^{\ast}(X)$ into a \define{comodule} over the Hopf algebra
$H^{\ast}(G)$. After composing with the projection to $H^{\ast}(T)$, we obtain
\[
	\delta_T \colon H^{\ast}(X) \to H^{\ast}(T) \otimes H^{\ast}(X),
\]
which makes $H^{\ast}(X)$ into a comodule over the cohomology of the torus $T$.
An interesting point, observed by Ng\^o \cite[p.~116]{Ngo-lemme}, is that on
the level of cohomology, things behave as if the torus $T$ acted on $X$
(which, of course, it doesn't).

We can now state the main result of this chapter.
Using the language of Hopf algebras and comodules, the proof becomes relatively short.

\begin{theorem} \label{thm:freeness}
	Let $X$ be a compact complex space that is K\"ahler. Let $G$ be a
	connected meromorphic group that acts meromorphically and faithfully on $X$; as in
	\eqref{eq:Chevalley}, let $T$ be the maximal compact quotient of $G$. Then
	\[
		H^{\ast}(X) \cong H^{\ast}(T) \tensor_{\RR} H^{\ast}(X)_{\coinv}
	\]
	is \emph{free} as a comodule over the Hopf algebra $H^{\ast}(T)$. Here
	\[
		H^{\ast}(X)_{\coinv}
		= \menge{\alpha \in H^{\ast}(X)}{\delta_T(\alpha) = 1 \tensor \alpha},
	\]
	and the isomorphism respects the mixed Hodge structures on both sides.
\end{theorem}

\begin{example}
The simplest example is a compact K\"ahler space $X$ with a faithful
action by a compact complex torus $T$. In that case, the quotient $X/T$ exists
as a complex space \cite[(5.1)]{fujiki-auto}, and one has
\[
	H^{\ast}(X) \cong H^{\ast}(T) \tensor_{\RR} H^{\ast}(X/T).
\]
One can view the theorem as a generalization of this case.
\end{example}

\subsection{Hopf algebras}

In this section, we review a few basic facts about Hopf algebras.
We are going to use the following simplified definition, based on
\cite[\S3.C]{Hatcher}. In the paper by Milnor and Moore \cite[\S7]{Milnor+Moore},
this would be called a ``connected finite-dimensional commutative and associative
quasi-Hopf algebra over $\RR$''.

\begin{definition}
	A \define{Hopf algebra} is a finite-dimensional graded $\RR$-algebra
	\[
		H = \bigoplus_{k \in \NN} H^k
	\]
	with unit, subject to the following conditions:
	\begin{enumerate}[label=(\alph*)]
	\item $H$ is connected, meaning that $H^0 = \RR \cdot 1$.
	\item The product in $H$ is associative and graded-commutative.
	\item There is a coproduct $\Delta \colon H \to H \tensor_{\RR} H$ that is a
		morphism of graded $\RR$-algebras, where the multiplication on $H \tensor_{\RR}
		H$ is defined by the rule
		\[
			(h_1 \tensor h_2) \cdot (h_3 \tensor h_4) = (-1)^{\deg (h_2) \deg (h_3)} (h_1
			h_3) \tensor (h_2 h_4).
		\]
	\item For every $h \in H^k$ with $k \geq 1$, one has
		\[
			\Delta(h) \equiv 1 \tensor h + h \tensor 1
				\mod \bigoplus_{i=1}^{k-1} H^i \tensor H^{k-i},
		\]
		with $\Delta(1) = 1 \tensor 1$ as a degenerate special case.
	\end{enumerate}
\end{definition}

\begin{example}
The prototypical example of a Hopf algebra is the cohomology $H^{\ast}(G) =
H^{\ast}(G, \RR)$ of a connected complex Lie group $G$. The algebra structure
\[
	\cup \colon H^{\ast}(G) \tensor H^{\ast}(G) \to H^{\ast}(G)
\]
comes from the K\"unneth formula and pullback along the diagonal embedding
\[
	G \to G \times G.
\]
It is clearly associative and graded-commutative. Since $G$ is connected, we have
$H^0(G) = \RR \cdot 1$, where the generator is the pullback of $1 \in \RR$ along the
morphism $G \to \pt$; of course, $1 \in H^0(G)$ is the identity element for the algebra
structure.

Because $G$ is a group, we also get a \emph{coalgebra} structure
\[
	\Delta \colon H^{\ast}(G) \to H^{\ast}(G) \tensor H^{\ast}(G)
\]
from the K\"unneth formula and pullback along the group operation $m \colon
G \times G \to G$. This is a morphism of graded $\RR$-algebras, where the
multiplication on the tensor product $H^{\ast}(G) \tensor H^{\ast}(G)$ is given by
the formula
\[
	(\alpha_1 \tensor \alpha_2) \cup (\alpha_3 \tensor \alpha_4)
	= (-1)^{\deg (\alpha_2) \deg (\alpha_3)} (\alpha_1 \cup \alpha_3) \tensor
	(\alpha_2 \cup \alpha_4).
\]
The seemingly odd property (d) in the definition, namely that
\[
	\Delta(\alpha) \equiv 1 \tensor \alpha + \alpha \tensor 1
	\mod \, \bigoplus_{i=1}^{k-1} H^i(G) \tensor H^{k-i}(G)
\]
for $\alpha \in H^k(G)$, is a consequence of the following commutative diagram:
\begin{tcd}
	G \rar{(e,\id)} \drar[bend right=25]{\id} & G \times G \dar{m} & G
		\lar[swap]{(\id,e)} \dlar[swap,bend left=25]{\id} \\
			 & G
\end{tcd}
Here $e \colon \pt \to G$ is the identity element of the group, and the main point is
that $e^{\ast}(1) = 1$. In particular, we always have $\Delta(1) = 1 \tensor 1$.
\end{example}

The structure of Hopf algebras (in our sense) is very simple, by a classical
result due to Hopf. An element $h \in H^k(G)$ is called \emph{primitive} if $\Delta(h) =
1 \tensor h + h \tensor 1$. By definition, all of $H^1$ is primitive; it is also not
hard to see that nonzero primitive elements can only exist in odd degrees. Indeed,
suppose that $h \in H^k$ is primitive and $k$ is even. Let $n \geq 0$ be the smallest
integer such that $h^{n+1} = 0$. Then
\[
	\Delta(h^{2n}) = (1 \tensor h + h \tensor 1)^{2n}
	= \sum_{i=0}^{2n} \binom{2n}{i} h^i \tensor h^{2n-i}
	= \binom{2n}{n} h^n \tensor h^n,
\]
which is a contradiction unless $n=0$.

\begin{theorem}[Hopf]
	Any Hopf algebra is isomorphic to the wedge algebra on the graded $\RR$-vector
	space consisting of all primitive elements.
\end{theorem}

For a sketch of the proof, see for example \cite[\S2.4]{Cartier-primer}. Hopf's
theorem implies, among many other things, that the coproduct $\Delta \colon H \to H
\tensor H$ is uniquely determined by the product: $H$ is generated as a graded
$\RR$-algebra by primitive elements, and $\Delta$ is by definition a morphism of
graded $\RR$-algebras.

\begin{remark}
The primitive elements in the homology (or cohomology) of a complex Lie group have a
concrete topological interpretation. The homology $H_{\ast}(G)$ is
again a Hopf algebra: the product is now induced by the group operation $m \colon G
\times G \to G$, and the coproduct by the diagonal embedding $G \to G
\times G$.  Except for a sign change in the grading, $H_{\ast}(G)$ is the Hopf
algebra dual to $H^{\ast}(G)$.

Milnor and Moore \cite[Appendix]{Milnor+Moore} show that the subspace of primitive
elements in $H_i(G)$ is exactly the image of the Hurewicz homomorphism
\[
	\pi_i(G,e) \tensor_{\ZZ} \RR \to H_i(G).
\]
Note that $\pi_1(G,e)$ is abelian, and therefore isomorphic to $H_1(G, \ZZ)$, due to $G$
being a Lie group. So all the elements in $H_1(G)$ are primitive, as they should be.
\end{remark}

\begin{example} \label{ex:abelian}
If the complex Lie group $G$ is connected and abelian, then it is a quotient of the
$\CC$-vector space $\Lie G$, and so all higher homotopy groups are trivial.
The subspace of primitive elements is therefore $H^1(G)$, and so
\[
	H^{\ast}(G) \cong \bigwedge H^1(G)
\]
by Hopf's theorem.
\end{example}

\begin{example} \label{ex:coproduct}
Let's see concretely what the coproduct in a Hopf algebra $H$ looks like. Choose a
basis $e_1, \dotsc, e_d$ for the subspace of primitive elements in $H$; because $\deg
e_1, \dotsc, \deg e_d$ are odd integers, we have $e_i e_j + e_j e_i = 0$.
For any nonempty subset $I \subseteq \{1, \dotsc, d\}$, we list the elements of $I$
in increasing order as $I = \{i_1 < \dotsb < i_k\}$, and then define $\deg I = \deg
e_{i_1} + \dotsb + \deg e_{i_k}$ and
\[
	e_I = e_{i_1} \dotsb e_{i_k} \in H^{\deg (I)},
\]
with the understanding that $e_{\emptyset} = 1$. All the $2^d$ elements $e_I$
together then form a vector space basis for the Hopf algebra $H$. In terms of this basis, we have
\begin{equation} \label{eq:Delta}
	\Delta(e_I) = (1 \tensor e_{i_1} + e_{i_1} \tensor 1) \dotsb (1 \tensor e_{i_k} +
	e_{i_k} \tensor 1) =
	\sum_{I = J \sqcup K} \sgn(J, K) \, e_J \tensor e_K.
\end{equation}
Here the sum runs over all subsets $J, K \subseteq I$ such that $J \cap K =
\emptyset$ and $I = J \cup K$; the sign factor $\sgn(J,K)$ is defined by the rule
$e_J \cup e_K = \sgn(J, K) e_{J \cup K}$ whenever $J \cap K = \emptyset$.
\end{example}

\begin{remark}
	The definition of a Hopf algebra usually includes an additional operator $S \colon
	H \to H$ called the \define{antipode}. When $H = H^{\ast}(G)$, the antipode is
	given by $S(\alpha) = i^{\ast}(\alpha)$, where $i \colon G \to G$ is the inverse
	in the group. In general, the defining property of the antipode is that the
	composition
	\begin{tcd}
		H \rar{\Delta} & H \tensor H \rar{S \tensor \id} & H \tensor H \rar & H
	\end{tcd}
	is projection to the subspace $H^0$ in the grading. In a Hopf algebra $H$ in our
	sense, the antipode is given by the formula $S(h) = (-1)^k h$ for $h \in H^k$.
	Indeed, for a primitive element $h \in H^k$, we have
	\[
		(S \tensor \id)(\Delta(h)) = 1 \tensor h + (Sh) \tensor 1,
	\]
	and so $h + Sh = 0$; the claim now follows from Hopf's theorem, because all primitive
	elements have odd degree.
\end{remark}

\subsection{Mixed Hodge structures} \label{sec:MHS}

Let $G$ be a meromorphic group, and assume that the compactification $G^{\ast}$ is a compact
K\"ahler space, or is at least in Fujiki's class $\fuj$, meaning that $G^{\ast}$ is the image
of a compact K\"ahler manifold under a proper holomorphic mapping. This is true
for all the meromorphic groups that we encounter in this paper.
Building on Deligne's results in \cite{deligne-hodgeII} and \cite{deligne-hodgeIII}, Fujiki \cite{fujiki-MHS}
proved that each cohomology group $H^k(G) = H^k(G, \RR)$ has a functorial real mixed Hodge structure of
weight $\geq k$; the lowest weight piece
\[
	W_k H^k(G) = \im \bigl( H^k(G^{\ast}) \to H^k(G) \bigr)
\]
is the image of the restriction mapping. By definition, the group operation $m \colon G \times G \to G$
extends to a meromorphic mapping $G^{\ast} \times G^{\ast} \mto G^{\ast}$, and this
implies (as in Deligne's setting) that the pullback morphism
\[
	m^{\ast} \colon H^{\ast}(G) \to H^{\ast}(G) \tensor H^{\ast}(G)
\]
is a morphism of mixed Hodge structures (over $\RR$). The Hopf algebra $H^{\ast}(G)$
therefore comes with a compatible mixed Hodge structure. We will see in the next section
that the piece of lowest weight $W_k H^k(G)$ is the cohomology of the maximal compact quotient of $G$.

Now suppose that $G$ acts meromorphically on a compact K\"ahler space $X$. Each cohomology
group $H^k(X)$ then has a mixed Hodge structure of weight $\leq k$, and because
$a \colon G \times X \to X$ is the restriction of a meromorphic mapping $G^{\ast} \times X \mto X$,
the pullback morphism
\[
	a^{\ast} \colon H^{\ast}(X) \to H^{\ast}(G) \tensor H^{\ast}(X)
\]
is again a morphism of mixed Hodge structures. This  means the comodule
structure on $H^{\ast}(X)$ in \Cref{sec:comodules} is also compatible with mixed Hodge structures.

\subsection{Cohomology of meromorphic groups}

The following lemma describes the cohomology of the meromorphic group $G$. We define
the compact complex torus $T$ and the linear algebraic group $L$ as in
\eqref{eq:Chevalley}.

\begin{lemma} \label{lem:HG}
	There is a functorial isomorphism $H^{\ast}(G) \cong H^{\ast}(T) \tensor H^{\ast}(L)$
	as Hopf algebras and as mixed Hodge structures (over $\RR$).
\end{lemma}

\begin{proof}
Let's first prove that the two sides are isomorphic as $\RR$-vector spaces.
Write $q \colon G \to T$ for the morphism in \eqref{eq:Chevalley}.
Consider the direct image $\derR \ql \RR_G \in \Dbc(T, \RR)$, in the
derived category of complexes of constructible sheaves on $T$. The direct image is
constructible because $G$, being a meromorphic group, is Zariski-open in a compact
complex manifold of class $\fuj$ whose Albanese torus is $T$; this means that $q \colon G \to T$
is the composition of a Zariski-open embedding and a proper morphism.
Fujiki \cite[Lem.~3.10]{fujiki-auto} shows that  there is a finite covering space
$\pi \colon T' \to T$ such that the fiber product $T' \times_T G \to T'$ is trivial
as a $C^{\infty}$-fiber bundle. This means that
\[
	\pi^{\ast} \derR \ql \RR_G \cong
	\bigoplus_{i=0}^{2 \dim L} H^i(L) \tensor \RR_{T'} \decal{-i}
\]
splits into a direct sum of constant local systems. Because $\derR \ql \RR_G$
is a direct summand in $\derR \pi_{\ast} \pi^{\ast} \derR \ql \RR_G$, it therefore
also splits as a sum of local systems with finite monodromy. Each of the local
systems $R^i \ql \RR_G$ is actually trivial. The reason is that the
diagram
\begin{tcd}
	L \times G \rar{m} \dar{p_2} & G \dar{q} \\
	G \rar{q} & T
\end{tcd}
is Cartesian, and so the projection formula shows that the pullback of $R^i
\ql \RR_G$ to $G$ is a trivial local system. Because $L = \ker q$
is connected, the original local system on $T$ is therefore also trivial. By
taking cohomology, we then get an isomorphism of graded vector spaces
\[
	H^{\ast}(G) \cong H^{\ast} \bigl( T, \derR \ql \RR_G \bigr)
	\cong H^{\ast}(T) \tensor H^{\ast}(L).
\]
In particular, by consideration of the edge morphisms, $H^{\ast}(T) \to H^{\ast}(G)$ is injective and $H^{\ast}(G) \to
H^{\ast}(L)$ is surjective.

Now let's deal with Hopf algebras and mixed Hodge structures.
By Hopf's theorem, $H^{\ast}(G)$ is isomorphic to the wedge algebra on
the space of primitive elements $H^{\ast}(G)_{\prim}$. As the kernel of
\begin{tcd}
		H^n(G) \rar{\Delta} & \displaystyle \bigoplus_{i=0}^n H^i(G) \tensor
		H^{n-i}(G) \rar & \displaystyle \bigoplus_{i=1}^{n-1} H^i(G) \tensor
		H^{n-i}(G),
\end{tcd}
which is a morphism of mixed Hodge structures, each $H^n(G)_{\prim}$ is a mixed Hodge
substructure of $H^n(G)$. In degree $1$, we have a short exact sequence
\[
	0 \to H^1(T) \to H^1(G) \to H^1(L) \to 0,
\]
where $H^1(T)$ is pure of weight $1$, and $H^1(L)$ pure of weight $2$ by
\cite[Thm.~9.1.5]{deligne-hodgeIII}. According to a construction by Deligne
\cite[Lem.~1.2.11]{deligne-hodgeII}, any mixed Hodge
structure over $\RR$ with only two adjacent weights splits in a unique and functorial
way;  therefore, we have a unique isomorphism of mixed Hodge structures
\[
	H^1(G) \cong H^1(T) \oplus H^1(L).
\]
In degree $n \geq 2$, the morphism of mixed Hodge structures $H^n(G) \to H^n(L)$
is also surjective. Its kernel is the direct sum of the subspaces $H^i(T) \tensor H^{n-i}(L)$
with $i \geq 1$, and therefore does not contain any primitive elements (by Hopf's theorem).
It follows that $H^n(G)_{\prim} \to H^n(L)_{\prim}$ is
an isomorphism for $n \geq 2$. Thus
\[
	H^{\ast}(G)_{\prim} \cong H^1(T) \oplus H^{\ast}(L)_{\prim}
\]
as mixed Hodge structures. We can now apply Hopf's theorem again and
conclude that $H^{\ast}(G)$ is isomorphic, as a Hopf algebra, to the tensor
product of $H^{\ast}(T)$ and $H^{\ast}(L)$. The isomorphism respects the mixed
Hodge structures because we know that $H^{\ast}(G)_{\prim}$ is a mixed Hodge
substructure.
\end{proof}

Together with Deligne's result \cite[Thm.~9.1.5]{deligne-hodgeIII}, the lemma says that
$H^k(T)$ is the unique summand
of $H^k(G)$ that is pure of weight $k$. Indeed, we have
\[
	H^k(G) \cong H^k(T) \tensor H^0(L) \oplus \bigoplus_{i=1}^k H^{k-i}(T) \tensor
	H^i(L),
\]
and because $W_i H^i(L) = 0$ for $i \geq 1$, all the other summands have weight $\geq
k+1$. Moreover, as $H^0(L) = \RR \cdot 1$, the projection
\[
	H^{\ast}(G) \cong H^{\ast}(T) \tensor H^{\ast}(L) \to H^{\ast}(T)
\]
is both a morphism of Hopf algebras and a morphism of mixed Hodge structures.

\subsection{Comodules} \label{sec:comodules}

The freeness theorem is stated in terms of comodules, so let's briefly consider
comodules over an arbitrary Hopf algebra $H$.

\begin{definition}
A \emph{comodule} over the Hopf algebra $H$ is a finite-dimensional graded
$\RR$-vector space $M$ together with a morphism
\[
	\delta \colon M \to H \tensor M,
\]
subject to the following two conditions:
\begin{enumerate}[label=(\alph*)]
\item For every $m \in M^k$, one has
	\[
		\delta(m) \equiv 1 \tensor m \mod \bigoplus_{i=1}^k H^i \tensor M^{k-i}.
	\]
\item The following diagram commutes:
	\begin{tcd}
		M \dar{\delta} \rar{\delta} & H \tensor M \dar{\Delta \tensor \id} \\
		H \tensor M \rar{\id \tensor \delta} & H \tensor H \tensor M
	\end{tcd}
\end{enumerate}
	Note that (a) and (b) look almost like the definition of a module, except that all the arrows are reversed.
\end{definition}

\begin{example}
The comodules we are interested in come from the cohomology of spaces with
$G$-action. Let $G$ be a connected complex Lie group, and suppose that $a \colon G
\times X \to X$ is a holomorphic $G$-action on a complex space $X$. Pulling back
along $a$ and using the K\"unneth isomorphism, we get a morphism
\[
	a^{\ast} \colon H^{\ast}(X) \to H^{\ast}(G) \tensor H^{\ast}(X)
\]
that makes $H^{\ast}(X)$ into a comodule over $H^{\ast}(G)$. The required
compatibilities holds because of the following commutative diagrams:
\[
	\begin{tikzcd}
		G \times G \times X \rar{\id \times a} \dar{m \times \id} & G \times X \dar{a}
		\\
		G \times X \rar{a} & X
	\end{tikzcd}
	\hspace{1cm}
	\begin{tikzcd}
			X \rar{(e,\id)} \drar[swap,bend right=30]{\id} & G \times X \dar{m} \\
			& X
	\end{tikzcd}
\]
\end{example}

Hopf's theorem tells us that every Hopf algebra is a wedge algebra. This allows us to
describe comodules over $H$ very concretely. Suppose that
\[
	\delta \colon M \to H \tensor M
\]
is a comodule structure on a finite-dimensional graded $\RR$-vector space $M$. As in
\Cref{ex:coproduct}, choose a basis $e_1, \dotsc, e_d$ for the subspace of primitive
elements; then the $2^d$ elements $e_I \in H^{\deg I}$ are a basis for $H$.
This allows us to write $\delta$ as
\begin{equation} \label{eq:delta}
	\delta(m) = \sum_I e_I \tensor \delta_I(m),
\end{equation}
where each $\delta_I \colon M^{\bullet} \to M^{\bullet - \deg I}$ is $\RR$-linear.
Because $M$ is a comodule, $e_{\emptyset} = 1$ and $\delta_{\emptyset}(m) =
m$. Let's now investigate the comodule condition $(\Delta \tensor \id) \circ \delta =
(\id \tensor \delta) \circ \delta$. From the formula for $\Delta$ in
\eqref{eq:Delta}, we get
\[
	(\Delta \tensor \id) \bigl( \delta(m) \bigr)
	= \sum_I \Delta(e_I) \tensor \delta_I(m)
	= \sum_{J \cap K = \emptyset} \sgn(J,K) \, e_J \tensor e_K \tensor \delta_{J \cup K}(m).
\]
On the other hand, we have
\[
	(\id \tensor \delta) \bigl( \delta(m) \bigr)
	= \sum_J e_J \tensor \delta \bigl( \delta_J(m) \bigr)
	= \sum_{J,K} e_J \tensor e_K \tensor \delta_K \bigl( \delta_J(m) \bigr).
\]
Comparing the two formulas, we find that $\delta_K \circ \delta_J = 0$ whenever $J
\cap K \neq \emptyset$; moreover, if $J \cap K = \emptyset$, we get $\sgn(J, K)
\delta_{J \cup K} = \delta_K \circ \delta_J$. From this, it is easily deduced that if
$I = \{i_1 < \dotsb < i_k\}$ is any ordered subset of $\{1, \dotsc, d\}$, then
\[
	\delta_I = \delta_{i_k} \circ \dotsb \circ \delta_{i_1}.
\]
The entire comodule structure is therefore determined by the $d$ linear mappings
\[
	\delta_i \colon M^{\bullet} \to M^{\bullet-\deg (e_i)}.
\]
These anticommute: $\delta_i \circ \delta_j + \delta_j \circ \delta_i = 0$, hence
also $\delta_i \circ \delta_i = 0$.

\begin{definition}
A comodule is called \emph{free} if there is a graded $\RR$-vector space $V$ and an
isomorphism $\phi \colon M \to H \tensor V$, in such a way that the diagram
\begin{tcd}
		M \dar{\phi} \rar{\delta} & H \tensor M \dar{\id \tensor \phi} \\
		H \tensor V \rar{\Delta \tensor \id} & H \tensor H \tensor V
\end{tcd}
is commutative.
\end{definition}

As before, this is dual to the familiar notion of a free module.
It is not hard to see that if $M$ is a free comodule over $H$, then $V$ is
isomorphic, as a graded $\RR$-vector space, to the space of \define{coinvariants}
\[
	M_{\coinv} = \menge{m \in M}{\delta(m) = 1 \tensor m}.
\]
For later use, we note the following simple lemma; the main point is that a free
comodule over $H$ is also a free module over the dual Hopf algebra $\Hom_{\RR}(H, \RR)$.

\begin{lemma} \label{lem:comodule-free}
Suppose that $M$ is a free comodule over $H$, such that the space of coinvariants
$M_{\coinv}$ lives in a single degree $k$. Let $d = \dim H_{\prim}$ be the dimension
of the subspace of primitive elements. Then
\begin{tcd}
\Hom_{\RR}(H^i, \RR) \tensor M^{d+k} \rar{\id \tensor \delta } &
\Hom_{\RR}(H^i, \RR) \tensor H^i \tensor M^{d+k-i} \rar{\operatorname{ev} \tensor \id}
& M^{d+k-i}
\end{tcd}
is an isomorphism for all $i \in \ZZ$.
\end{lemma}

\begin{proof}
	Note that $M^{d+k} \cong H^d \tensor M_{\coinv}^k$.
	As $M \cong H \tensor M_{\coinv}$, we immediately reduce to the case $M = H$, which
	which is an exercise in linear algebra. Indeed, if $e_1, \dotsc, e_d$ are a basis for the
	subspace of primitive elements in $H$, then
	\[
	\Delta(e_{\{1, \dotsc, n\}}) =
	\sum_{\{1, \dotsc, n\} = J \sqcup K} \sgn(J, K) \, e_J \tensor e_K.
	\]
	as a special case of the formula for the coproduct in \eqref{eq:Delta}. It follows that
	\[
		\Hom_{\RR}(H^i, \RR) \tensor H^d \to H^{d-i}
	\]
	is an isomorphism for $i = 0, \dotsc, d$, as claimed.
\end{proof}

\subsection{Hopf modules and freeness}

Let $H$ be a Hopf algebra. In the same way that a Hopf algebra is both an algebra
and a coalgebra, a Hopf module is both a module and a comodule
over $H$ in a compatible way.
\begin{definition}
	Let $\delta \colon M \to H \tensor M$ be a comodule over the Hopf algebra $H$.
	If $M$ is also a left $H$-module, in such a way that
	\[
		\delta(h \cdot m) = \Delta(h) \cdot \delta(m)
	\]
	for all $h \in H$ and all $m \in M$, then $M$ is called a \define{Hopf module}.
\end{definition}
Here $H \tensor M$ is viewed as a left $H \tensor H$-module by the formula
\begin{align*}
	(H \tensor H) \tensor (H \tensor M) &\to H \tensor M \\
	(h_1 \tensor h_2) \tensor (h_3 \tensor m) &\mapsto
	(-1)^{\deg (h_2) \deg (h_3)}	(h_1 h_3) \tensor (h_2 \cdot m),
\end{align*}
where the first product is multiplication in the algebra $H$. We will deduce the
freeness theorem from the following algebraic result about Hopf modules. In the
literature, this is called the ``fundamental theorem of Hopf modules''
\cite[\S1.9]{Montgomery}.

\begin{proposition} \label{prop:Hopf-module}
	Every Hopf module is free. More precisely, if $M$ is a Hopf module over $H$, then
	$M \cong H \tensor M_{\coinv}$ both as a comodule and as a left $H$-module.
\end{proposition}

\begin{proof}
We include an elementary proof for the benefit of the reader.
Let $e_1, \dotsc, e_d$ be a basis
for the subspace of primitive elements in $H$. Recall from \eqref{eq:delta} that the
comodule structure on $M$ takes the form
\[
	\delta(m) = 1 \tensor m + \sum_{i=1}^d e_i \tensor \delta_i(m)
	+ \sum_{\abs{I} \geq 2} e_I \tensor \delta_I(m),
\]
and that the $d$ anticommuting linear maps $\delta_i \colon M^{\bullet} \to
M^{\bullet - \deg e_i}$ determine the comodule structure because $\delta_I =
\delta_{i_k} \circ \dotsb \circ \delta_{i_1}$ for any ordered set $I = \{i_1 < \dotsb
< i_k\}$. For simplicity, let us write $e_i \colon M^{\bullet} \to M^{\bullet + \deg
e_i}$ for multiplication by the primitive element $e_i \in H^{\deg e_i}$.
Because $M$ is a Hopf module, we have
\[
	\delta \bigl( e_j m \bigr) = \Delta(e_j) \cdot \delta(m) =
		(1 \tensor e_j + e_j \tensor 1) \cdot \delta(m)
\]
for any $m \in M$. From this, one deduces (by looking at the term with $e_i$) that
\begin{equation} \label{eq:relations}
	\begin{split}
	\delta_i \circ e_j + e_j \circ \delta_i &= 0 \quad \text{for $i \neq j$,} \\
	\delta_i \circ e_i + e_i \circ \delta_i &= \id.
	\end{split}
\end{equation}
For example, the second relation and the fact that $\delta_i \circ \delta_i = 0$ and
$e_i \circ e_i = 0$ imply that $M = \ker \delta_i \oplus \im e_i$ for every $i = 1,
\dotsc, d$.

Now our goal is to prove that we have an isomorphism
\begin{equation} \label{eq:isomorphism}
	\phi \colon H \tensor M_{\coinv} \to M, \quad \phi(h,m) = hm,
\end{equation}
both as comodules and as left $H$-modules. The subspace of coinvariants is
\[
	M_{\coinv} = \menge{m \in M}{\delta(m) = 1 \tensor m}
	= \menge{m \in M}{\text{$\delta_i(m) = 0$ for $i=1, \dotsc, d$}}.
\]
By construction, \eqref{eq:isomorphism} is a
morphism of left $H$-modules. It is also a morphism of comodules. For that, we need
to prove that the diagram
\begin{tcd}
	H \tensor M_{\coinv} \rar{\phi} \dar[xshift=-1em]{\Delta \tensor \id} & M \dar{\delta} \\
	H \tensor H \tensor M_{\coinv} \rar{\id \tensor \phi} & H \tensor M
\end{tcd}
commutes. Because $m \in M_{\coinv}$ satisfies $\delta(m) = 1 \tensor m$, we have
\[
	\delta(e_I m) = \Delta(e_I) \cdot \delta(m)
	= \sum_{I = J \sqcup K} \sgn(J,K) \, e_J \tensor (e_K m).
\]
On the other hand,
\[
	\Delta(e_I) \tensor m = \sum_{I = J \sqcup K} \sgn(J, K) \, e_J \tensor e_K
		\tensor m,
\]
and this clearly goes to the same element under $\id \tensor \phi$.

All that is left is to show that \eqref{eq:isomorphism} is bijective. Let's first
prove injectivity. Consider an arbitrary element
\[
	\sum_I e_I \tensor m_I \in \ker \phi.
\]
By decomposing according to the grading, we may assume that $m_I \in H^{k-\deg I}$.
According to the computation we have just done, $\delta \circ \phi$ takes our
element to
\[
	0 = \sum_{J \cap K = \emptyset} \sgn(J, K) \, e_J \tensor \phi(e_K \tensor m_{J
	\cup K}).
\]
Since the $e_J$ form a basis in $H$, it follows that for each $J
\subseteq \{1, \dotsc, d\}$, the element
\[
	\sum_{K \colon J \cap K = \emptyset} \sgn(J, K) \, e_K \tensor m_{J \cup K}
\]
also belongs to the kernel of $\phi$. By induction on the degree, we conclude that
$m_{J \cup K} = 0$ whenever $J \cap K = \emptyset$ and $\abs{J} \geq 1$; a moment's
thought shows that therefore $m_I = 0$ for $\abs{I} \geq 1$. Our element
in the kernel now has the form $1 \tensor m$, but as $\phi(1 \tensor m) = m$, we
conclude that $m = 0$, and hence that $\phi$ is injective.

Let's prove that \eqref{eq:isomorphism} is also surjective. Consider the morphism
\[
	\phi \colon (H \tensor M_{\coinv})^k \to M^k
\]
in a given degree $k \geq 0$. From the relations in \eqref{eq:relations}, it is easy to
prove that
\[
	M^k = M_{\coinv}^k \oplus \sum_{i=1}^d e_i M^{k-1}.
\]
Since $e_i m = \phi(e_i \tensor m)$, and since $\phi$ is surjective in degree $k-1$
by induction, all the terms except for $M_{\coinv}^k$ are clearly in the image of
$\phi$. But for $m \in M_{\coinv}^k$, we have $\phi(1 \tensor m) = m$, and so the
same is true for the first summand. This finishes the proof that
\eqref{eq:isomorphism} is an isomorphism of Hopf modules.
\end{proof}

\subsection{Proof of the Freeness Theorem}
In this section, we prove the Freeness Theorem, i.e.  \Cref{thm:freeness}. Let $X$ be a compact complex space
that is K\"ahler, and let $G$ be a connected meromorphic group that acts faithfully and
meromorphically on $X$. Denote by $T$ the maximal compact quotient of $G$, as in
\eqref{eq:Chevalley}.

We start by simplifying the problem a bit. Without loss of generality, we may assume that $X$
is connected. We can further arrange that $X$ is reduced. Indeed, the reduction $X_{\red}$ has the
same cohomology as $X$ itself, and since $G$ acts faithfully and meromorphically on $X$,
the kernel of $G \to \Aut^\circ(X_{\red})$ is (meromorphically isomorphic to) a
linear algebraic group (see \Cref{rmk:stabilizer}). The kernel
therefore has finite image in $T$, which means that the image of $G$ in $\Aut^{\circ}(X_{\red})$
is a meromorphic group that acts faithfully and meromorphically on $X_{\red}$ and whose
maximal compact quotient is isomorphic to the quotient of $T$ by a finite subgroup (and
therefore has the same cohomology as $T$).

From now on, $X$ is connected and reduced. Consider the comodule structure
\[
	\delta \colon H^{\ast}(X) \to H^{\ast}(G) \tensor H^{\ast}(X)
\]
coming from the action $a \colon G \times X \to X$. From \Cref{lem:HG},
we have a surjective morphism of Hopf algebras (and mixed Hodge structures)
$H^{\ast}(G) \to H^{\ast}(T)$, and so we can also view the cohomology of $X$ as a comodule
\[
	\delta_T \colon H^{\ast}(X) \to H^{\ast}(T) \tensor H^{\ast}(X)
\]
over the smaller Hopf algebra $H^{\ast}(T)$. The most important component of $\delta$ is
\[
	\delta \colon H^1(X) \to H^0(G) \tensor H^1(X) \oplus H^1(G) \tensor H^0(X).
\]
This gives us a homomorphism $p \colon H^1(X) \to H^1(G)$ with the property that
\[
	\delta(\alpha) = 1 \tensor \alpha + p(\alpha) \tensor 1
	\qquad \text{for $\alpha \in H^1(X)$.}
\]
As the action is meromorphic, $p$ is a morphism of mixed Hodge structures.
For weight reasons, its image must therefore be contained in $W_1 H^1(G) \cong H^1(T)$.
Composing with the projection from $H^{\ast}(G)$ to $H^{\ast}(T)$, we also have
\[
	\delta_T(\alpha) = 1 \tensor \alpha + p(\alpha) \tensor 1
	\qquad \text{for $\alpha \in H^1(X)$.}
\]
The key point is now the following lemma.

\begin{lemma} \label{lem:H1}
	The homomorphism $p \colon H^1(X) \to H^1(T)$ is surjective.
\end{lemma}

We will prove this lemma in the next section. By Deligne's result, we again have a
unique decomposition $H^1(X) \cong W_0 H^1(X) \oplus \gr_1^W H^1(X)$ as mixed Hodge
structures. Because $X$ is a K\"ahler space, all the Hodge structures are polarized;
this allows us to pick a section
\[
	s \colon H^1(T) \to H^1(X)
\]
with $p \circ s = \id$ that is itself a morphism of mixed Hodge structures. It
extends in the obvious way to a morphism of graded $\RR$-algebras
\[
	s \colon H^{\ast}(T) \to H^{\ast}(X).
\]
Using cup product, this turns $H^{\ast}(X)$ into a left module over the algebra
$H^{\ast}(T)$:
\[
	H^{\ast}(T) \tensor H^{\ast}(X) \to H^{\ast}(X), \quad
	\alpha \tensor \beta \mapsto s(\alpha) \cup \beta.
\]
For $\alpha \in H^1(T)$, the fact that $\delta$ preserves cup products leads to
the identity
\[
	\delta \bigl( s(\alpha) \cup \beta \bigr)
	= \delta \bigl( s(\alpha) \bigr) \cup \delta(\beta)
	= (1 \tensor \alpha + \alpha \tensor 1) \cup \delta(\beta).
\]
Because $H^{\ast}(T)$ is isomorphic to the wedge algebra on $H^1(T)$, it follows that
\[
	\delta_T \bigl( s(\alpha) \cup \beta \bigr)
	= \Delta(\alpha) \cup \delta_T(\beta) \qquad
	\text{for $\alpha \in H^{\ast}(T)$ and $\beta \in H^{\ast}(X)$.}
\]
In other words, $H^{\ast}(X)$ is a Hopf module over the Hopf algebra $H^{\ast}(T)$.
We can now apply \Cref{prop:Hopf-module} and conclude that $H^{\ast}(X)$ is free as a
comodule over $H^{\ast}(T)$.

\begin{example} \label{ex:freeness-Q}
	Here is an example that shows why we stated the freeness theorem over $\RR$.
  Let $T$ be a compact complex torus, $L \in \Pic^0(T)$ a
	holomorphic line bundle of infinite order. We have an exact sequence of commutative
	complex Lie groups
	\[
		1 \to \CC^{\ast} \to G \to T \to 1,
	\]
	where $G$ is the total space of the line bundle with the zero section removed. The
	projectivization of the vector bundle $L \oplus \shO_T$ is a $\PP^1$-bundle over $T$
	with two sections $s_0$ and $s_1$. We now glue these two sections together, by
	identifying the point $s_0(x)$ with the point $s_1(x) + \alpha$, where $\alpha \in T$
	is an element that generates a dense subgroup of $T$. This produces a compact complex
	space $X$ that is
	singular along a copy of $T$, with the property that there is no holomorphic mapping
	from $X$ to any compact complex torus.
	By comparing $X$ to the $\PP^1$-bundle, one finds an exact sequence
	\[
		0 \to \ZZ(0) \to H^1(X, \ZZ) \to H^1(T, \ZZ) \to 0
	\]
	of mixed Hodge structures (over $\ZZ$) whose extension class is exactly $\alpha \in T$.
	This extension does not split over $\QQ$, and so the cohomology $H^{\ast}(X, \QQ)$
	cannot be free as a comodule over $H^{\ast}(T, \QQ)$ in a way that is compatible with
	mixed Hodge structures.
\end{example}

\subsection{Proof of Lemma~\ref*{lem:H1}}
\label{subs: pf lemma surjective}
In this section, we prove \Cref{lem:H1}, which claims that $H^1(X)$ maps onto $H^1(T)$.
The linear mapping $p \colon H^1(X) \to
H^1(T)$ has a simple description in terms of the operators $\delta_i$.
Write $r = \dim T$ for the dimension of the compact complex torus $T$, and choose a
basis $e_1, \dotsc, e_{2r} \in H^1(T)$. As $H^0(X) = \RR \cdot 1$, we may view each
$\delta_i \colon H^1(X) \to H^0(X)$ as a linear functional on $H^1(X)$, and then
\[
	p \colon H^1(X) \to H^1(T), \quad
	p(\alpha) = \sum_{i=1}^{2r} \delta_i(\alpha) e_i.
\]
We claim that the map is surjective, or equivalently, that $\delta_1, \dotsc,
\delta_{2r}$ are linearly independent as elements of the dual space. Suppose by contradiction
that they are linearly dependent. Without loss of generality, we then have a linear relation of the
form
\[
	\delta_{2r} = \sum_{j=1}^{2r-1} a_j \delta_j,
\]
Because $\delta_i \circ \delta_j + \delta_j \circ \delta_i = 0$, this implies that
the composition
\[
	\delta_{2r} \circ \dotsb \circ \delta_1 \colon H^{2r}(X) \to H^0(X)
\]
is equal to zero. But we know from \eqref{eq:delta} that in the morphism
\[
	\delta_T \colon H^{2r}(X) \to \bigoplus_{i=0}^{2r} H^i(T) \tensor H^{2r-i}(X),
\]
the component in degree $i=2r$ is given by the formula
\[
	H^{2r}(X) \to H^{2r}(T) \tensor H^0(X), \quad m \mapsto e_1 \dotsm e_{2r}
	\tensor (\delta_{2r} \circ \dotsb \circ \delta_1)(m).
\]
So this component is also zero. In terms of the $G$-action on $X$, this is saying
that if $x \in X$ is an arbitrary point, then the pullback along
\[
	G \to X, \quad g \mapsto g \cdot x,
\]
induces the trivial map $H^{2r}(X) \to H^{2r}(G)$.

We can now get a contradiction from the fact that $X$ is K\"ahler. As $G$ is acting
meromorphically on $X$, the closure of any orbit is again a union of
orbits. By choosing an orbit of minimal dimension, we can therefore find a compact
submanifold $Y \subseteq X$ that is a single orbit of the $G$-action.
This makes $Y$ a homogeneous compact K\"ahler manifold; by Fujiki's version of the
theorem of Borel and Remmert \cite[\S6b]{fujiki-auto}, $Y$ is biholomorphic to $A
\times F$, where $A \cong \Alb(Y)$ is a compact complex torus, and $F$ is a rational
homogeneous projective manifold. Clearly $G$ acts meromorphically on $Y$, and the
kernel of the homomorphism $G \to \Aut^\circ(Y)$ is the stabilizer of a point on $X$,
hence a linear algebraic group \cite[Prop.~2.7]{fujiki-auto}, due to the fact
that $G$ acts faithfully on $X$. Composing with the projection to $A$, we obtain a
morphism of meromorphic groups $G \to A$; it factors through the maximal compact
quotient of $G$ and produces a homomorphism
\[
	\varphi \colon T \to A.
\]
Now $\varphi$ is surjective (because $G$ acts transitively on $Y$), and because $A
\cong \Alb(Y)$, we can apply \cite[Thm.~5.5]{fujiki-auto} and conclude that $\ker
\varphi$ is a finite group. In particular, $\dim A = \dim T = r$. The result is
that we have a commutative diagram
\begin{tcd}
	& A \dar \drar[bend left=20] \\
	G \dar{q} \rar& Y \dar \rar& X \\
	T \rar& \Alb(Y)
\end{tcd}
in which the composition of the two vertical arrows is an isomorphism and the bottom
arrow is an isogeny. In cohomology, this gives
\begin{tcd}
	& H^{2r} \bigl( \Alb(Y) \bigr) \dar \rar{\cong} & H^{2r}(T) \dar[hook] \\
	H^{2r}(X) \drar[bend right=20] \rar& H^{2r}(Y) \dar \rar& H^{2r}(G) \\
	& H^{2r}(A),
\end{tcd}
and because $H^{2r}(X) \to H^{2r}(T)$ is the zero map, it follows that $H^{2r}(X) \to
H^{2r}(A)$ is also zero. Let $\omega \in H^2(X)$ be a K\"ahler class; then
$\omega^r \in H^{2r}(X)$ goes to zero in $H^{2r}(A)$, and this contradicts the fact
that $A$ is a compact K\"ahler manifold of dimension $r$.
Therefore $p \colon H^1(X) \to H^1(T)$ must be surjective.

\section{The support theorem} \label{sec:support-theorem}

\subsection{Introduction}

In this chapter, we prove a version of Ng\^o's support theorem. Here is the general
setting. Let $f \colon X \to B$ be a proper holomorphic mapping
from a K\"ahler manifold $X$ to a complex manifold $B$. We make the following
assumptions:
\begin{enumerate}
	\item The morphism $f$ is surjective, and all the fibers $X_b = f^{-1}(b)$ are
		connected and have the same dimension $n$. Consequently, $\dim X = \dim B + n$,
		and $f$ is flat of relative dimension $n$.
	\item We have a family $p \colon G \to B$ of meromorphic groups, and a
		meromorphic action $a \colon G \times_B X \to X$, such that each group $G_b =
		p^{-1}(b)$ is $n$-dimensional and commutative, and acts faithfully on the
		(possibly nonreduced) fiber $X_b$.
	\item The family of meromorphic groups is \define{$\delta$-regular}. This means
		that
		\[
			\codim_B \menge{b \in B}{\dim T_b \leq r} \geq n-r
			\quad \text{for all $r=0, \dotsc, n$},
		\]
		where $T_b$ is the maximal compact quotient of $G_b$ as in
		\eqref{eq:Chevalley}.
\end{enumerate}
In particular, $f$ is generically a \define{torus fibration}, in the sense that the
nonsingular fibers $X_b$ are compact complex tori of dimension $n$, with
$G_b^{\circ} \cong \Aut^\circ(X_b)$. In the projective case, this is basically Ng\^o's definition of a
\define{$\delta$-regular weak abelian fibration} \cite[7.1.5]{Ngo-lemme}, except
that we assume that the action is faithful (which implies by \Cref{rmk:stabilizer}
that the stabilizer of every point is affine) and that we do not need the condition on
``polarizability of the Tate module''.
Recall that for each $b \in B$, we have the maximal compact quotient $G^{\circ}_b \to T_b$.
Indeed, by \Cref{prop:regular-type}, each $G_b^{\circ}$ is a connected meromorphic
group of regular type (either because it is commutative, or because it acts faithfully
on the compact K\"ahler space $X_b$).

The support theorem describes the structure of the perverse sheaves that appear in
the decomposition theorem for the morphism $f \colon X \to B$. The content of the
decomposition theorem is that there is an isomorphism
\begin{equation} \label{eq:decomposition}
	\derR \fl \RR_X[\dim X] \cong \bigoplus_{j=-n}^n P_j \decal{-j},
\end{equation}
in the derived category $\Dbc(B, \RR)$ of constructible complexes with coefficients in
$\RR$. Each $P_j$ is a perverse sheaf on $B$, and the bounds in the sum come from the
fact that $f$ is equidimensional. Note that the decomposition theorem for proper
holomorphic mappings from K\"ahler manifolds is now fully known, due to Saito's work
\cite{Saito-Kaehler}, augmented by Mochizuki's more recent paper \cite{mochizuki}.
Each perverse sheaf $P_j$ further admits a locally finite \define{decomposition by
strict support}
\begin{equation} \label{eq:support-decomposition}
	P_j = \bigoplus_{Z \subseteq B} P_{j, Z},
\end{equation}
where the sum runs over irreducible closed analytic subsets of $B$. Here $P_{j,Z}$ is
the intersection complex of a local system defined on a Zariski-dense open subset of $Z$.
An irreducible closed subset $Z \subseteq B$ is called a \define{support} of $f \colon X \to B$ if some
$P_{j,Z} \neq 0$.

\begin{remark}
	In fact, each $P_j$ is part of a polarizable real Hodge module of weight $j + \dim
	X$, and the local system from which $P_{j,Z}$ is constructed comes from a
	polarizable real variation of Hodge structure of weight $j + \dim X - \dim Z$.
\end{remark}

Having introduced all the relevant objects, we can now state the support theorem. The
result was first proved in the algebraic case (with a few extra assumptions on
the family of groups) by Ng\^o \cite[Thm.~7.2.1]{Ngo-lemme}.

\begin{theorem} \label{thm:support}
	Let $f \colon X \to B$ and $p \colon G \to B$ be as above. Let $Z \subseteq B$ be
	one of the supports in the decomposition theorem for $f \colon X \to B$. Then the
	following is true:
	\begin{enumerate}[label=(\alph*)]
		\item One has $\codim_B Z = n-r$ for some $r \in \{0, 1, \dotsc, n\}$, and
			\[
				\bigoplus_{j \in \ZZ} P_{j,Z} = \bigoplus_{j=-r}^{r} P_{j,Z}.
			\]
		\item There is a Zariski-dense open  subset $U \subseteq Z$ such that
			\[
				P_{j,Z} \cong \IC_Z \Bigl( \locC_Z \tensor \bigwedge^{r+j} \locT_Z \Bigr),
			\]
			for all $j \in \ZZ$, where $\locC_Z$ and $\locT_Z$ are local systems on $U$.
		\item For each $b \in U$, one has $\dim T_b = r$ and a canonical isomorphism
		\[
			(\locT_Z)_b \cong H^1(T_b).
		\]
		\item The local system $\locC_Z$ has finite monodromy, and $\IC_Z(\locC_Z) \cong j_* \locC_Z[\dim Z]$,
			where $j \colon U \to B$ denotes the inclusion. Moreover, the constructible sheaf
			$j_* \locC_Z$ is a direct summand in $R^{2n} \fl \RR_X$.
	\end{enumerate}
\end{theorem}

By proper base change, the stalk of $R^{2n} \fl \RR_X$ at a
point $b \in B$ is the $\RR$-vector space $H^{2n}(X_b)$; its dimension is the number
of irreducible components of the fiber $X_b$. Because of the result in (d), the
presence of a nontrivial support $Z \neq B$ forces the fibers over points in $Z$ to have
more than one irreducible component; so if all the fibers $X_b$ are \emph{irreducible},
then $Z = B$ must be the only support.

\begin{corollary}
	If all fibers of $f \colon X \to B$ are irreducible, then the only
	support in the decomposition theorem is $Z=B$. In that case, $\derR \fl
	\RR_X[\dim X]$ is determined (up to isomorphism) by its restriction to the smooth
	locus of $f$.
\end{corollary}

\subsection{Deligne's splitting construction}
\label{par:splitting}

In this section, we review a linear algebra result from a letter that Deligne
wrote to Cattani and Kaplan in 1993; for more details, see for example \cite{Schwarz}.
Consider a finite-dimensional vector space $V$ (over a field such as $\QQ$ or $\RR$)
with an increasing filtration $W =
W_{\bullet} V$, and a nilpotent operator $N \in \End(V)$ that preserves
the filtration, in the sense that $N(W_{\ell}) \subseteq W_{\ell}$. There
is at most one increasing filtration $M = M_{\bullet} V$ such that $N(M_k)
\subseteq M_{k-2}$ and such that
\[
	N^k \colon \gr_{\ell+k}^M \gr_{\ell}^W \to \gr_{\ell-k}^M \gr_{\ell}^W
\]
is an isomorphism for every $k \geq 1$ and every $\ell \in \ZZ$. This filtration, if
it exists, is called the \define{relative weight filtration for $(N,W)$} and is
denoted by the symbol $M(N,W)$. We make the following two assumptions:
\begin{enumerate}
	\item The relative weight filtration $M = M(N, W)$ exists.
	\item There is a splitting $H \in \End(V)$ for the filtration $M$, with the
		property that $(\ad H)(N) = [H, N] = -2N$.
\end{enumerate}
To say that $H$ is a splitting for $M$ means that $H$ is semisimple with integer eigenvalues,
and  that $M_k = E_k(H) \oplus M_{k-1}$ for every $k \in \ZZ$. The condition $(\ad H)(N)
= -2N$ means that $N (E_k(H)) \subseteq E_{k-2}(H)$ for all $k$; in other words, $N$ has
``weight'' $-2$ with respect to the semisimple operator $H$.

Now suppose that we have a splitting $H'$ for the filtration $W$ that
commutes with $H$; again, this means that $W_{\ell} = E_{\ell}(H') \oplus W_{\ell-1}$
for every $\ell \in \ZZ$. By decomposing $V$ into simultaneous eigenspaces, we get an
isomorphism
\begin{equation}\label{eq: simult e}
	V = \bigoplus_{k,\ell} E_k(H) \cap E_{\ell}(H') \cong
	\bigoplus_{k,\ell} \gr_k^M \gr_{\ell}^W.
\end{equation}
We can decompose the nilpotent operator $N$ into eigenvectors with respect to the
action of $\ad H'$; because $N$ preserves the filtration $W$, this is going to look like
\begin{equation} \label{eq:components-N}
	N = N_0 + N_{-1} + N_{-2} + \dotsb,
\end{equation}
where $\ad H'(N_{\ell}) = [H',N_{\ell}] = \ell N_{\ell}$. Because $M$ is the relative weight
filtration, and $H$ is a splitting of $M$, it is easy to see from the eigenspace
decomposition above that the semisimple operator $H - H'$ and the nilpotent operator
$N_0$ together determine a representation of the Lie algebra $\sltwo$ on the vector space $V$.
In fact, this is just the direct sum of the $\sltwo$-representations induced by $N$
on each subquotient $\gr_{\ell}^W$. Deligne's result
\cite[Prop.~3.5]{Deligne-decompositions} is the following.

\begin{lemma}[Deligne] \label{lem:Deligne}
	There is a \emph{unique} choice of splitting $H'$ for the filtration $W$
	such that $[H, H'] = 0$ and such that the components $N_{\ell}$ with $\ell \leq -1$ are
	primitive with respect to the induced $\sltwo$-representation on $\End(V)$.
\end{lemma}

Note that the component $N_{\ell}$ has weight $-(\ell+2)$ in the
$\sltwo$-representation on $\End(V)$ induced by $\ad(H-H')$ and $\ad N_0$, because
\[
	\ad(H-H')(N_{\ell}) = -(2+\ell) N_{\ell}.
\]
The condition that $N_{\ell}$ is primitive therefore means that $(\ad
N_0)^{-(\ell+1)} N_{\ell} = 0$; in particular, $N_{-1} = 0$. The splitting is
``minimal'', in the sense that the Lefschetz decomposition of $N \in \End(V)$ is as
small as possible. The construction of $H'$ is elementary: choose an arbitrary
splitting that commutes with $H$, and adjust it in finitely many steps in order to
remove all non-primitive components of $N$.

We also need the following fact about Deligne's splitting.
\begin{lemma} \label{lem:kerH'}
	Suppose that an endomorphism $f \in \End(V)$ preserves the
	filtration $W$ and satisfies $\ad H(f) = f$ and $(\ad N)^2 f = 0$. Then $\ad
	H'(f) = 0$, which means that $f$ commutes with Deligne's splitting.
\end{lemma}

\begin{proof}
	The decomposition of $N$ into eigenvectors for $\ad H'$ in \eqref{eq:components-N}
	is of course also the decomposition into eigenvectors for $\ad(H-H')$. To make the
	proof easier to follow, let's index things by their weights in the
	$\sltwo$-representation, so
	\[
		N = N_{-2} + N_0 + N_1 + \dotsb,
	\]
	where now $\ad (H-H') N_{\ell} = \ell N_{\ell}$. With this notation, the
	$\sltwo$-representation is determined by $N_{-2}$ and $H-H'$, and we have
	$(\ad N_{-2})^{\ell+1} N_{\ell} = 0$. Since $\ad H(f) = f$, the decomposition of
	$f$ into eigenvectors for $\ad H'$ is also the decomposition into eigenvectors for
	$\ad (H-H')$. The eigenvalues with respect to $\ad H'$ are all $\leq 0$, because
	$f$ preserves the filtration $W$. In terms of $\ad(H-H')$, this gives
	\[
		f = f_1 + f_2 + f_3 + \dotsb,
	\]
	where again $\ad(H-H') f_{\ell} = \ell f_{\ell}$.

	We are going to prove by
	induction that $f_{\ell} = 0$ for $\ell \geq 2$. Since $\ad H'(f_1) = 0$, this
	will give the desired result. Consider the identity
	\[
		(\ad N_{-2} + \ad N_0 + \dotsb)^2 (f_1 + f_2 + f_3 + \dotsb) = (\ad N)^2 f = 0.
	\]
	The smallest weight that occurs is $-3$, and as there is only one term with that
	weight, we get $(\ad N_{-2})^2 f_1 = 0$. This means that $f_1$ is primitive of
	weight $1$. From the terms of weight $-2$, we deduce that
	\[
		(\ad N_{-2})^2 f_2 = 0,
	\]
	and because $f_2$ has weight $2$, this forces $f_2 = 0$. Let's now assume by
	induction that $f_2 = \dotsb = f_{\ell} = 0$ for some $\ell \geq 2$. Then we
	get
	\[
		(\ad N_{-2} + \ad N_0 + \dotsb)^2 (f_1 + f_{\ell+1} + f_{\ell+2} + \dotsb) = 0,
	\]
	and the terms of weight $\ell-3$ are
	\[
		(\ad N_{-2})^2 f_{\ell+1} + \sum_{i+j=\ell-4} (\ad N_i \circ \ad N_j)f_1 = 0.
	\]
	If we apply $(\ad N_{-2})^{\ell-1}$ to this, it becomes
	\[
		(\ad N_{-2})^{\ell+1} f_{\ell+1} = - \sum_{i+j=\ell-4} (\ad N_{-2})^{\ell-1}
		\circ (\ad N_i \circ \ad N_j)f_1.
	\]
	The right-hand side is zero because $\ad N_{-2}$ is a derivation and $(\ad
	N_{-2})^{i+1} N_i = 0$ and $(\ad N_{-2})^2 f_1 = 0$. This gives $(\ad
	N_{-2})^{\ell+1} f_{\ell+1} = 0$, from which we again deduce that $f_{\ell+1} =
	0$. Therefore $f = f_1$ lies in the kernel of $\ad H'$.
\end{proof}

\subsection{Cohomology of the fibers}

In this section, we look at what the decomposition theorem says about the cohomology
of the fibers. Consider a holomorphic mapping $f \colon X \to Y$ from a K\"ahler
manifold $X$ to a complex space $Y$. We assume that $f$ is proper and
surjective, and that the fibers $X_y = f^{-1}(y)$ are connected; each fiber is
then a (possibly nonreduced) connected compact K\"ahler space. The decomposition theorem for
the constant sheaf $\RR_X$ takes the form
\begin{equation} \label{eq:decomposition-general}
	\derR \fl \RR_X[\dim X] \cong \bigoplus_{j \in \ZZ} P_j[-j],
\end{equation}
where each $P_j$ is a perverse sheaf with coefficients in $\RR$.
Once we fix a K\"ahler class $\omega \in H^2(X)$, there is
an especially nice decomposition -- Deligne \cite[Rmk.~1.8]{Deligne-Lefschetz}
calls it ``more beautiful than
the others'' -- that is constructed using the representation theory of $\sltwo$. Since
it is important in what follows, we briefly summarize the construction. We
can think of $\omega$ as a morphism from $\RR_X$ to $\RR_X[2]$ in the derived
category $\Dbc(X, \RR)$. We get an induced morphism
\[
	\omega \colon \derR \fl \RR_X[\dim X] \to \derR \fl \RR_X[\dim X+2],
\]
and with respect to any decomposition as in \eqref{eq:decomposition-general}, this morphism
decomposes as
\[
	\omega = \omega_2 + \omega_1 + \omega_0 + \dotsb,
\]
where $\omega_k$ maps each summand $P_j$ into the summand $P_{j+k}[2-k]$. According to
the relative Hard Lefschetz theorem,
\[
	\omega_2^j \colon P_{-j} \to P_j
\]
is an isomorphism for every $j \geq 0$, and so we have a representation of the Lie
algebra $\sltwo$ on the direct sum of all the $P_j$. Deligne
\cite[Prop.~3.5]{Deligne-decompositions} proves that there is a \emph{unique} choice
of decomposition (in the derived category) such that $\omega_1 = 0$
and $\omega_k$ is primitive for $k \leq 0$. Here ``primitive'' means concretely that
$(\ad \omega_2)^{-k+1} \omega_k = 0$. This decomposition is ``minimal'' in the sense
that the Lefschetz decompositions of the other components $\omega_k$ are as small as
they can possibly be. (The proof is again elementary: start from an arbitrary
decomposition, and adjust it in finitely many steps to remove any non-primitive
components in the Lefschetz decomposition.) This will be our choice for the decomposition in
the decomposition theorem throughout the paper.

Let's see what happens when we restrict the decomposition in
\eqref{eq:decomposition-general} to a point $y \in Y$. Let $i \colon \pt \into Y$ be
the inclusion of the point. By proper base change,
\begin{equation} \label{eq:Deligne-fiber}
	H^{\dim X+k}(X_y) \cong H^k \iu \bigl( \derR \fl \RR_X[\dim X] \bigr)
	\cong \bigoplus_{j \in \ZZ} H^{k-j} \iu P_j,
\end{equation}
and so the cohomology of the fiber $X_y$ inherits a canonical decomposition.
Can we describe this decomposition in a way that
uses, as much as possible, only objects on the fiber?

It turns out that we can, and that all we need is the perverse filtration on the
cohomology of $X_y$, and the image of the K\"ahler class in $H^2(X_y)$. The crucial
point, which is similar to \cite[Prop.~5.2.4]{deCataldo+Migliorini}, is that one can
interpret the existence of the decomposition in terms of (relative) weight
filtrations for certain nilpotent operators. Before we describe this, let's quickly review
the definition of the perverse filtration. Let $\tau_{\leq \ell}$ denote the
truncation functors with respect to the perverse $t$-structure on the derived
category $\Dbc(Y, \RR)$.

\begin{definition} \label{def:P}
From the morphism $\tau_{\leq \ell} \derR \fl \RR_X[\dim X] \to \derR \fl \RR_X[\dim
X]$, we get
\[
	H^k \iu \bigl( \tau_{\leq \ell} \derR \fl \RR_X[\dim X] \bigr)
	\to H^k \iu \bigl( \derR \fl \RR_X[\dim X] \bigr)
	\cong H^{\dim X+k}(X_y).
\]
This is injective (by the decomposition theorem), and therefore defines a subspace
\[
	P_{\ell} H^{\dim X+k}(X_y) \subseteq H^{\dim X + k}(X_y).
\]
The resulting filtration $P_{\bullet} H^{\dim X+k}(X_y)$ is called the
\emph{perverse filtration}.
\end{definition}

The perverse filtration is completely intrinsic to the morphism $f \colon X
\to Y$, whereas the decomposition of $H^{\ast}(X_y)$ depends on the specific choice
we made in \eqref{eq:decomposition-general}. By construction, we have
\[
	\gr_{\ell}^P H^{\dim X+k}(X_y) \cong H^{k-\ell} \iu P_{\ell}.
\]
The restriction of $\omega \in H^2(X)$ gives us a K\"ahler class
$\omega_y \in H^2(X_y)$. Let's denote by $N \colon H^{\ast}(X_y) \to H^{\ast+2}(X_y)$ the
nilpotent operator obtained by multiplying with this class; for example, if we apply $N$
to the element $1 \in H^0(X_y)$, we get $\omega_y = N(1) \in H^2(X_y)$.
The perverse filtration is not preserved by $N$, so let's adjust the
degrees and define the \emph{shifted perverse filtration}
\[
	\SP_{\ell} H^k(X_y) := P_{\ell+k} H^k(X_y).
\]
This is an increasing filtration of $H^{\ast}(X_y)$ that is now preserved by $N$.
We usually denote it by the symbol $\SP$ alone. Let's also define the filtration
\begin{equation}\label{eq:MM}
	M_k H^{\ast}(X_y) = \bigoplus_{j \geq -k} H^j(X_y),
\end{equation}
which is basically the filtration by cohomological degree, but changed so
that $N$ maps $M_k$ into $M_{k-2}$. The reason for this definition is the
following lemma.

\begin{lemma}
	$M$ is the relative weight filtration of $N$ with respect to $\SP$.
\end{lemma}

\begin{proof}
	The relative weight filtration, when it exists, is uniquely characterized by two
	properties: $N(M_k) \subseteq M_{k-2}$, and
	\[
		N^k \colon \gr_{\ell+k}^M \gr_{\ell}^{\SP} \to \gr_{\ell-k}^M \gr_{\ell}^{\SP}
	\]
	is an isomorphism for $k \geq 1$. The first property is obvious in our case. As
	for the second one, we have
	\[
		\gr_{\ell+k}^M \gr_{\ell}^{\SP}
		\cong \gr_{\ell}^{\SP} \! H^{-(k+\ell)}(X_y)
		\cong \gr_{-k}^P H^{-(k+\ell)}(X_y)
		\cong H^{-(\dim X + \ell)} \iu P_{-k}.
	\]
	and so the statement we want follows from the relative Hard Lefschetz theorem.
\end{proof}

The relationship with the decomposition theorem is that the decomposition
\begin{equation} \label{eq:decomposition-general-fiber}
	H^{\ast}(X_y) \cong \bigoplus_{j,k \in \ZZ} H^k \iu P_j
\end{equation}
induced by \eqref{eq:decomposition-general} is exactly the one corresponding to
Deligne's splitting $H'$ from \Cref{lem:Deligne}, with $V = H^{\ast}(X_y)$, $W = P^{+}$, and
$N$ multiplication by $\omega_y$. To see why, let's go back to the definition of the shifted
perverse filtration $\SP$. It gives
\[
	\gr_{\ell}^{\SP} \! H^k(X_y) = \gr_{\ell+k}^P H^k(X_y)
	\cong H^{-(\dim X+\ell)} \iu P_{k+\ell}.
\]
We can therefore define a splitting $H'$ for the filtration $\SP$ by letting $H'$ act
on the summand $H^k \iu P_j$ in the decomposition \eqref{eq:decomposition-general-fiber} as
multiplication by $-(\dim X+k)$. The difference $H-H'$ then acts on it
as multiplication by $-j$.

Let's consider the decomposition $\omega = \omega_2 + \omega_0 + \omega_{-1} +
\dotsb$ for the action of the K\"ahler form on \eqref{eq:decomposition-general}. Here
$\omega_{\ell}$ maps the summand $P_j$ in the decomposition into the summand
$P_{j+\ell}[2-\ell]$, and is primitive in the sense that $(\ad \omega_2)^{-\ell+1}
\omega_{\ell} = 0$. On stalks, $\omega_{\ell}$ therefore induces morphisms
\[
	H^k \iu P_j \to H^{k+2-\ell} \iu P_{j+\ell}.
\]
Because $H'$ acts on the source as $-(\dim X+k)$, and on the target as $-(\dim X+k+2-\ell)$, we
see that $\ad H'(\omega_{\ell}) = (\ell-2) \omega_{\ell}$. So if we set $N_{\ell} =
\omega_{\ell+2}$, then the decomposition
\[
	\omega = \omega_2 + \omega_0 + \omega_{-1} + \dotsb
\]
turns into a decomposition of the nilpotent operator $N$ of the form
\[
	N = N_0 + N_{-2} + N_{-3} + \dotsb,
\]
in which the terms $N_{\ell}$ with $\ell \leq -2$ are primitive because
\[
	(\ad N_0)^{-(\ell+1)} N_{\ell} = (\ad \omega_2)^{-(\ell+1)} \omega_{\ell+2} = 0.
\]
This shows that the splitting $H'$ is exactly the one in \Cref{lem:Deligne}. What this
means is that the decomposition in \eqref{eq:decomposition-general-fiber} is the common
eigenspace decomposition \eqref{eq: simult e} of the two commuting semisimple operators $H$ and $H'$. The
uniqueness of Deligne's splitting therefore tells us that this decomposition is
uniquely determined by the shifted perverse filtration $\SP$ and the
K\"ahler class $\omega_y \in H^2(X_y)$.

\begin{remark}
Throughout the paper, we use the decomposition in \eqref{eq:Deligne-fiber}
to identify each $\RR$-vector space $H^{k-j} \iu P_j$ with a subspace of the
cohomology $H^{\dim X +k}(X_y)$ of the fiber. As we explained above, this identification
is uniquely determined by the K\"ahler class $\omega_y \in H^2(X_y)$ and by the shifted
perverse filtration on $H^{\ast}(X_y)$.
\end{remark}

\subsection{Proof of the Support Theorem}
\label{par:support-theorem-proof}

We prove the Support Theorem in the precise form  given by \Cref{thm:support}.
Let's now consider a holomorphic mapping $f \colon X \to B$, with the assumptions
made at the beginning of this chapter. We fix a K\"ahler class $\omega \in H^2(X)$
and use Deligne's decomposition as our choice of decomposition in the decomposition theorem
\eqref{eq:decomposition}. As explained in \eqref{eq:Deligne-fiber}, on each fiber $X_b$, we
have an induced decomposition
\begin{equation} \label{eq:decomposition-fiber}
	H^{\ast}(X_b) \cong \bigoplus_{j,k \in \ZZ} H^k \iu P_j,
\end{equation}
where $i \colon \pt \into B$ is the inclusion of the point $b \in B$. This decomposition
is uniquely determined by the K\"ahler class $\omega_b \in H^2(X_b)$ and by the shifted perverse filtration
on $H^{\ast}(X_b)$, via Deligne's splitting construction in \Cref{lem:Deligne}. We may therefore
consider each $H^{k} \iu P_j$ as a subspace of the cohomology $H^{2n+k+j}(X_b)$.

As the proof of the Support Theorem involves only $G^\circ$, we can assume without
loss of generality that each group $G_b$ is connected (by replacing $p \colon G \to B$ by the family $G^\circ$ of
neutral components as in \Cref{prop:neutral-component}).

We showed in the previous chapter that the cohomology $H^{\ast}(X_b)$ of each fiber
is a free comodule over the Hopf algebra $H^{\ast}(T_b)$, where $T_b$ is the maximal
compact quotient of the commutative connected meromorphic group $G_b$. We are going to use the
special properties of Deligne's splitting to prove that for every support $Z \subseteq B$
and every $k \in \ZZ$, the direct sum
\[
	\bigoplus_{j \in \ZZ} H^k \iu P_{j,Z}
\]
is itself a free comodule over $H^{\ast}(T_b)$. Together with the $\delta$-regularity
inequality, this will then easily imply the support theorem.

\begin{remark}
In \S7.6 of his paper, Ng\^{o} writes (to paraphrase) that it is tempting to look for a
decomposition \eqref{eq:decomposition} that is compatible with the comodule
structure, but that such a decomposition does not seem to exist. Fortunately,  Deligne's decomposition
is compatible with the comodule structure. This observation
makes our proof quite a bit simpler than Ng\^{o}'s.
\end{remark}

The group $G_b$ is abelian, and so $H^{\ast}(G_b)$ is the wedge algebra on $H^1(G_b)$
(by \Cref{ex:abelian}). Recall from the discussion in \Cref{sec:comodules} that
the comodule structure
\[
	\delta \colon H^{\ast}(X_b) \to H^{\ast}(G_b) \tensor H^{\ast}(X_b)
\]
is therefore determined by the $d$ anticommuting linear operators
\[
	\delta_1, \dotsc, \delta_d \colon H^{\ast}(X_b) \to H^{\ast-1}(X_b),
\]
where $e_1, \dotsc, e_d \in H^1(G_b)$ is a basis.
Let's try to understand how the individual $\delta_i$ interact
with Deligne's splitting $H'$. Recall that in the decomposition
\[
	H^{\ast}(X_b) \cong \bigoplus_{j,k \in \ZZ} H^k \iu P_{j,Z},
\]
the splitting $H$ acts on the term $H^k \iu P_{j,Z}$ as multiplication by $-(\dim X + j +k)$,
whereas Deligne's splitting $H'$ is multiplication by $-(\dim X + k)$.
In particular, $H$ acts on $H^k(X_b)$ as multiplication by $-k$,
and so $\ad H(\delta_i) = \delta_i$.

\begin{lemma}
	We have $(\ad N)^2 \delta_i = 0$.
\end{lemma}

\begin{proof}
	Recall that the nilpotent operator $N$ is multiplication by the K\"ahler class $\omega_b \in H^2(X_b)$.
	For any cohomology class $\alpha \in H^k(X_b)$, we have
	\[
		\delta(N \alpha) = \delta(\omega_b \cup \alpha)
		= \delta(\omega_b) \cup \delta(\alpha)
	\]
	because $\delta$ is compatible with cup product. By looking at the component in
	$H^1(G_b) \tensor H^{k+1}(X_b)$, we obtain the identity
	\[
		\sum_{i=1}^{d} e_i \tensor \delta_i(N \alpha) =
		\sum_{i=1}^{d} e_i \tensor \omega_b \cup \delta_i(\alpha)
		+ \sum_{i=1}^{d} e_i \tensor \delta_i(\omega_b) \cup \alpha,
	\]
	which tells us that $\delta_i(N \alpha) - N \delta_i(\alpha) = \delta_i(\omega_b)
	\cup \alpha$. In other words, the operator $\ad N(\delta_i) = [N, \delta_i]$ is
	cup product with the cohomology class $-\delta_i(\omega_b) \in H^1(X_b)$.
	Because this commutes with $N$, we get the desired result.
\end{proof}

We also need the following somewhat technical result; roughly speaking, it says that
the comodule structure on $H^{\ast}(X_b)$ can be ``spread out'' to nearby fibers.

\begin{proposition} \label{prop:shifted-perverse}
	At each point $b \in B$, the shifted perverse filtration on $H^{\ast}(X_b)$ is a
	filtration by comodules over $H^{\ast}(G_b)$. Moreover, the comodule structure on each
	subquotient $\gr_{\ell}^{\SP} \!  H^{\ast}(X_b)$ is compatible with the decomposition
	by strict support.
\end{proposition}

The comodule structure is determined by the operators $\delta_1, \dotsc,
\delta_d$, and so the content of the proposition is that
\[
	\delta_i \bigl( \SP_{\ell} H^k(X_b) \bigr) \subseteq \SP_{\ell} H^{k-1}(X_b).
\]
In other words, the shifted perverse filtration is preserved by each $\delta_i$.
The proof takes a bit of work, so let's postpone it until \Cref{subs:pfprop comodP} and first
deduce the support theorem from the results we have so far.

Let's start by proving that each $\delta_i$ commutes with Deligne's splitting $H'$.
Because $(\ad N)^2 \delta_i = 0$ and $\ad H(\delta_i) = \delta_i$, we can apply
\Cref{lem:kerH'} to the individual operators $\delta_i$. The conclusion is
that $[H', \delta_i] = 0$, which means that $\delta_i$ preserves
the eigenspaces of $H'$ in the decomposition \eqref{eq:decomposition-fiber}. This
says concretely that, for each fixed value of $k \in \ZZ$, the graded subspace
\[
	\bigoplus_{j=-n}^n H^k \iu P_j \subseteq H^{\ast}(X_b)
\]
is a direct summand as a comodule over the Hopf algebra $H^{\ast}(G_b)$; here $H^k \iu P_j$
is a subspace of $H^{\dim X + j + k}(X_b)$. This gives us an
isomorphism of comodules
\[
	H^{\ast}(X_b) \cong \bigoplus_{\ell \in \ZZ} \gr_{\ell}^{\SP} \! H^{\ast}(X_b).
\]
According to the second half of \Cref{prop:shifted-perverse}, the comodule
structure is therefore also compatible with the decomposition by strict support. It
follows that for each possible support $Z \subseteq B$, the subspace
\[
	\bigoplus_{j=-n}^n H^k \iu P_{j,Z} \subseteq H^{\ast}(X_b)
\]
is a direct summand as a comodule over the $H^{\ast}(G_b)$. Because the entire
comodule is free over the smaller Hopf algebra $H^{\ast}(T_b)$ (by \Cref{thm:freeness}),
each of these summands is free as well. We summarize the conclusion in the following lemma.

\begin{lemma} \label{lem:comodule-decomposition}
For every $k \in \ZZ$ and every support $Z \subseteq B$, the graded vector space
\[
	\bigoplus_{j=-n}^n H^k \iu P_{j, Z}
\]
is a comodule over the Hopf algebra $H^{\ast}(G_b)$. It is free as a
comodule over the quotient Hopf algebra $H^{\ast}(T_b)$.
\end{lemma}

We are now ready to prove the support theorem, under the assumptions made at
the beginning of this chapter.

\begin{proof}[Proof of \Cref{thm:support}]
Let $Z \subseteq B$ be one of the supports in the decomposition theorem for
$f \colon X \to B$. Define the integer $r = n - \codim_B Z$, and let $b \in Z$ be a
general point. The $\delta$-regularity condition tells us that the compact complex
torus $T_b$ has dimension at least
\[
	r = n - \dim B + \dim Z = 2n - \dim X + \dim Z.
\]
On a dense open subset $U \subseteq Z$, each $P_{j,Z}$ has the form $\locL_{j,Z}[\dim Z]$,
where $\locL_{j,Z}$ is a local system.  If $i \colon \pt \into B$  denotes the inclusion
of the point, we therefore get
\[
	\bigoplus_{j=-n}^n H^{-\dim Z} \iu P_{j, Z}
	\cong \bigoplus_{j=-n}^n (\locL_{j,Z})_b.
\]
The $j$-th summand contributes to $H^{j + \dim X-\dim Z}(X_b)$, and can therefore be
nonzero only when $j + \dim X - \dim Z \leq 2n$, which translates into $j \leq r$.
Combining this fact with the relative Hard Lefschetz theorem, we can rewrite the
sum as
\[
	\bigoplus_{j=-r}^r H^{-\dim Z} \iu P_{j, Z}
	\cong \bigoplus_{j=-r}^r (\locL_{j,Z})_b.
\]
But we know that this is a free comodule over the Hopf algebra $H^{\ast}(T_b)$, and
because $\dim T_b \geq r$, this is only possible if $\locL_{r,Z} \neq 0$ and $\dim
T_b = r$. This proves (a).

It only takes a little bit of book-keeping to deduce the other three assertions in the
support theorem from the freeness theorem. We already know that
\[
	\bigoplus_{j=-r}^r P_{j,Z}[-j]
\]
is a direct summand in $\derR \fl \RR_X[\dim X]$. On cohomology sheaves, $\shH^i P_{r,Z}$
is therefore a summand of $R^{i+r+\dim X} \fl \RR_X$, and since $f$ is equidimensional,
this vanishes once
\[
	2n < i + r + \dim X = i + 2n + \dim Z
\]
or $-\dim Z < i$. But $P_{r,Z}$ is a perverse sheaf supported on $Z$, and so
$\shH^i P_{r,Z} = 0$ except when $-\dim Z \leq i \leq 0$; the conclusion is that
\[
	P_{r,Z}[-\dim Z] \cong \shH^{-\dim Z} P_{r,Z}
\]
is a single constructible sheaf that is isomorphic to a direct summand of $R^{2n} \fl
\RR_X$. By proper base change, the stalk
$(\locL_{r,Z})_b$ is a summand of $H^{2n}(X_b)$ for every $b \in U$, and so the monodromy group of
the local system is a subgroup of the group of permutations of the irreducible components
of the fiber, and is therefore finite. We now define $\locC_Z = \locL_{-r,Z}$.
Since $\locL_{r,Z} \cong \locL_{-r,Z}$ by the
relative Hard Lefschetz theorem, $\locC_Z$ is a local system with
finite monodromy, and so (d) is proved.

Now let $b \in U$ be an arbitrary point. For dimension reasons, the subspace of
coinvariants in the free $H^{\ast}(T_b)$-comodule
\[
	M_{Z, b} = \bigoplus_{j=-r}^r (\locL_{j,Z})_b
\]
must be exactly $(\locC_Z)_b = (\locL_{-r,Z})_b$, and therefore
\[
	(\locL_{j,Z})_b \cong    (\locC_Z)_b \tensor H^{r+j}(T_b)
	\cong    (\locC_Z)_b   \tensor \bigwedge^{r+j} H^1(T_b) .
\]
Now we observe that the comodule structure
\[
	M_{Z,b} \to H^{\ast}(G_b) \tensor M_{Z,b}
\]
actually factors through the Hopf subalgebra $H^{\ast}(T_b) \subseteq H^{\ast}(G_b)$.
The reason is that $M_{Z,b}^j \subseteq H^{2n+r-j}(X_b)$ has a pure Hodge structure
of weight $2n+r-j$ by \Cref{lem:MHS} below; for weight reasons, the morphisms
\[
	M_{Z,b}^j \to H^1(G_b) \tensor M_{Z,b}^{j-1}
\]
therefore factors through $H^1(T_b)$. In particular, the image of the induced morphism
\[
	M_{Z,b}^j \tensor \Hom_{\RR} \bigl( M_{Z,b}^{j-1}, \RR \bigr) \to H^1(G_b)
\]
is exactly the subspace $H^1(T_b) \subseteq H^1(G_b)$.

To construct the local system $\locT_Z$, we appeal to two results from the next section.
Let $K_G = \derR p_! \RR_G[2n]$, which is a constructible complex on $B$.
According to \eqref{eq:K-action} below, we have a morphism
\[
	K_G \tensor \derR \fl \RR_X[\dim X] \to \derR \fl \RR_X[\dim X]
\]
in the derived category $\Dbc(B, \RR)$; by combining it with the splitting in
\Cref{lem:K} and the decomposition by strict support, we obtain morphisms
\[
	P_{j,Z} \tensor \bigwedge^{r+j} \shH^{-1} K_G \to P_{-r,Z}
\]
for $-r \leq j \leq r$. Upon restricting to $U \subseteq B$, we get
\[
	\locL_{j,Z} \tensor \bigwedge^{r+j} \shH^{-1} K_G \vert_U \to \locL_{-r,Z}
\]
where everything is now a local system. According to the discussion above,
the image of the morphism
\[
	\shH^{-1} K_G \vert_U \to \shHom_{\RR} \bigl( \locL_{-r+1,Z}, \locL_{-r,Z} \bigr)
\]
is then a local system whose fiber at a point $b \in U$ is isomorphic to $H_1(T_b)$.
Let $\locT_Z$ be the dual local system, with fibers $H^1(T_b)$; then from the above,
we get
\[
	\locL_{j,Z} \cong \locC_Z \tensor \bigwedge^{r+j} \locT_Z,
\]
and so (b) and (c) hold. This finishes the proof of the support theorem, modulo the
material in the next section (especially the proof of  \Cref{prop:shifted-perverse}).
\end{proof}

We used the following technical lemma during the proof.

\begin{lemma} \label{lem:MHS}
At each point $b \in U$, the subspace $(\locL_{j,Z})_b \subseteq H^{2n+j-r}(X_b)$
is a mixed Hodge substructure that is pure of weight $2n+j-r$.
\end{lemma}
\begin{proof}
By \Cref{lem:HG}, we have an isomorphism of Hopf algebras
\[
	H^{\ast}(G_b) \cong H^{\ast}(L_b) \tensor H^{\ast}(T_b)
\]
that respects the mixed Hodge structures on both sides. This makes the cohomology $H^{\ast}(X_b)$
into a comodule over $H^{\ast}(T_b)$, and the comodule structure
\[
	\delta_T \colon H^{\ast}(X_b) \to H^{\ast}(T_b) \tensor H^{\ast}(X_b)
\]
is again a morphism of mixed Hodge structures. We use this fact and the freeness theorem
to put a Hodge structure on each subspace $(\locL_{j,Z})_b \subseteq H^{2n+j-r}(X_b)$.

Since $\dim X_b = n$, the mixed Hodge structure on the top cohomology group
$H^{2n}(X_b)$ is pure of type $(n,n)$; its dimension is equal to the number of
irreducible components of $X_b$. Consequently, the subspace
\[
	(\locL_{r,Z})_b \subseteq  H^{2n}(X_b)
\]
inherits a pure Hodge structure of type $(n,n)$. Recall that $U \subseteq Z$ is the Zariski-open
subset over which $P_{j,Z}$ is a local system, and that $i \colon \pt \to B$
denotes the inclusion of the point $b \in U$.  By \Cref{lem:comodule-decomposition}, the
graded subspace
\[
	\bigoplus_{j=-r}^r (\locL_{j,Z})_b
	= \bigoplus_{j=-r}^r H^{-\dim Z} \iu P_{j,Z} \subseteq H^{\ast}(X_b)
\]
is a comodule over $H^{\ast}(G_b)$ that is free over $H^{\ast}(T_b)$. Therefore
the diagram
\begin{equation} \label{eq:MHS-1}
\begin{tikzcd}
	(\locL_{r,Z})_b \rar{\delta_T} \dar[hook] & H^{r-j}(T_b) \tensor
	(\locL_{j,Z})_b \dar[hook] \\
	H^{2n}(X_b) \rar{\delta_T} & H^{r-j}(T_b) \tensor H^{2n+j-r}(X_b)
\end{tikzcd}
\end{equation}
is commutative; the bottom arrow is a morphism of mixed Hodge structures. We now
apply the freeness theorem and \Cref{lem:comodule-free}; the conclusion is that if we
move the factor $H^{r-j}(T_b)$ to the left-hand side, then in the diagram
\begin{equation} \label{eq:MHS-2}
\begin{tikzcd}
	\Hom_{\RR} \bigl( H^{r-j}(T_b), \RR \bigr) \tensor (\locL_{r,Z})_b
	\rar{\cong} \dar[hook] & (\locL_{j,Z})_b \dar[hook] \\
	\Hom_{\RR} \bigl( H^{r-j}(T_b), \RR \bigr) \tensor
	H^{2n}(X_b) \rar & H^{2n+j-r}(X_b)
\end{tikzcd}
\end{equation}
the top arrow becomes an isomorphism, and the bottom arrow stays a morphism of mixed Hodge
structures. It follows that $(\locL_{j,Z})_b$ has a unique Hodge structure of weight $2n+j-r$
that makes the top arrow in \eqref{eq:MHS-2} into an isomorphism of Hodge structures,
and thus the top arrow in \eqref{eq:MHS-1} into a morphism of Hodge structures; because of the shape of
the two commutative diagrams, $(\locL_{j,Z})_b$ is automatically a mixed Hodge substructure of $H^{2n+j-r}(X_b)$.
\end{proof}

\subsection{Proof of Proposition~\ref*{prop:shifted-perverse}}\label{subs:pfprop comodP}

Now let's return to \Cref{prop:shifted-perverse}. It is not a pointwise statement -- because of
the presence of the perverse filtration, which is constructed locally on $B$ -- and so
we need the family of meromorphic groups $p \colon G \to B$. Recall the notation in
\eqref{eq:action} for the fiberwise action of $G$ on $X$. At any point $b \in B$, we
are going to argue that the comodule structure
\[
	\delta \colon H^{\ast}(X_b) \to H^{\ast}(G_b) \tensor H^{\ast}(X_b)
\]
comes from a morphism in the derived category
\[
	\derR \fl \RR_X \big\vert_U \to H^{\ast}(G_b) \tensor \derR \fl
	\RR_X \big\vert_U,
\]
where $U \subseteq B$ is a small open ball around the point $b \in B$. It takes a
certain amount of sheaf theory to get this morphism, but fortunately, everything we
need exists for holomorphic mappings between complex spaces. We use the comprehensive book
by Kashiwara and Schapira \cite{Kashiwara+Schapira} as our reference, with only
one change in notation: if $f \colon X \to Y$ is a morphism, Kashiwara and Schapira denote
the  sheaf-theoretic pullback by $f^{-1}$, but we prefer the symbol $\fu$.

Let's first try to understand what structure we have on the fiberwise cohomology of
$p \colon G \to B$. Here it is better to work with compactly supported cohomology,
because we can then apply proper base change and the projection formula. We therefore
consider the constructible complex
\[
	K_G = \derR p_! \RR_G[2n] \in \Dbc(B, \RR).
\]
Constructibility follows from the fact that $p \colon G \to B$ is a family of
meromorphic groups (which means that $p$ is the composition of a Zariski-open
embedding and a proper holomorphic mapping). By proper base change
\cite[Prop.~2.6.7]{Kashiwara+Schapira}, we have
\[
	\shH^{-i} K_{G, b}  \cong H_c^{2n-i}(G_b)
\]
for any point $b \in B$. Because $G_b$ is a complex manifold, Poincar\'e duality gives
\[
	H_c^{2n-i}(G_b) \cong \Hom \bigl( H^i(G_b), \RR \bigr) \cong H_i(G_b),
\]
and so we have an isomorphism
\[
	K_{G,b} \cong \bigoplus_{i \in \NN} H_i(G_b)[i]
\]
in the derived category of complexes of $\RR$-vector spaces. Because $G_b$ is
commutative and connected, we know from Hopf's theorem that $H_{\ast}(G_b)$ is
the wedge algebra on $H_1(G_b)$. In terms of the stalks of $K_G = \derR p_! \RR_G[2n]$,
this says that
\[
	K_{G, b} \cong \bigoplus_{i \in \NN} \left( \bigwedge\nolimits^i \shH^{-1} K_{G, b} \right)[i].
\]
The next step is to prove a sheaf version of this result, and deduce from
it that the complex $K_G$ splits in the derived category.

\begin{lemma} \label{lem:K}
	In the derived category $\Dbc(B, \RR)$, we have an isomorphism
	\[
		K_G \cong \bigoplus_{i \in \NN} \left( \bigwedge\nolimits^i \shH^{-1} K_G \right)[i].
	\]
	In particular, $K_G$ is isomorphic to a direct sum of shifts of constructible sheaves.
\end{lemma}

\begin{proof}
In the construction, we shall often use the following fact: if $f \colon X \to Y$ is a
holomorphic mapping between complex manifolds, then there is a natural morphism
\[
	\derR f_! \RR_X \to \RR_Y[2c],
\]
where $c = \dim Y - \dim X$; this is sometimes called the trace map.
Let's first consider the following commutative diagram:
\begin{tcd}
		G \rar{\Delta} \drar[bend right=20]{\id} & G \times_B G \dar{p_1}
			\rar{p_2} & G \dar{p} \\
							& G \rar{p} & B
\end{tcd}
From the diagonal embedding $\Delta$, we get a morphism $\derR \Delta_! \RR_G \to \RR_{G \times_B G}[2n]$, and
after applying the functor $\derR (p_1)_!$, this becomes a morphism
\[
	\RR_G \to \derR (p_1)_! \RR_{G \times_B G}[2n].
\]
By proper base change \cite[Prop.~2.6.7]{Kashiwara+Schapira}, we have
\[
	\derR (p_1)_! \RR_{G \times_B G}[2n] \cong \derR (p_1)_! \pu_2 \RR_G[2n]
	\cong \pu \derR p_! \RR_G[2n] = \pu K_G,
\]
and so we can rewrite our morphism as $\RR_G \to \pu K_G$. We can now apply the
functor $\derR p_!$ and use the projection formula to get a morphism
\begin{equation} \label{eq:comultiplication}
	K_G = \derR p_! \RR_G[2n] \to K_G \tensor \derR p_! \RR_G[2n] = K_G \tensor K_G.
\end{equation}
Presumably, this makes $K_G$ into a coalgebra object in the derived category
(after some compatibility checking), but we don't need much
information about it. The only thing
is that the comultiplication is co-commutative (because $\Delta$ is symmetric).

Now let's prove that $K_G$ splits. From the diagram
\begin{tcd}
		B \rar{e} \drar[bend right=25]{\id} & G \dar{p} \\
														& B
\end{tcd}
we get a morphism $\derR e_! \RR_B \to \RR_G[2n]$, and therefore two morphisms
\[
	\RR_B \to K_G = \derR p_! \RR_G[2n] \to \RR_B
\]
whose composition is the identity. Therefore $K_G$ splits as
\[
	K_G \cong \shH^0 K_G \oplus K_G' \cong \RR_B \oplus K_G',
\]
where $K_G'$ lives in cohomological degrees $\leq -1$. This gives us a morphism
\[
	K_G \to K_G' \to \shH^{-1} K_G'[1] \cong \shH^{-1} K_G[1]
\]
in the derived category, where $\shH^i K_G$ means the $i$-th cohomology sheaf of the
constructible complex $K_G$. If we apply the comultiplication in
\eqref{eq:comultiplication} to this morphism $i$ times, it produces for each $i \geq
1$ a morphism
\[
	K_G \to \shH^{-1} K_G[1] \tensor \dotsb \tensor \shH^{-1} K_G[1],
\]
and because \eqref{eq:comultiplication} is co-commutative, this factors through a
morphism
\[
	K_G \to \left( \bigwedge\nolimits^i \shH^{-1} K_G \right)[i].
\]
Let's combine all of them into one big morphism
\[
	K_G \to \bigoplus_{i \in \NN} \left( \bigwedge\nolimits^i \shH^{-1} K_G \right)[i].
\]
We know from Hopf's theorem that this induces an isomorphism on the stalk
at every point $b \in B$; therefore it must be an isomorphism on cohomology sheaves,
hence an isomorphism in the derived category.
\end{proof}

The (meromorphic) action $a \colon G \times_B X \to X$ induces a morphism
\begin{equation} \label{eq:K-action}
	K_G \tensor \derR \fl \RR_X \to \derR \fl \RR_X,
\end{equation}
which presumably turns the complex $\derR \fl \RR_X$ into a $K_G$-module object
in the derived category (after more compatibility checking). Again, we only need the
existence of the morphism. This follows by looking at the commutative diagram
\begin{tcd}
		X \dar{f} & G \times_B X \rar{a} \lar[swap]{p_2} \dar{p_1} & X \dar{f} \\
		B & G \rar{p} \lar[swap]{p} & B.
\end{tcd}
We start from the morphism $\derR a_! \RR_{G \times_B X}[2n] \to \RR_X$. Apply
$\derR f_!$ to get
\[
	\derR p_! \derR (p_1)_! \RR_{G \times_B X}[2n] \cong \derR f_! \derR a_!
	\RR_{G \times_B X}[2n] \to \derR f_! \RR_X.
\]
Using proper base change \cite[Prop.~2.6.7]{Kashiwara+Schapira}, we have
\[
	\derR (p_1)_! \RR_{G \times_B X} \cong \derR (p_1)_! \pu_2 \RR_X
	\cong \pu \derR f_! \RR_X,
\]
and the projection formula \cite[Prop.~2.6.6]{Kashiwara+Schapira} therefore implies that
\[
	\derR p_! \derR (p_1)_! \RR_{G \times_B X} \cong \derR p_! \RR_G \tensor \derR f_! \RR_X.
\]
Putting everything together and remembering that $K_G = \derR p_! \RR_G[2n]$ and
$\derR f_! \RR_X = \derR \fl \RR_X$ because $f$ is proper, we finally arrive at
the morphism in \eqref{eq:K-action}.

Now let $U$ be a sufficiently small open ball containing a given point $b \in B$.
Because the complex $K_G$ is constructible, we have \cite[Prop.~8.1.4]{Kashiwara+Schapira}
\[
	H^{\ast}(K_G \vert_U) \cong K_{G, b} \cong H_{\ast}(G_b),
\]
and this gives us a morphism
\[
	H_{\ast}(G_b) \tensor \RR_U \to K_G \big\vert_U
\]
in the derived category. After composing this with the morphism in
\eqref{eq:K-action}, we obtain
\[
	H_{\ast}(G_b) \tensor \derR \fl \RR_X \big\vert_U \to \derR \fl \RR_X \big\vert_U.
\]
Since $H_{\ast}(G_b)$ is dual to the cohomology $H^{\ast}(G_b)$, we can move the
graded $\RR$-vector space to the other side and obtain the desired morphism
\begin{equation} \label{eq:delta-U}
	\derR \fl \RR_X \big\vert_U \to H^{\ast}(G_b) \tensor \derR \fl \RR_X \big\vert_U.
\end{equation}
One then has to check that on stalks, meaning after applying the functor $\iu$, this
agrees with the morphism
\[
	\delta \colon H^{\ast}(X_b) \to H^{\ast}(G_b) \tensor H^{\ast}(X_b)
\]
coming from the action $a \colon G_b \times X_b \to X_b$. This follows (with some work)
from the functoriality of the isomorphisms that we used during the construction;
we give the proof in more generality in \Cref{lem:comodule-stalks}.

\begin{proof}[Proof of \Cref{prop:shifted-perverse}]
Let's apply some cohomological shifts and put the morphism from \eqref{eq:delta-U} into the form
\[
	\derR \fl \RR_X[\dim X] \big\vert_U
	\to H^{\ast}(G_b) \tensor \derR \fl \RR_X[\dim X] \big\vert_U.
\]
By composing with the projection $H^{\ast}(G_b) \to H^{\ast}(T_b)$ given by
\Cref{lem:HG}, and then with the projection from $H^{\ast}(T_b)$ to the summand
$H^1(T_b)$, we finally obtain
\begin{equation} \label{eq:morphism}
	\gamma \colon \derR \fl \RR_X[\dim X] \big\vert_U \to H^1(T_b)[-1] \tensor
	\derR \fl \RR_X[\dim X] \big\vert_U.
\end{equation}
If we now apply the truncation functor $\tau_{\leq \ell}$ for the perverse $t$-structure,
we get a commutative diagram
\begin{tcd}
		\tau_{\leq \ell} \derR \fl \RR_X[\dim X] \big\vert_U \dar \rar&
		H^1(T_b)[-1] \tensor \tau_{\leq \ell-1} \derR \fl \RR_X[\dim X]
			\big\vert_U \dar \\
		\derR \fl \RR_X[\dim X] \big\vert_U \rar&
			H^1(T_b)[-1]  \tensor \derR \fl \RR_X[\dim X] \big\vert_U.
\end{tcd}
On stalks, meaning after applying the functor $\iu$, the truncation functor
$\tau_{\leq \ell}$ produces the perverse filtration; so we get a commutative diagram
of $\RR$-vector spaces
\begin{tcd}
		P_{\ell} H^{\dim X+\ast}(X_b) \dar[hook] \rar&
		H^1(T_b) \tensor P_{\ell-1} H^{\dim X-1+\ast}(X_b) \dar[hook] \\
		H^{\dim X+\ast}(X_b) \rar& H^1(T_b) \tensor H^{\dim X-1+\ast}(X_b).
\end{tcd}
The bottom row is the degree $1$ component of the comultiplication $\delta$ that we are
interested in. The conclusion is that
\[
	\delta_i \bigl( P_{\ell} H^k(X_b) \bigr) \subseteq P_{\ell-1} H^{k-1}(X_b)
\]
for each $i = 1, \dotsc, 2 \dim T_b$. Because the shifted perverse filtration is
defined as $\SP_{\ell} H^k(X_b) = P_{\ell+k} H^k(X_b)$, we obtain the first assertion in
\Cref{prop:shifted-perverse}.

We need to extract one extra piece of information. Using Deligne's decomposition in
\eqref{eq:decomposition}, we can break up the morphism in \eqref{eq:morphism} into a
finite sum of the form
\[
	\gamma = \gamma_{-1} + \gamma_{-2} + \dotsb,
\]
where $\gamma_{\ell}$ maps each summand $P_j$ in the decomposition
of $\derR \fl \RR_X[\dim X]$ into the summand $H^1(T_b) \tensor P_{j+\ell}[-\ell-1]$. In
particular, the topmost component $\gamma_{-1}$ is a morphism of perverse sheaves
\[
	\gamma_{-1} \colon P_j \vert_U \to H^1(T_b) \tensor P_{j-1} \vert_U.
\]
But $P_j$ also has a decomposition by strict support, of the form
\[
	P_j \cong \bigoplus_{Z \subseteq B} P_{j,Z},
\]
where $P_{j,Z}$ is the intersection complex of a local system on a dense open subset
of $Z$. Like any morphism of perverse sheaves, $\gamma_{-1}$ is automatically compatible with this
decomposition, and so it further breaks up into a sum of morphisms
\[
	\gamma_{-1, Z} \colon P_{j,Z} \vert_U \to H^1(T_b) \tensor P_{j-1,Z} \vert_U
\]
between the individual summands. Because the action by $\delta_i$ on the subquotient
$\gr_{\ell}^{\SP} \! H^{\ast}(X_b)$ comes from $\gamma_{-1}$, this proves the second
half of \Cref{prop:shifted-perverse}.
\end{proof}

\subsection{Rational structure and Hodge modules}
\label{sec:rational}

Throughout the paper, we have been consistently working with real coefficients, for two reasons:
\begin{enumerate}
	\item The K\"ahler class $\omega \in H^2(X, \RR)$ is not rational in general.
	Consequently, Deligne's decomposition in the decomposition theorem typically exists
	only over $\RR$.
	\item When $G$ is a meromorphic group, we have
	\[
		H^{\ast}(G, \RR) \cong H^{\ast}(L, \RR) \tensor H^{\ast}(T, \RR)
	\]
	as mixed Hodge structures over $\RR$, but not over $\QQ$. Every comodule over
	the Hopf algebra $H^{\ast}(G, \RR)$ is therefore naturally a comodule over
	$H^{\ast}(T, \RR)$, but this is not true over $\QQ$.
	\end{enumerate}

That said, most of the results are actually true with rational coefficients, but the
resulting isomorphisms are no longer canonical. The decomposition theorem
\[
	\derR \fl \QQ_X[\dim X] \cong \bigoplus_{j=-n}^n P_j[-j]
\]
is of course true over $\QQ$; now each $P_j$ is a perverse sheaf that is part of a rational
Hodge module of weight $\dim X + j$ (such that $P_j \tensor \RR$ is polarizable). The
same is true for the decomposition by strict support
\[
	P_j \cong \bigoplus_{Z \subseteq B} P_{j,Z}.
\]
The proof of the support theorem shows that the local systems $\locC_Z$ and $\locT_Z$
are actually defined over $\QQ$; we then get
\[
	P_{j,Z} \cong \IC_Z \left( \locC_Z \tensor \bigwedge^{r+j} \locT_Z \right),
\]
as perverse sheaves with coefficients in $\QQ$, because this holds after tensoring by $\RR$.

\begin{remark}\label{rmk: Q1}
	When $X$ is a smooth quasi-projective algebraic variety, we can take
	$\omega \in H^2(X, \QQ)$ as the first Chern class of an ample line bundle.
	Deligne's decomposition is then defined over $\QQ$, and the proof of the
	support theorem can be carried out over $\QQ$. To see why, suppose that $Z \subseteq B$
	is a support and $b \in Z$ is a general point. The comodule structure
	$H^{\ast}(X_b, \QQ) \to H^{\ast}(G_b, \QQ) \tensor H^{\ast}(X_b, \QQ)$ over the $\QQ$-Hopf algebra
	$H^{\ast}(G_b, \QQ)$ induces a comodule structure on each summand
	\[
		\bigoplus_{j=-n}^n H^k \iu P_{j,Z},
	\]
	and this factors through $H^{\ast}(T_b, \QQ)$ for weight reasons. Therefore
	the freeness theorem, which holds over $\RR$, implies the freeness of each
	summand over $\QQ$.
\end{remark}

\begin{remark}\label{rmk: Q2}
There is also a version of the freeness theorem over $\QQ$. Suppose that $X$ is a
compact K\"ahler space with a meromorphic action by a meromorphic group $G$, and let $T$
be the maximal compact quotient of $G$. The extension
\[
	0 \to H^1(T, \QQ) \to H^1(G, \QQ) \to H^1(L, \QQ) \to 0
\]
is typically not split as mixed Hodge structures over $\QQ$. But we can simply
choose a splitting in the category of $\QQ$-vector spaces, and together with Hopf's theorem,
this still gives us a (non-canonical) surjective morphism of Hopf algebras
\[
	H^{\ast}(G, \QQ) \to H^{\ast}(T, \QQ).
\]
This allows us to view $H^{\ast}(X, \QQ)$ as a comodule over $H^{\ast}(T, \QQ)$,
and because it becomes free after tensoring by $\RR$, it must be a free comodule.
The drawback is that the resulting isomorphism
\[
	H^{\ast}(X, \QQ) \cong H^{\ast}(T, \QQ) \tensor H^{\ast}(X, \QQ)_{\coinv}
\]
is no longer compatible with mixed Hodge structures; \Cref{ex:freeness-Q} shows
that this is unavoidable.
\end{remark}

Finally, we comment on the relation with Saito's theory. We already mentioned
that the decomposition theorem holds in the category of polarizable Hodge modules
(with coefficients in $\QQ$ or $\RR$). For each support $Z \subseteq B$, the local
systems $\locC_Z$ and $\locT_Z$ are actually polarizable variations of Hodge structure of
type $(n-r,n-r)$ respectively $(1,0) + (0,1)$, and the isomorphism
\[
	P_{j,Z} \cong \IC_Z \left( \locC_Z \tensor \bigwedge^{r+j} \locT_Z \right)
\]
is an isomorphism of Hodge modules of weight
\[
	\dim X + j = \dim Z + 2(n-r) + (j+r).
\]
The polarization on $\locT_Z$ is induced by the construction that we used during the proof
of the support theorem in \Cref{par:support-theorem-proof}. We showed there that
the local system dual to $\locT_Z$ embeds into
\[
	\shHom_{\RR} \bigl( \locL_{-r+1,Z}, \locL_{-r,Z} \bigr),
\]
and so it inherits a natural polarization.

\begin{remark}
At each $b \in B$, the decomposition theorem gives an isomorphism
\[
	H^k(X_b, \RR) \cong \bigoplus_{j=-n}^n H^{k-j} \iu P_j.
\]
Because each $P_j$ is a polarizable real Hodge module, Saito's theory puts a mixed
Hodge structure on the $\RR$-vector space on the right-hand side (by computing the functor
$\iu$ in terms of nearby and vanishing cycles). The left-hand side
also has a mixed Hodge structure, constructed by Deligne and Fujiki. In the algebraic
case, Saito \cite{Saito-complexes} proved that the two mixed Hodge structures agree;
but this seems to be not known in the analytic case. We carefully avoided this
issue during the proof of the support theorem, by establishing the special case in \Cref{lem:MHS}
by hand.
\end{remark}

\section{Results about Lagrangian fibrations}

In this chapter, we prove a few results specifically about Lagrangian fibrations,
and we deduce the support theorem for Lagrangian fibrations from the more general result in the
previous chapter. We assume that $X$ and $B$ are nonsingular; the case of singular $X$
is addressed in \Cref{sec: sing lf}.

\subsection{The \texorpdfstring{$\boldsymbol\delta$-regularity}{delta-regularity} condition}\label{sec: deltarz}

Let $f \colon X \to B$ be a Lagrangian fibration with $\dim X = 2n$ and $\dim B = n$,
and let $p \colon G \to B$ be the family of $n$-dimensional meromorphic groups that we
constructed in \Cref{sec:Liouville-Arnold} and \Cref{sec:meromorphic}.

\begin{proposition}\label{pr: deltareg}
The family of meromorphic groups $p \colon G \to B$ is $\delta$-regular. Concretely, this
means that for every $0 \leq r \leq n$, one has
\[
	\dim \menge{b \in B}{\dim T_b \leq r} \leq r.
\]
\end{proposition}

In the algebraic context, this is well-known, see for example \cite[Prop.~8.9]{Arinkin+Fedorov}.
We give a somewhat different proof, by relating $\delta$-regularity directly to the
Lagrangian fibration via the set of \define{singular cotangent vectors}
\[
	Z_f = \menge{(b, \beta) \in T^{\ast} B}%
	{\text{$\im(T_x X \to T_b B) \subseteq \ker \beta$ for some $x \in X_b$}}.
\]
This makes it clear where the bound on the dimension comes from.

It is relatively easy to show that $\dim Z_f \leq n$. Indeed, $Z_f$ is conical, and
so the image under $p \colon T^{\ast} B \to B$ of any irreducible component $W
\subseteq Z_f$ is a closed analytic subset of
$B$. One then proves, for example by stratifying the map $f$, that $W$ is contained in the conormal variety
$T_{p(W)}^{\ast} B$. This gives $\dim W \leq n$, with equality exactly when
$W$ is the conormal variety of $p(W) \subseteq B$.

\begin{remark}
The set $Z_f$ is also related to the decomposition theorem in
\eqref{eq:decomposition}. By Kashiwara's estimate \cite[Prop.~9.4.2]{Kashiwara+Schapira}
for the characteristic variety of
direct images, the characteristic variety of $\derR \fl \RR_X[2n]$ is contained in
the set $Z_f$. The conclusion is that the $n$-dimensional irreducible components of $Z_f$ are
conormal varieties; and that the conormal varieties $T_Z^{\ast} B$ to the supports in the
decomposition theorem are among them.
\end{remark}

Now let's try to understand what $Z_f$ has to do with $\delta$-regularity. The construction
in \Cref{sec:Liouville-Arnold} gives us, for every $b \in B$, an isomorphism
\[
	T_b^{\ast} B \cong \Lie G_b.
\]
What happens is that the fiber of $Z_f$ over a point $b \in B$ is exactly the Lie
algebra of the linear part $L_b$ of the meromorphic group $G_b$.
This is a straightforward consequence of the Borel fixed point theorem.

\begin{lemma} \label{lem:delta-regularity}
	We have $Z_f = \menge{(b, \beta) \in T^{\ast} B}{\beta \in \Lie L_b}$.
\end{lemma}

\begin{proof}
	The main point is that if we take a holomorphic $1$-form $\beta \in H^0(U, \Omega_B^1)$
	and convert it into a holomorphic vector field $\xi \in H^0(X_U, \shT_X)$ on $X_U = f^{-1}(U)$
	by the rule $\fu \beta = \xi \contr \sigma$, then $\xi$ has a zero at a point $x \in X_U$
	if and only if $\fu \beta$ annihilates $T_x X$ if and only if $\beta$ annihilates
	the image of $T_x X \to T_{f(x)} B$. The first equivalence holds because the symplectic
	form $\sigma$ is by definition non-degenerate.

	Let's prove the more important inclusion ``$\supseteq$''. Take any point $b \in B$.
	The linear part $L_b$ is a commutative linear
	algebraic group that acts meromorphically on $X_b$. By the Borel fixed point theorem
	\cite[Prop~6.9]{fujiki-auto}, it has a fixed point $x \in X_b$. This means that if we take
	a cotangent vector $\beta \in T_b^{\ast} B$ whose image under the isomorphism $\Lie G_b \cong T_b^{\ast} B$
	belongs to $\Lie L_b$, and if we extend $\beta$ to a holomorphic $1$-form on a small
	neighborhood $U$ of the point $b$, then the corresponding holomorphic vector field on $X_U$
	has a zero at the point $x$. As explained above, it follows that $\beta$ annihilates
	the image of $T_x X \to T_b B$, which says exactly that $(b, \beta) \in Z_f$.

	Now let's prove the other inclusion. By definition, $\beta \in T_b^{\ast} B \cap Z_f$ means that
	there is a point $x \in X_b$ such that $f^{\ast}(\beta)$ vanishes on the tangent
	space $T_x X$. As before, extend $\beta$ to a holomorphic $1$-form on a small neighborhood
	$U$ of the point $b$. The associated vector field, constructed using the symplectic form,
	then has a zero at the point $x$, which means that $x$ is a fixed point for
	the flow of the vector field. Under the isomorphism $\Lie G_b \cong T_b^{\ast} B$,
	our cotangent vector $\beta$ therefore maps into the Lie algebra of
	the stabilizer of $x \in X_b$. Because the stabilizer is a linear group by
	\Cref{rmk:stabilizer}, it has finite image in $T_b$; this says that $\beta \in \Lie L_b$.
\end{proof}

This actually gives something slightly stronger than the $\delta$-regularity condition.
Take any closed analytic
subset  $Z \subseteq B$ that is the image of an irreducible component of $Z_f$; for
example, $Z$ could be one of the supports in the decomposition
in \eqref{eq:decomposition}. Over a general point $b \in Z$, the intersection $Z_f \cap
T_b^{\ast} B$ is contained in the conormal bundle to $Z$, and therefore has dimension
$\leq n - \dim Z$. According to \Cref{lem:delta-regularity}, we get $\dim L_b
\leq n - \dim Z$, and therefore $\dim T_b \geq \dim Z$. (In the case where $Z$ is a
support, the support theorem improves this to $\dim T_b = \dim Z$.)

\subsection{The Support Theorem for Lagrangian fibrations with $X$ smooth}

The work in the previous chapters, together with \Cref{pr: deltareg}, now gives us the support theorem for Lagrangian
fibrations. Assume that $X$ is a holomorphic symplectic K\"ahler manifold of dimension $2n$,
and that $f \colon X \to B$ is a Lagrangian fibration over an $n$-dimensional complex
manifold $B$. Let $p \colon G \to B$ be the family of meromorphic groups
constructed in \Cref{sec:meromorphic}.

\begin{theorem} \label{thm:LF-support}
	Let $Z \subseteq B$ be one of the supports in the decomposition theorem for
	the Lagrangian fibration $f \colon X \to B$. Then \Cref{thm:support} is true with
	$r = \dim Z$.
\end{theorem}

\subsection{The neutral component of the family of meromorphic groups}\label{subs: netr}
Here we refine the results in \Cref{subs:neutral-component}.
For a Lagrangian fibration,  the neutral component of the family
of meromorphic groups $p \colon G \to B$ is closely related to the first cohomology
of the fibers. We now explain the connection, which goes
through the ``complex-analytic N\'eron models'' constructed in \cite{sch12}, and which
has a few surprising consequences.

Let $B^{\sm} \subseteq B$ denote the maximal open subset over which the Lagrangian
fibration $f \colon X \to B$
is submersive, and write $f^{\sm} \colon X^{\sm} \to B^{\sm}$ for the restriction.
By construction, we have a holomorphic mapping
\[
	\eps \colon T^{\ast} B \to G
\]
whose image is the family of neutral components $G^{\circ}$ (by \Cref{prop:neutral-component});
we also showed in \Cref{rmk:Lambda} that the preimage of the zero section
$\Lambda = \eps^{-1}(B)$ is equal to the closure of its restriction to $B^{\sm}$.
The smooth fibers $X_b$ of the Lagrangian fibration are $n$-dimensional compact complex
tori, and so
\[
	G_b \cong \Aut^\sm(X_b) \cong \Alb(X_b) \cong
		H^n(X_b, \Omega_{X_b}^{n-1}) \big/ H^{2n-1}(X_b, \ZZ).
\]
Under the isomorphism $H^n(X_b, \Omega_{X_b}^{n-1}) \cong T_b^{\ast} B$, the lattice
$\Lambda_b$ therefore corresponds to the subset $H^{2n-1}(X_b, \ZZ)$, which means that
the restriction of $\Lambda$ to the open subset $B^{\sm}$ is isomorphic to the
\'etal\'e space of the local system $\locHZ^{2n-1} = R^{2n-1} \fl^{\sm} \ZZ_{X^{\sm}}$.

\begin{lemma}\label{lm: neutrzz}
The family of neutral components $p \colon G^{\circ} \to B$
is isomorphic, as a complex manifold, to the ``complex-analytic N\'eron model''
for the integral variation of Hodge structure
on $\locHZ^{2n-1}$, as constructed in \cite{sch12}.
\end{lemma}

\begin{proof}
The dual variation of Hodge structure is $\locHZ^1 = R^1 \fl^{\sm} \ZZ_{X^{\sm}}$,
and its extension to a polarizable Hodge module on $B$ is exactly $P_{-n+1,B}$. By
Matsushita's theorem, the first nontrivial piece of the Hodge filtration is
\[
	F_{-1} \mathscr{P}_{-n+1,B} \cong F_{-n} \mathscr{P}_{n-1,B} \cong R^{n-1} \omega_{X/B}
	\cong \shT_B,
\]
and so the dual, which is used in the construction in \cite[Thm.~A]{sch12}, is the cotangent
bundle $T^{\ast} B$. The ``complex-analytic N\'eron model'' is then
\[
	\bar{J}(\locHZ^{2n-1}) \cong T^{\ast} B / T_{\ZZ},
\]
where $T_{\ZZ}$ is the closed subspace of the cotangent bundle corresponding to the
sheaf of abelian groups $j_{\ast} \locHZ^{2n-1}$; as usual, $j \colon B^{\sm} \to B$
stands for the embedding of the open subset $B^{\sm} = B \setminus \disc(f)$.
\end{proof}

In \cite[Ch.~3]{sch12}, it is proved that the closure of the \'etal\'e space of the local
system $\locHZ^{2n-1}$ inside the bundle $T^{\ast} B$ is the subspace $T_{\ZZ}$. Together with
\Cref{rmk:Lambda}, the conclusion is that $\Lambda \cong T_{\ZZ}$. What this means concretely is
that at any point $b \in B$, the discrete subgroup $\Lambda_b \subseteq T_b^{\ast} B$ is
\begin{equation} \label{eq:Lambda_b}
	\Lambda_b \cong \bigl( j_{\ast} \locHZ^{2n-1} \bigr)_b \cong
	H^0 \bigl(  U \cap B^{\sm}, \locHZ^{2n-1} \bigr);
\end{equation}
here $U \subseteq B$ is a sufficiently small open ball around the point $b \in B$.
So $\Lambda_b$ is a free $\ZZ$-module of rank $\leq 2n$.
Since $G_b^{\circ} = T_b^{\ast} B / \Lambda_b$, we get
\[
	H_1 \bigl( G_b^{\circ}, \ZZ \bigr) \cong \Lambda_b \cong
	H^0 \bigl(  U \cap B^{\sm}, \locHZ^{2n-1} \bigr).
\]
From the action $G_b^{\circ} \times X_b \to X_b$, we obtain as in \Cref{sec:freeness}
a homomorphism
\[
	p \colon H^1(X_b, \ZZ) \to H^1(G_b^{\circ}, \ZZ).
\]
We now give a concrete description of what $p$ does. By proper base change,
\[
	H^1(X_b, \ZZ) \cong (R^1 \fl \ZZ_X)_b
	\to \bigl( j_{\ast} \locHZ^1 \bigr)_b \cong H^0 \bigl(  U \cap B^{\sm}, \locHZ^1 \bigr),
\]
where $U \subseteq B$ is as above. The decomposition theorem shows that this becomes
an isomorphism after tensoring by $\QQ$; because $H^1(X_b, \ZZ)$ is torsion-free,
the map is therefore injective and the image has finite index. From the natural pairing between
the two local systems $\locHZ^1$ and $\locHZ^{2n-1}$, we get a pairing
\[
	H^1(X_b, \ZZ) \tensor_{\ZZ} H_1(G_b^{\sm}, \ZZ) \to
	H^0 \bigl(  U \cap B^{\sm}, \locHZ^{2n-1} \bigr) \tensor_{\ZZ}
	H^0 \bigl(  U \cap B^{\sm}, \locHZ^1 \bigr) \to \ZZ.
\]
This pairing may be viewed as a second homomorphism
\begin{equation} \label{eq:q-pairing}
	H^1(X_b, \ZZ) \to H^1(G_b^{\circ}, \ZZ).
\end{equation}
Not surprisingly, the two homomorphisms are the same; this may seem hard to prove at
first, but it turns out to be relatively easy.

\begin{lemma}
	The homomorphism in \eqref{eq:q-pairing} agrees with $p \colon H^1(X_b, \ZZ) \to
	H^1(G_b^{\circ}, \ZZ)$.
\end{lemma}

\begin{proof}
Take a class $\alpha \in H^1(X_b, \ZZ)$. If $U \subseteq B$ is a small open ball around the
point $b \in B$, we can extend $\alpha$ to a class $\alphat
\in H^1 \bigl( f^{-1}(U), \ZZ \bigr)$, for example by the proper base change theorem
or because $f^{-1}(U)$ deformation retracts onto $X_b$. By construction, the
restriction of $\alphat$ to nearby smooth fibers $X_{b'}$ is
nontrivial. Let $\lambda \in \Lambda_b$ be an arbitrary element. By
\eqref{eq:Lambda_b}, we can extend $\lambda$ uniquely to a section $\lambdat$ of
the constructible sheaf $j_{\ast} \locHZ^{2n-1}$ over $U$; we can also view $\lambdat$ as a
section of $p \colon \Lambda \to B$ over the open set $U$. From the embedding of $\Lambda$
into $T^{\ast} B$, we get an associated holomorphic $1$-form on
$U$, and this in turn determines a holomorphic vector field on $f^{-1}(U)$ that is
tangent to the fibers of $f$. Let $\Phi_t \colon f^{-1}(U) \to f^{-1}(U)$ denote the
flow of this vector field. Because $\lambdat$ takes values in $\Lambda$, the construction
of $G^{\circ}$ shows that $\Phi_1$ is the identity. (The reason is that it is the
identity on the nonsingular fibers, and these are dense in $f^{-1}(U)$.)

We can use $\Phi_t$ to understand the pullback morphism $H^1(X_b, \ZZ) \to H^1(G_b^{\circ},
\ZZ)$ induced by the action $a_b \colon G_b^{\circ} \times X_b \to X_b$. The induced map on homology
\begin{equation} \label{eq:action-homology}
	a_{b,\ast} \colon \Lambda_b \cong H_1(G_b^{\circ}, \ZZ) \to H_1(X_b, \ZZ)
\end{equation}
is very easy to describe. Choose any point $x_0 \in X_b$, for example a smooth point of
the reduction $X_{b, \red}$. As $t \in [0,1]$ moves from $t=0$ to $t=1$, the image
$\Phi_t(x_0)$ of the point $x_0$ under the flow of the vector field traces out a
$1$-dimensional closed loop on $X_b$; its class is the image of $\lambda \in
\Lambda_b$ under \eqref{eq:action-homology}. To get the induced map on cohomology, we
only need to evaluate $\alpha \in H^1(X_b, \ZZ)$ on the homology class $a_{b, \ast}(\lambda)$. Let's denote this evaluation by
\[
	\bigl\langle \alpha, a_{b, \ast}(\lambda) \bigr\rangle \in \ZZ.
\]
As $X$ is a manifold at the point $x_0$, we can take any nearby smooth fiber $X_{b'}$
with $b' \in U \cap B^{\sm}$, and then choose a point $x_1 \in X_{b'}$ that is close to
$x_0$ in $X$. By the same construction, we get a homology class $a_{b',
\ast}(\lambda') \in H_1(X_{b'}, \ZZ)$, where $\lambda' = \lambdat(b') \in
\Lambda_{b'}$. On a smooth fiber, $\Lambda_{b'} \cong H_1(X_{b'}, \ZZ)$ is of course
just the first homology group of the compact complex torus $X_{b'}$.

We can evaluate $\alpha$ on the homology class $a_{b, \ast}(\lambda)$ using nearby
smooth fibers. Even when $X_b$ is nonreduced, we can still find a continuous path
$\gamma \colon [0,1] \to X$ with the property that $\gamma(0) = x_0$ and $\gamma(1) =
x_1$, and such that $f \circ \gamma$ is injective. (This means that $\gamma$
intersects every fiber only once; one cannot do this with a holomorphic section, but
for continuous paths, there is no problem.) We can then apply our flow $\Phi_t$ to this path. The result is a
continuous mapping
\[
	[0,1] \times [0,1] \to X, \quad (s, t) \mapsto \Phi_t \bigl( \gamma(s) \bigr),
\]
whose image is a surface $\Sigma$ with boundary; $\partial \Sigma$ is the loop traced
out by $\Phi_t(x_1)$ in $X_{b'}$, minus the loop traced out by $\Phi_t(x_0)$ in
$X_b$. This shows that the homology class $a_{b, \ast}(\lambda) \in H_1(X_b, \ZZ)$
and the homology class $a_{b', \ast}(\lambda') \in H_1(X_{b'}, \ZZ)$ have the same image in
$H_1 \bigl( f^{-1}(U), \ZZ \bigr)$. Consequently, the cohomology class $\alphat \in
H^1 \bigl( f^{-1}(U), \ZZ \bigr)$ has the same value on both homology classes.
But $\Phi_t$ is just translation in the direction of $\lambda'$ on the compact complex
torus $X_{b'}$, and so we have $a_{b', \ast}(\lambda') = \lambda'$. Taken together,
this says that
\[
	\bigl\langle \alpha, a_{b, \ast}(\lambda) \bigr\rangle
	= \bigl\langle \alpha', \lambda' \bigr\rangle
\]
where $\alpha' \in H^1(X_{b'}, \ZZ)$ denotes the restriction of $\alphat$ to
$X_{b'}$. The second pairing is just the pairing between the monodromy invariant
cohomology class $\alpha' \in H^1(X_{b'}, \ZZ)$ and the monodromy invariant
homology class $\lambda' \in H_1(X_{b'}, \ZZ)$ on the nearby smooth fiber $X_{b'}$,
and this is enough to conclude that the homomorphism
\[
	a_b^{\ast} \colon H^1(X_b, \ZZ) \to H^1(G_b^{\circ}, \ZZ)
\]
agrees with the one defined in \eqref{eq:q-pairing}.
\end{proof}

The relation between $G^{\circ}$ and the ``complex-analytic N\'eron model'' has
the following interesting consequence. Recall from \Cref{lem:delta-regularity} that
\[
	Z_f = \bigcup_{b \in B} \Lie L_b \subseteq T^{\ast} B,
\]
and that the conormal varieties $T_Z^{\ast} B$ to the supports $Z \subseteq B$ in the
decomposition theorem are among the $n$-dimensional irreducible components of $Z_f$.

\begin{proposition}
The set $Z_f$ can be reconstructed from the variation of Hodge structure $\locHZ^{2n-1}$
over the open subset $B^{\sm} \subseteq B$.
\end{proposition}

This gives an a priori bound for the possible supports of all Lagrangian fibrations
with the same variation of Hodge structure on the first cohomology of their fibers. Indeed,
if $f \colon X \to B$ is any Lagrangian fibration that has nonsingular fibers over
$B^{\sm}$ and such that $R^1 \fl^{\sm} \ZZ_{X^{\sm}} \cong \locHZ^{2n-1}$, then
the proposition gives us a locally finite set of subvarieties of $B$, namely those whose
conormal variety appears in the set $Z_f$, that contains all the supports in the decomposition theorem
for $f \colon X \to B$.

\begin{proof}
This follows from \cite[Lem.~2.25]{sch12}, which identifies the maximal compact quotient
of $G_b^{\circ} \cong T_b^{\ast} B / \Lambda_b$ with the maximal compact quotient of
the generalized intermediate Jacobian associated to the mixed Hodge structure on
$H^{-n} \iu P_{n-1,B}$. This is uniquely determined by the variation of Hodge structure,
because the rational Hodge module $P_{n-1,B}$ is the intersection
complex of $\locHZ^{2n-1} \tensor_{\ZZ} \QQ$.
\end{proof}

The following example by Hwang and Oguiso \cite{hwang-ogu09} shows that the homomorphism
$H^1(X_b, \ZZ) \to (j_{\ast} \locHZ^1)_b$ is usually not an isomorphism over $\ZZ$.

\begin{example}
Let $E$ be the elliptic curve with Weierstrass equation $y^2 = x^3 +1$, and let $F$ be some other elliptic curve.
If $\zeta$ is a $6$-th root of unity, and if $a \in F$ is a point of order $6$, then
$\mu_6 \cong \ZZ/6\ZZ$ acts on the product
\[
	E \times F \times \CC \times \CC
\]
by mapping $(x,y,z,s,t)$ to $(\zeta^2 x, \zeta^3 y, z + a, s, \zeta t)$.
The quotient is a holomorphic symplectic fourfold with a Lagrangian fibration
over $\CC^2$ (with coordinates $s$ and $u=t^6$), and all the fibers over $u=0$ are
nonreduced. A computation shows that the first cohomology of the singular fibers is
isomorphic to $H^1(F / \langle a \rangle, \ZZ)$, whereas the monodromy-invariant part
of the first cohomology of a nearby fiber is isomorphic to $H^1(F, \ZZ)$. So in this
example, $H^1(X_b, \ZZ) \subseteq (j_{\ast} \locHZ^1)_b$ is a subgroup of index $6$.
\end{example}

\begin{remark}\label{rmk:hitchin}
We are grateful to Mirko Mauri for pointing out what follows.
Let $C$ be a smooth, connected, projective curve of genus at least $2$, let $G$ be a reductive group, let $d \in \pi_1(G)$, and let
$f\colon X \to B$ be the corresponding Hitchin fibration, where $X$ is the moduli space of semistable $G$-Higgs pairs $(E,\theta)$ with $E$ a $G$-torsor of topological type $d$ on $C$ and $\theta$ a section of $\mathrm{ad}(E)\otimes\omega_C$, and where $B$ is the Hitchin base, a vector space; see \cite{decataldo2025hitchinfibrationsngofibrations}. For $G=\mathrm{GL}_r$, the invariant $d$ is the degree of the associated vector bundles. The Hitchin morphism is a Lagrangian fibration with smooth base and, in general, singular total space. It comes equipped with the smooth commutative group scheme $P^\circ/B$ with connected fibers acting on the fibration (cf.~\cite[Prop.~4.17]{decataldo2025hitchinfibrationsngofibrations}).

Since $P^\circ$ acts on $X$, there is a homomorphism of group schemes
\[
P^\circ \longrightarrow \Aut(X/B).
\]
By construction, this factors through $G$, because $G$ is the largest subgroup of $\Aut(X/B)$ with irreducible total space. Since $P^\circ$ has connected fibers, it further factors through $G^\circ$, yielding a morphism of group schemes
\[
u\colon P^\circ \longrightarrow G^\circ .
\]

We claim that $u$ is an isomorphism. Let $B^{\mathrm{sm}}\subset B$ denote the locus over which the Hitchin fibration is smooth. Over $B^{\mathrm{sm}}$, the fibration is simultaneously a $P^{\circ,\mathrm{sm}}$-torsor and a $G^{\circ,\mathrm{sm}}$-torsor. Hence $u$ is birational.

Moreover, for both $P^\circ$ (cf.~\cite[Prop.~7.14]{decataldo2025hitchinfibrationsngofibrations}) and $G^\circ$ (by the construction via vertical vector fields), the differential of the action is identified, via the symplectic form, with the coderivative of the Hitchin fibration. Consequently, the induced morphism of Lie algebras is an isomorphism at every point $b\in B$ for which the fiber $X_b$ contains a smooth point of $f$. For the Hitchin fibration this locus is open, and its complement has codimension at least two. Since the Lie algebras of $P^\circ$ and $G^\circ$ are vector bundles on the smooth base $B$, the induced morphism of Lie algebras is therefore an isomorphism everywhere on $B$. As both $P^\circ$ and $G^\circ$ are smooth over $B$, it follows that $u$ is \'etale.

Since $u$ is separated and $G^\circ$ is nonsingular, hence normal, Zariski's Main Theorem implies that $u$ is an open immersion. Its image is a subgroup scheme of $G^\circ$. Therefore, fiberwise, it is an open subgroup of the connected group $G^\circ_b$, and hence coincides with $G^\circ_b$. Thus $u$ is surjective, and therefore an isomorphism.

We note that in particular, the action of $P^\circ$ is thus  faithful.
\end{remark}

\begin{remark}\label{rmk:via-neron}
Here we reach the same conclusion of \Cref{rmk:hitchin} for $G=\operatorname{GL}_r$, but via the use of N\'eron
models.
Let $\mathscr C \to B$ be the universal family of spectral curves of degree $r$ over $C$. Then the group scheme $P^\circ$ is the group scheme $\Pic^\circ(\mathscr C/B) \to B$ of fiberwise neutral Picard groups acting  naturally on $X$ via tensor product with Higgs bundles. Again, we claim that $u: P^\circ \to G^\circ$ is an isomorphism.
It is enough to show this over codimension $1$ points in $B$ (cf. \cite[Cor~IX.1.4]{ray70}). Over codimension $1$ points, we can use the algebraic N\'eron model theory. Since the total space $\mathscr C$ is a smooth variety, \cite[Thm~9.5.4]{neron} shows that there exists a N\'eron model $P'$ containing $P^\circ$ as a subgroup. Since $P'$ is a N\'eron model, we have a homomorphism $G^\circ \to P'$ such that the composition $P^\circ \to G^\circ \to P'$ is an open immersion. Hence $P^\circ \to G^\circ$ is an open immersion, and therefore an isomorphism.
\end{remark}

\begin{remark}
	In the above, we have discussed the subgroup $G^\circ \subset G$. We point out that it is sometimes possible to extend $G$ to a larger group $G'$, the (algebraic) N\'eron model of $G^{\sm} = G|_{B^{\sm}}$ in the sense of \cite{neron} and \cite[Def~6.1]{hol-mol-ore-poi23}. Note that $G'$ is the terminal object among the smooth groups $H \to B$ with $H|_{B^{\sm}} = G^{\sm}$. The construction of such $G'$ is presented in \cite[\S 5--6.1]{kim24} but under the condition \eqref{eq:non-multiple fiber condition} in the introduction. We point out that, with the result in this paper, the condition \eqref{eq:non-multiple fiber condition} can be relaxed with essentially the same proof there. Let us quickly summarize this. All the theorem numbers below refer to those in \cite{kim24}.

	Assume that $X$ is a smooth symplectic variety, $B$ is a smooth variety, and $f : X \to B$ is a projective Lagrangian fibration. Take the $\delta$-regular group scheme $G$ and consider the set
	\[ B^\sharp = \menge{b \in B}{\dim T_b \ge n-1} .\]
	Notice a sequence of open subsets $B^{\sm} \subseteq B^\sharp \subseteq B$ and that $B^\sharp$ is a big open subset of $B$ by the $\delta$-regularity of $G$. Let us consider a new condition
	\begin{equation} \label{eq:non-multiple fiber condition in codimension 1}
		f(X_{\mathrm{nc}/B}) \supset B^\sharp .
	\end{equation}
	In this setup, we show that the N\'eron model $G'$ of $G^{\sm}$ exists under the milder condition \eqref{eq:non-multiple fiber condition in codimension 1} than the condition \eqref{eq:non-multiple fiber condition} imposed in their original Theorem~6.1.

	The first step is to drop the condition \eqref{eq:non-multiple fiber condition} in Theorem~5.2(2): we show that $G^\sharp = G|_{B^\sharp} \to B^\sharp$ is a N\'eron model of $G^{\sm}$ (but only over $B^\sharp$). The idea of the original proof was as follows. Up to shrinking $B$, the statement is equivalent to showing that every birational automorphism $\phi : X \dashrightarrow X$ over $B$ is a regular automorphism over $B^\sharp$. The undefined locus $\operatorname{Undef}(\phi)$ is uniruled, $G^\circ$-invariant, and has dimension $\le 2n-2$. But by the $\delta$-regularity, any $G^\circ$-invariant subvariety of $X$ with dimension $\le 2n-2$ contains no rational curves over $B^\sharp$, so this means $\operatorname{Undef}(\phi)$ is empty over $B^\sharp$. This argument does not use the condition \eqref{eq:non-multiple fiber condition} at all, so it can be applied verbatim to show our claim.

	The second step is to consider the relative Picard algebraic space $\Pic(X/B) \to B$. Its connected component $\Pic^\bullet(X/B)$ is a smooth group space over $B$ as shown in Lemma~6.9. To proceed, we need to drop the condition \eqref{eq:non-multiple fiber condition} in Proposition~6.10: we show that $\Pic^\bullet(X/B)$ is a N\'eron model of its restriction over $B^\sharp$. Take two five-term exact sequences as below, where the vertical homomorphisms between them are restriction homomorphisms:
	\[\begin{tikzcd}
		0 \arrow[r] & \Pic(B) \arrow[r] \arrow[d] & \Pic(X) \arrow[r] \arrow[d] & \Pic(X/B) \arrow[r] \arrow[d] & \operatorname{Br}(B) \arrow[r] \arrow[d] & \operatorname{Br}(X) \arrow[d] \\
		0 \arrow[r] & \Pic(B^\sharp) \arrow[r] & \Pic(X^\sharp) \arrow[r] & \Pic(X^\sharp/B^\sharp) \arrow[r] & \operatorname{Br}(B^\sharp) \arrow[r] & \operatorname{Br}(X^\sharp)
	\end{tikzcd}.\]
	The four vertical arrows except for the middle one are isomorphisms because the Picard and Brauer groups of a smooth variety are invariant under taking codimension $2$ complements. The five-lemma concludes the middle vertical arrow is an isomorphism, which implies that $\Pic^\bullet(X/B)$ is a N\'eron model of its restriction over $B^\sharp$.

	The third step is to generalize Lemma~6.12 to construct the N\'eron model $P^\sharp \to B^\sharp$ of $\Pic^\bullet (X^{\sm}/B^{\sm})$. Only the last paragraph of its proof needs to be modified as follows. Take the same restriction diagram as above, but this time from $B^\sharp$ to $B^{\sm}$. Restriction homomorphisms of the Picard and Brauer groups of a smooth variety are always surjective and injective, respectively. Hence the second vertical arrow is surjective, and the fourth and fifth arrows are injective. Shrinking $B^\sharp$ \'etale locally, we can further assume that the fourth arrow is surjective \cite[Thm~3.7.1(ii)]{col-the-sko:brauer_group}. Hence the four-lemma concludes the surjectivity of the homomorphism $\Pic(X^\sharp/B^\sharp) \to \Pic(X^{\sm}/B^{\sm})$. Now letting $P^\sharp$ be the separated quotient group of $\Pic^\bullet (X^\sharp/B^\sharp)$, we can conclude that $P^\sharp(B^\sharp) \to P^\sharp (B^{\sm})$ is an isomorphism, showing $P^\sharp$ is a N\'eron model. This drops the condition \eqref{eq:non-multiple fiber condition} from Lemma~6.12.

	The final step is to generalize Theorem~6.1 to construct the N\'eron model $G'$ of $G^{\sm}$. The proof applies verbatim and we only recall its idea. Start with a polarization homomorphism $G^{\sm} \to \Pic^\bullet (X^{\sm}/B^{\sm})$ using an $f$-ample line bundle on $X$. Since $P^\sharp$ is a N\'eron model of $\Pic^\bullet (X^{\sm}/B^{\sm})$, this extends to a homomorphism $G^\sharp \to P^\sharp$, or a homomorphism $G^\sharp \to (P^\sharp)^\circ$ up to composing it with a sufficiently divisible multiplication map. At this point, the condition \eqref{eq:non-multiple fiber condition in codimension 1} guarantees $(P^\sharp)^\circ = (\Pic^\bullet (X^\sharp/B^\sharp))^\circ$, so we can extend the homomorphism again to $G \to \Pic^\bullet(X/B)$ because $\Pic^\bullet(X/B)$ is a N\'eron model. Now use Corollary~3.27 in the paper to construct the N\'eron model $G'$.
\end{remark}

\subsection{Lagrangian fibrations with section}

When $f \colon X \to B$ has a holomorphic section
$s \colon B \to X$ with $f \circ s = \id$, one can realize the family of meromorphic groups $G$ as a Zariski-open
subset of $X$. Indeed, the composition
\begin{tcd}
G \rar & G \times_B B \rar{\id \times s} & G \times_B X \rar{a} & X
\end{tcd}
defines a holomorphic mapping $i \colon G \to X$, and we claim that this is a Zariski-open
embedding. Because $G$ and $X$ are nonsingular and $\dim G = \dim X$, and because
$i$ is the restriction of a meromorphic mapping $D \mto X$,  it is enough to show
that $i$ is injective. If we let $K = i^{-1} \bigl( s(B) \bigr)$,
then injectivity of $i$ is equivalent to $K = e(B)$. The proof of this is similar to
\Cref{rmk:Lambda}. The existence of the section $s \colon B \to X$ means that $X$ is a complex
manifold of dimension $\dim B + n$ at all points of $s(B)$. Therefore $s(B)$ is locally
defined by $n$ holomorphic functions on $X$, and so every irreducible component
of $K$ has codimension $\leq n$, hence dimension $\geq \dim B$. But $K_b = G_b \cap K$
is the stabilizer of the point $s(b)$, and because $f \colon X \to B$ is smooth at the point $s(b)$,
\Cref{pr: deltareg}  shows that $\Lie L_b = \{0\}$, and hence that the stabilizer is a discrete subgroup of $G_b$. This gives
$\dim K = \dim B$, and for dimension reasons, $K$ must again
be the closure of its restriction to the smooth locus of $f$. When $X_b$ is a compact
complex torus and $G_b = \Aut^{\circ}(X_b)$, we have $K_b = \{e\}$, and so $K = e(B)$.

\section{Singular Lagrangian fibrations}\label{sec: sing lf}

In this chapter, we explain how to adapt the proof of the support theorem to Lagrangian
fibrations on singular holomorphic symplectic complex spaces in Beauville's sense
\cite{Beauville}. Beauville works with algebraic varieties, but his definition carries over
unchanged to complex spaces.

\begin{definition}
	A normal complex space $X$ is \define{holomorphic symplectic} if there is
	a closed nondegenerate holomorphic $2$-form $\sigma \in H^0(\Xreg, \Omega_{\Xreg}^2)$
	on the regular locus, such that $\sigma$ extends holomorphically to
	any resolution of singularities of $X$.
\end{definition}

The condition that $\sigma$ extends holomorphically to any resolution is
equivalent to $X$ having at worst \define{rational
singularities}, which means that $R^i \mu_{\ast} \shO_{\Xt} = 0$ for any
resolution of singularities $\mu \colon \Xt \to X$ and any $i > 0$. This is the content
of a theorem by Namikawa \cite[Thm.~4]{Namikawa}. As in the nonsingular case, the dimension of a holomorphic symplectic space
is always even; we usually set $\dim X = 2n$.

\begin{definition}
	Let $X$ be a holomorphic symplectic complex space of dimension $2n$.
	A \define{Lagrangian fibration} is a proper holomorphic mapping
	$f \colon X \to B$ to a normal complex space $B$ of dimension $n$, such that:
	\begin{enumerate}
		\item $f$ is surjective with connected fibers: $\fl \OX \cong \OB$
		\item A general fiber $X_b = f^{-1}(b)$ is Lagrangian: $\dim X_b = n$ and $\sigma \vert_{X_b} = 0$
	\end{enumerate}
\end{definition}

In the rest of the chapter, we describe how to modify the proofs from the earlier
chapters in order to extend the results to the singular case. We keep the assumption that $B$ is a complex manifold, which is needed
in several places.

\begin{theorem} \label{thm:singLF-support}
	Let $f \colon X \to B$ be a Lagrangian fibration on a holomorphic symplectic
	complex space, with $B$ a complex manifold.  Let $Z \subseteq B$ be one of the
	supports in the decomposition theorem for $\derR \fl \IC_X$, where $\IC_X$ is the (perverse)
	intersection complex on $X$.  Then \Cref{thm:support} holds with $r = \dim Z$.
\end{theorem}

\subsection{Equidimensionality and flatness}

One can describe holomorphic symplectic spaces in the language of Poisson structures
\cite{Kaledin}, which is more convenient in the singular setting.
Recall that a \define{Poisson algebra}
(over $\CC$) is a commutative $\CC$-algebra $A$ with unit $1 \in A$, together with a
skew-symmetric bilinear (over $\CC)$ pairing $\{\argbl, \argbl\} \colon A \tensor A \to A$ (called the Poisson bracket) such that
\begin{align*}
	\{a, bc\} &= \{a, b\} c + b \{a, c\} \quad \text{and} \quad \{ a, 1\} = 0; \\
	0 &= \{a, \{b, c\}\} + \{b, \{c, a\}\} + \{c, \{a, b\}\}.
\end{align*}
The conditions on the first line are saying that $\{a, \argbl\}$ is a derivation of $A$;
the  condition on the second line is the Jacobi identity. A complex space $X$ is called a \define{Poisson space} if its
structure sheaf $\OX$ is a sheaf of Poisson algebras. If $f \in \Gamma(U, \OX)$ is a
holomorphic function, then $H_f = \{f, \argbl\}$ is a holomorphic vector field on $U$,
called the \define{Hamiltonian vector field} associated to $f$.

\begin{example}
A holomorphic symplectic space $X$ is automatically a Poisson space. Let
$\sigma \in H^0( \Xreg, \Omega_{\Xreg}^2)$ be the symplectic form on the regular locus
of $X$.  The Poisson bracket of two holomorphic functions $f,g \in \Gamma(U, \OX)$ is
defined as follows. Let $H_f \in \Gamma(U, \shT_X)$ be the unique section of the reflexive
sheaf $\shT_X = \shHom_{\OX}(\Omega_X^1, \OX)$ such that $df = H_f \contr \sigma$ on $\Ureg$, and define
\[
	\{f, g\} = H_f(g) \in \Gamma(U, \OX).
\]
Note that these holomorphic functions automatically extend over the singular locus
of $X$ because $X$ is normal. Since $d \sigma = 0$,
the bracket $\{\argbl, \argbl\}$ satisfies the Jacobi identity, and so $X$ is a Poisson space.
\end{example}

Matsushita \cite{mat00} showed that Lagrangian fibrations on holomorphic symplectic
K\"ahler manifolds are equidimensional. The same thing is
true for any Lagrangian fibration $f \colon X \to B$ on a holomorphic symplectic complex space,
as long as $X$ is K\"ahler.
Under the assumption that $B$ is a complex manifold, it follows that $f$ is flat
(because $X$ has rational singularities, hence is Cohen-Macaulay).

\begin{proposition} \label{prop:equidimensional}
Let $f \colon X \to B$ be a Lagrangian fibration on a holomorphic symplectic complex space.
If $X$ is K\"ahler, then $f$ is equidimensional.
\end{proposition}

\begin{proof}
This result is due to Schwald \cite[Thm.~17]{Schwald} for projective $X$. We briefly
describe how to adapt it to the case when $X$ is K\"ahler but not necessarily compact.

The most important input is a theorem by Kaledin \cite[Thm.~2.3]{Kaledin}, proved using
Poisson geometry. Recall that a closed subspace $Z$ of a Poisson space
$X$ is called a \define{Poisson subspace} if its ideal sheaf $\shI_Z$ is a Poisson ideal,
meaning that $\{f, g\} \in \shI_Z$ for every $f \in \OX$ and every $g \in \shI_Z$. Kaledin
proves that a holomorphic symplectic complex space has a canonical locally finite stratification,
in which each stratum $S$ is nonsingular and holomorphic symplectic; the Poisson subspaces
of $X$ are exactly the closures $\overline{S}$ of strata; and the normalization of each
$\overline{S}$ is again a holomorphic symplectic space. Moreover, $X$ is locally trivial along
each stratum, and the stratification is a Whitney stratification. (Kaledin actually proves the
local triviality only for a formal neighborhood, but a recent paper by Kaplan and Schedler \cite[Appendix]{Kaplan+Schedler}
shows that this ``holomorphic Weinstein splitting theorem'' holds locally in the analytic topology.)

The next step is a result by Matsushita \cite[Thm.~3.1]{mat15}, which he calls the
``Theorem of Matreshka'': for each stratum $S \subseteq X$,
one has $\dim S = 2 \dim f(S)$, and the morphism from the normalization of $\overline{S}$
to the normalization of $\overline{f(S)}$ is again a Lagrangian fibration. Matsushita's
proof uses algebraic language (such as function fields), but the argument easily
carries over
to complex spaces. Indeed, the key inequality $2 \dim f(S) \leq \dim S$ can be proved
analytically as follows. Choose a general point $x \in S$; then the
mapping $T_x S \to T_{f(x)} f(S)$ is surjective. Set $k = \dim f(S)$, and choose
local coordinates $t_1, \dotsc, t_n$ on $B$ in such a way that $t_1, \dotsc, t_k$
restrict to coordinates on $f(S)$. The Hamiltonian vector fields
$\{\fu t_1, \argbl\}, \dotsc, \{\fu t_k, \argbl\}$ are linearly independent at $x$
and tangent to the fibers of $f \vert_S \colon S \to f(S)$, and so the fiber dimension is
at least $k$; this gives $\dim S \geq k + \dim f(S) = 2k$.

By induction on the dimension,  the morphism from the normalization of $\overline{S}$
to the normalization of $\overline{f(S)}$ is equidimensional for every stratum $S
\subseteq X_{\sing}$. Since $\dim S < \dim X$, every irreducible component $Y$ of every
fiber of $f \colon X \to B$ therefore has to intersect $X_{\reg}$. To prove that $f$
is equidimensional, it is thus enough to show that $\sigma$ restricts to zero on
$Y \cap X_{\reg}$. The argument for this is the same as in \cite{mat00}.
Let $\mu \colon \Xt \to X$ be a resolution of singularities such
that $\Xt$ is K\"ahler manifold, $\mu$ is an isomorphism over $X_{\reg}$, and the
preimage $\tilde{Y} = \mu^{-1}(Y)$ is nonsingular. Because $X$ has rational singularities,
\[
	\derR \mu_{\ast} \shO_{\Xt} \cong \OX
	\cong \omega_X \cong \derR \mu_{\ast} \omega_{\Xt},
\]
and so Koll\'ar's theorem, extended to the K\"ahler setting by Takegoshi \cite{Takegoshi},
implies that the higher direct image sheaf
\[
		R^2 \fl \OX \cong R^2 (f \circ \mu)_{\ast} \shO_{\Xt}
\]
is torsion-free on $B$. Let $\tilde{\sigma} \in H^0(\Xt, \Omega_{\Xt}^2)$ be the
holomorphic extension of the pullback of $\sigma$; then $d \tilde{\sigma} = 0$.
The complex conjugate of $\tilde{\sigma}$ is a closed $(0,2)$-form, and so it
defines a global section of the sheaf $R^2 \fl \OX$. This section vanishes over the
smooth locus of $f$ (because $f$ is a Lagrangian fibration), hence is trivial.
Consequently, the conjugate of $\tilde{\sigma}$ restricted to $\tilde{Y}$ is
a $\bar{\partial}$-exact $(0,2)$-form, and because  $\tilde{Y}$ is a compact K\"ahler
manifold, it follows that $\tilde{\sigma} \vert_{\tilde{Y}} = 0$, hence
$\sigma \vert_{Y \cap X_{\reg}} = 0$.
\end{proof}

\subsection{The family of meromorphic groups}

Let $X$ be a holomorphic symplectic space of dimension $2n$, and consider a Lagrangian
fibration  $f \colon X \to B$ over an $n$-dimensional complex manifold $B$.
As in \Cref{sec:Liouville-Arnold}, there is a commutative diagram
\begin{tcd}
T^{\ast} B \times_B X \rar{\Phi} \dar{p_2} & X \dar{f} \\
X \rar{f} & B.
\end{tcd}
We briefly summarize the construction, in the
language of Poisson spaces. Choose local coordinates $t_1, \dotsc, t_n \in \Gamma(U, \OB)$
on $B$, and let $\xi_j = \{\fu t_j, \argbl\} \in \Gamma(X_U, \shT_X)$ be the
associated Hamiltonian vector fields on $X_U = f^{-1}(U)$; these commute and are tangent
to the regular locus of $X_{b, \red}$ for any $b \in U$. Since $X$ is reduced, the
results by Kaup \cite{Kaup} still apply, and so we can integrate the vector fields into a holomorphic
mapping $\Phi_U \colon \CC^n \times X_U \to X_U$. The rest of the argument is unchanged.
In particular, \Cref{lem:faithful-action} about the injectivity of the map
$T_b^{\ast} B \to \Lie \Aut^{\circ}(X_b)$ continues to hold: as we saw in the previous section,
the singular locus of $X$ intersects every irreducible component of
$X_b = f^{-1}(b)$ in a proper subset,
and so we can choose a point $x \in \Xreg \cap X_b$ and run the argument  there.

The conclusion of \Cref{sec:meromorphic} is that the action by the cotangent
bundle determines a meromorphic action
\[
	a \colon G \times_B X \to X
\]
by a family of meromorphic groups $p \colon G \to B$, whose family of neutral components
is again a quotient of $T^{\ast} B$. This only uses the nonsingularity of $B$ and the flatness
of $f \colon X \to B$, which follow from our assumptions on $X$ and $B$.

\begin{example}
Let $\tilde f : \tilde X \to B$ be a minimal elliptic fibration over an affine line $B = \CC$ with one singular fiber $\tilde X_b = C_0 \cup C_1 \cup C_2$ of Kodaira type $\mathrm I_3$. Let $\tilde X \to X$ be the contraction of $C_0$ to an ordinary double point. Both $\tilde f$ and $f$ are Lagrangian fibrations of symplectic complex spaces $\tilde X$ and $X$, one smooth and the other singular. Therefore, we obtain two family of meromorphic groups $\tilde G \to B$ and $G \to B$. Kodaira and N\'eron showed $\tilde G_b = \CC^* \times \ZZ/3$. Let us show $G_b = \CC^*$, demonstrating that the group $G$ depends on the choice of a birational model.

The group $\tilde G$ is a N\'eron model, so we obtain a unique homomorphism $G \to \tilde G$. This means that the $G$-action uniquely lifts to $\tilde X$. Note that $X_b$ is a union of two rational curves $C_1$ and $C_2$, and $G_b$ acts generically freely on $X_b$. Hence $G_b$ is either $\CC^*$ or $\CC^* \times \ZZ/2$. Assume on the contrary $G_b = \CC^* \times \ZZ/2$; this means $\ZZ/2$ swaps $C_1$ and $C_2$. The $G_b$-action lifts to $\tilde X_b$ by the above group homomorphism $G_b = \CC^* \times \ZZ/2 \to \tilde G_b = \CC^* \times \ZZ/3$, where $\ZZ/2$ is forced to be mapped into $\CC^*$. This means the $\ZZ/2$-action on $\tilde X_b$ preserves all three irreducible components $C_0$, $C_1$, and $C_2$. This contradicts the previous fact that it swaps $C_1$ and $C_2$.
\end{example}

\subsection{The \texorpdfstring{$\boldsymbol\delta$-regularity}{delta-regularity} condition}

Perhaps suprisingly, the proof of $\delta$-regularity is almost the same as
in the smooth case. Because each $G_b = p^{-1}(b)$ is a commutative meromorphic group, we have
a short exact sequence
\[
	1 \to L_b \to G_b \to T_b \to 1
\]
with $T_b$ a compact complex torus and $L_b$ (meromorphically isomorphic to)
a linear algebraic group. Our goal is to prove that
\begin{equation} \label{eq:delta-resolution}
	\dim \menge{b \in B}{\dim L_b \geq r} \leq n - r
		\quad \text{for every $0 \leq r \leq n$.}
\end{equation}
Let $\mu \colon \Xt \to X$ be a functorial resolution of singularities; this exists
for complex spaces by \cite{Wlodarczyk}. Denote by
\[
	Z_{\ft} = \menge{(b, \beta) \in T^{\ast} B}%
	{\text{$\im(T_x \Xt \to T_b B) \subseteq \ker \beta$ for some $x \in \Xt_b$}}
\]
the set of singular cotangent vectors for the morphism $\ft = f \circ \mu
\colon \Xt \to B$. As in the smooth case, the key point for $\delta$-regularity
is the following inclusion.

\begin{lemma}
	Inside the cotangent bundle $T^{\ast} B$, we have
	\[
		\bigcup_{b \in B} \Lie L_b \subseteq Z_{\ft}.
	\]
\end{lemma}

\begin{proof}
The symplectic form $\sigma$ on $\Xreg$ extends to a holomorphic
$2$-form $\sigmat$ on the resolution; of course, $\sigmat$ will no longer be nondegenerate.
Take any point $b \in B$, and let $t_1, \dotsc, t_n$ be local holomorphic coordinates
on an open neighborhod $U \subseteq B$. Set $X_U = f^{-1}(U)$ and $\Xt_U = \ft^{-1}(U)$.
Let $\xi_j = \{\fu t_j, \argbl\}$ be the Hamiltonian vector field corresponding to the
function $\fu t_j$.
By functoriality of the resolution, $\xi_j$ lifts to a holomorphic vector field $\xit_j$
on $\Xt_U$. On $X_{U, \reg}$, we have $\xi_j \contr \sigma = \fu dt_j$; by the identity theorem,
it follows that $\xit_j \contr \sigmat = \ft^{\ast} dt_j$. The important property that
carries over from the smooth case is that if the vector field $\xit_j$ vanishes at some point
$x \in \Xt_U$,  then the $1$-form $\ft^{\ast} dt_j$ also vanishes there,
which means that $dt_j$ annihilates the image of $T_x \Xt \to T_b B$.

The action morphism $a \colon G \times_B X \to X$ is smooth, and because the resolution
algorithm is functorial for smooth morphisms, we get an induced morphism
$\tilde{a} \colon G \times_B \Xt \to \Xt$; one can also get this from \cite[Lem.~2.5]{fujiki-auto}.
Since $\mu \colon \Xt \to X$ is bimeromorphic,
this induced action on $\Xt$ is still meromorphic. Over any point $b \in B$, the
linear algebraic group $L_b$ acts meromorphically on the compact K\"ahler space $\Xt_b$,
and by the Borel fixed point theorem, there is a fixed point $x \in \Xt_b$.
At this point, the vector field corresponding to any $\beta \in \Lie L_b$ must therefore
have a zero. According to the computation above, it follows that the cotangent vector
$\ft^{\ast} \beta$ vanishes on $T_x \Xt$, which means exactly that $(b, \beta) \in Z_{\ft}$
is in the set of singular cotangent vectors.
\end{proof}

Now $\delta$-regularity is immediate: we have $\dim Z_{\ft} \leq n$, and so \eqref{eq:delta-resolution}
holds, for the same reason as in \Cref{sec: deltarz}.

\subsection{Equivariant constructible complexes and freeness}

In this section, we prove a few technical results about (weakly) equivariant constructible
complexes, as well as a more general version of the freeness theorem.
Let $X$ be a compact complex space, and let $a \colon G \times X \to X$
be a holomorphic action by a complex Lie group $G$. Denote by $\Dbc(X, \RR)$ the derived
category of constructible complexes of sheaves of $\RR$-vector spaces on $X$.

\begin{definition} \label{def:equivariant}
	A constructible complex $K \in \Dbc(X, \RR)$ is called \define{weakly equivariant}
	if there is an isomorphism $\phi \colon \au K \to \pu_2 K$ such that the diagram
	\[
		\begin{tikzcd}[column sep=small, row sep=large]
		& (m \times \id)^{\ast} \au K \dlar[bend right=20,swap]{(m \times \id)^{\ast} \phi} \rar[equal] &
			(\id \times a)^{\ast} \au K \drar[bend left=20]{(\id \times a)^{\ast} \phi} \\
		(m \times \id)^{\ast} \pu_2 K \drar[bend right=20,equal] &&& (\id \times a)^{\ast}
		\pu_2 K \dlar[bend left=20,equal] \\
		& \pu_{23} \pu_2 K & \pu_{23} \lar[swap]{\pu_{23} \phi} \au K
		\end{tikzcd}
	\]
	on $G \times G \times X$ commutes. The arrows marked ``equal'' are the usual isomorphisms
	coming from the identities $a \circ (\id \times a) = a \circ (m \times \id)$, etc.
\end{definition}

Let $K \in \Dbc(X, \RR)$ be a weakly equivariant constructible complex. Since $X$ is compact, the
cohomology $H^{\ast}(X, K)$ is a finite-dimensional graded $\RR$-vector space. From
$\phi \colon \au K \to \pu_2 K$ and the projection formula \cite[Prop.~2.6.6]{Kashiwara+Schapira}
(which applies here because $X$ is compact and $K$ is constructible),  we obtain
\[
	H^{\ast}(G \times X, \au K) \cong H^{\ast}(G \times X, \pu_2 K)
		\cong H^{\ast}(G) \tensor H^{\ast}(X, K).
\]
The compatibility condition in the definition guarantees that the resulting map
\[
	\delta_K \colon H^{\ast}(X, K) \to H^{\ast}(G) \tensor H^{\ast}(X, K)
\]
makes $H^{\ast}(X, K)$ into a comodule over the Hopf algebra $H^{\ast}(G)$. At the same
time, $H^{\ast}(X, K)$ is also a graded module over the cohomology algebra $H^{\ast}(X)$.
The reason is that pullback by the diagonal embedding $\Delta \colon X \to X \times X$ defines a map
\[
	H^{\ast}(X) \tensor H^{\ast}(X, K) \cong H^{\ast}(X \times X, \pu_2 K)
		\to H^{\ast}(X, K).
\]
Just as in the case of constant coefficients, the module and comodule structure are
compatible: for $h \in H^j(X)$ and $m \in H^k(X, K)$, one has
\[
	\delta_K(h \cdot m) = \delta(h) \cdot \delta_K(m),
\]
where $\delta \colon H^{\ast}(X) \to H^{\ast}(G) \tensor H^{\ast}(X)$ is the usual comodule
structure on the cohomology of the compact complex space $X$. The proof is the same as
in the smooth case, using the commutative diagram
\begin{tcd}
	G \times X \rar{a} \dar{\Delta_{G \times X}} & X \dar{\Delta_X} \\
	G \times X \times G \times X \rar{a \times a} & X \times X.
\end{tcd}

Now suppose that the compact complex space $X$ is connected and K\"ahler, that $G$ is
a meromorphic group, and that the action $a \colon G \times X \to X$ is meromorphic.
Let $q \colon G \to T$ be the maximal compact quotient, as in \eqref{eq:Chevalley}.
As in \Cref{sec:freeness}, we can choose a section $s \colon H^1(T) \to H^1(X)$,
which turns $H^{\ast}(X, K)$ into a Hopf module over the Hopf algebra $H^{\ast}(T)$.
The abstract freeness theorem in \Cref{thm:freeness} therefore implies the following
result.

\begin{proposition}\label{pr: K weq}
Under these assumptions,
\[
	H^{\ast}(X, K) \cong H^{\ast}(T) \tensor H^{\ast}(X, K)_{\coinv}
\]
is free as a comodule over $H^{\ast}(T)$.
\end{proposition}

During the proof of the support theorem, we are going to need to know how the
comodule structure interacts with duality.

\begin{remark} \label{rem:f!}
We use the fact that if $f \colon X \to Y$ is a smooth morphism between
complex spaces, hence a topological submersion, then the morphism of functors
\[
	\fu(\argbl) \tensor f^! \RR_Y \to f^!(\argbl)
\]
is an isomorphism \cite[Thm.~3.3.2]{Kashiwara+Schapira}. Note that $f^! \RR_Y \cong
\RR_X \decal{2c}$, because the fibers of $f$ are complex manifolds of dimension $c = \dim X
- \dim Y$, hence oriented.
\end{remark}

Let $K' = \DD_X K \in \Dbc(X, \RR)$ be
the Verdier dual of $K$. It is again weakly equivariant: the isomorphism $\phi' \colon \au K'
\to \pu_2 K'$ is obtained as the composition
\begin{tcd}
	\au \DD K \rar{\cong} & \DD a^! K \rar{\cong} &
	\DD \au K[2n] \rar{\DD\phi} &
	\DD \pu_2 K[2n] \rar{\cong} & \pu_2 \DD K.
\end{tcd}
Here we used the fact that duality exchanges the two functors $\au$ and $a^!$, and that
$a \colon G \times X \to X$ is smooth of relative dimension $n = \dim G$. The cohomology of
$K' = \DD K$ is therefore also a comodule
\[
	\delta_{K'} \colon H^{\ast}(X, K') \to H^{\ast}(G) \tensor H^{\ast}(X, K')
\]
over $H^{\ast}(G)$. Since $X$ is compact, Verdier duality gives
\begin{equation} \label{eq:Verdier}
	H^i(X, K') \cong \Hom_{\RR} \bigl( H^{-i}(X, K), \RR \bigr),
\end{equation}
and so $H^{\ast}(X, K')$ is isomorphic, as a graded $\RR$-vector space, to the dual of
$H^{\ast}(X, K)$. Not surprisingly, the two comodule structures are dual to each other.

\begin{lemma} \label{lem:dual-comodule}
Under \eqref{eq:Verdier}, the two comodules $H^{\ast}(X, K)$ and $H^{\ast}(X, K')$ are dual.
\end{lemma}

\begin{proof}
Let's briefly review duality for comodules. If $M$ is a finite-dimensional graded
$\RR$-vector space, the natural grading on the dual space $\Hom(M, \RR)$ is
\[
	\Hom(M, \RR)^k = \Hom(M^{-k}, \RR).
\]
Now suppose that $M$ is a comodule over a Hopf algebra $H$. From the comodule structure
$\delta \colon M \to H \tensor M$, we obtain a bilinear pairing
\[
	M \tensor \Hom(M, \RR) \to H, \quad
	m \tensor f \mapsto (\id \tensor f)(\delta(m)),
\]
by considering the composition
\begin{tcd}
	M \rar{\delta} & H \tensor M \rar{\id \tensor f} & H \tensor \RR \rar{\cong} & H,
\end{tcd}
as an $\RR$-linear mapping from $M$ to $H$. After swapping the role of $M$ and $\Hom(M, \RR)$
and reversing the construction, we obtain the desired comodule structure
\[
	\delta' \colon \Hom(M, \RR) \to H \tensor \Hom(M, \RR)
\]
on the dual space.

To prepare for the proof, we should express the duality in \eqref{eq:Verdier} in more categorical terms.
Consider the object $K \boxtimes K' = \pu_1 K \tensor \pu_2 \DD K$ on $X \times X$.
As $X$ is compact, the projection formula \cite[Prop.~2.6.6]{Kashiwara+Schapira} gives
\[
	H^{\ast}(X, K) \tensor H^{\ast}(X, K') \cong
	H^{\ast}(X \times X, K \boxtimes K').
\]
Let $\Delta \colon X \to X \times X$ be the diagonal embedding. From
\[
	K \boxtimes K' \to \derR \Delta_{\ast} \Delta^{\ast} (K \boxtimes K')
	\cong \derR \Delta_{\ast}(K \tensor K'),
\]
we obtain, upon taking cohomology, a homomorphism
\[
	\mu \colon H^{\ast}(X, K) \tensor H^{\ast}(X, K') \to H^{\ast}(X, K \tensor K').
\]
By \cite[Def.~3.1.16]{Kashiwara+Schapira}, we have $K' = \DD_X K
= \derR \shHom(K, p^! \RR)$, where $p \colon X \to \pt$ is the morphism to a point.
From \cite[Prop.~3.4.4]{Kashiwara+Schapira}, we get a natural morphism $K \tensor K'
\to p^! \RR$. Putting everything together, we have a homomorphism
\begin{equation} \label{eq:Verdier-pairing}
	H^{\ast}(X, K) \tensor H^{\ast}(X, K') \to H^{\ast}(X, p^! \RR)
	\cong \Hom \bigl( H^{\ast}(X), \RR \bigr) \to \RR,
\end{equation}
by composing with the dual of the inclusion $\RR \to H^0(X) \to H^{\ast}(X)$. Then
\eqref{eq:Verdier-pairing} is a perfect pairing, which corresponds to the duality
isomorphism in \eqref{eq:Verdier}.

Now we are ready to prove the assertion. Let $\sigma \colon X \times X \to X \times X$
be the automorphism $\sigma(x,y) = (y,x)$. Consider the commutative diagram
\begin{equation} \label{eq:diagram-annoying}
\begin{tikzcd}
	G \times X \rar{\id \times \Delta} \dar{\id \times \Delta}
		& G \times X \times X \dar{(a \times \id) \circ \sigma} \\
	G \times X \times X \rar{a \times \id} & X \times X.
\end{tikzcd}
\end{equation}
The comodule structure on $H^{\ast}(X, K)$ comes from $K \to \derR \al \au K$.
Since
\[
	(a \times \id)^{\ast}(K \boxtimes K')
	\cong \pu_{12} \au K \tensor \pu_3 K',
\]
the morphism on cohomology induced by $a \times \id$ is equal to
\[
	\delta_K \tensor \id \colon H^{\ast}(X, K) \tensor H^{\ast}(X, K') \to
	H^{\ast}(G) \tensor H^{\ast}(X, K) \tensor H^{\ast}(X, K').
\]
But since $\au K \cong \pu_2 K$, we also have
\[
	\pu_{12} \au K \tensor \pu_3 K'
	\cong \pu_2 K \tensor \pu_3 K'
	\cong \pu_{23}(K \boxtimes K'),
\]
and so the morphism on cohomology induced by $\id \times \Delta$ is
\[
	\id \tensor \mu \colon H^{\ast}(G) \tensor H^{\ast}(X, K) \tensor H^{\ast}(X, K')
	\to H^{\ast}(G) \tensor H^{\ast}(X, K \tensor K').
\]
We now compose this with the projection  $H^{\ast}(X, K \tensor K')\to \RR$ from
\eqref{eq:Verdier-pairing}. The result is that pulling back along
$(\id \times \Delta) \circ (a \times \id)$ gives us a bilinear pairing
\[
	H^{\ast}(X, K) \tensor H^{\ast}(X, K') \to H^{\ast}(G);
\]
and under the isomorphism in \eqref{eq:Verdier}, this is exactly the bilinear
pairing associated to the comodule structure on $H^{\ast}(X, K)$.
We can apply the same construction to the other two edges of the commutative diagram
in \eqref{eq:diagram-annoying}, to see that
\[
	H^{\ast}(X, K) \tensor H^{\ast}(X, K') \to H^{\ast}(G)
\]
is also the bilinear pairing associated to the comodule structure on $H^{\ast}(X, K')$.
This is enough to conclude that the two comodules are dual to each other.
\end{proof}

The lemma translates into another very useful compatibility condition among comodules. We
continue to assume that $K \in \Dbc(X, \RR)$ is a weakly equivariant constructible complex.
Recall that the comodule structure
\[
	\delta_K \colon H^{\ast}(X, K) \to H^{\ast}(G) \tensor H^{\ast}(X, K),
\]
is induced by the composition
\begin{tcd}
	K \rar & \derR \al \au K \rar{\derR \al \phi} & \derR \al \pu_2 K
\end{tcd}
upon taking cohomology. In a similar way, the composition
\begin{tcd}
	\derR a_! \pu_2 K \decal{2n} \rar{\derR a_! \phi^{-1}} &
	\derR a_! \au K \decal{2n} \rar{\cong} & \derR a_! a^{!} K \rar & K
\end{tcd}
induces, after passing to cohomology and using the projection formula, a morphism
\[
	H_c^{2n + \ast}(G) \tensor H^{\ast}(X, K) \to H^{\ast}(X, K).
\]
The claim is that these two morphisms are the same if we use
\begin{equation}  \label{eq:duality-cohomology}
	\Hom_{\RR} \bigl( H_c^{2n-\ast}(G), \RR \bigr) \cong H^{\ast}(G)
\end{equation}
to move the first factor from one side to the other.

\begin{corollary} \label{cor:comodules}
Under \eqref{eq:duality-cohomology}, the morphism
\[
	H_c^{2n + \ast}(G) \tensor H^{\ast}(X, K) \to H^{\ast}(X, K)
\]
turns into the comodule structure $\delta_K \colon H^{\ast}(X, K) \to
H^{\ast}(G) \tensor H^{\ast}(X, K)$.
\end{corollary}

\begin{proof}
Let $\DD K \in \Dbc(X, \RR)$ be the Verdier dual; we saw above that this is
weakly equivariant. Verdier duality interchanges $\derR a_! a^{!} K \to K$
and $\DD K \to \derR \al \au \DD K$, and so an equivalent formulation is that
the induced comodule structure on
\[
	H^{\ast}(X, \DD K) \cong \Hom_{\RR} \bigl( H^{\ast}(X, K), \RR \bigr)
\]
is dual to the comodule structure on $H^{\ast}(X, K)$. This follows from \Cref{lem:dual-comodule}.
\end{proof}

\subsection{The Support Theorem for Lagrangian fibrations with $X$ singular}

Let $f \colon X \to B$ be a Lagrangian fibration on a holomorphic symplectic space
$X$ of dimension $2n$. The decomposition theorem
\[
	\derR \fl \IC_X \cong \bigoplus_{j \in \ZZ} P_j[-j]
\]
holds for the intersection complex (which is the unique extension of $\RR_{\Xreg}[2n]$
to a simple perverse sheaf on $X$). The reason is that $\IC_X$ is part of a polarized
Hodge module on $X$, and that the decomposition theorem is true for direct images of
polarized Hodge modules under proper morphisms from K\"ahler manifolds.
Surprisingly, the amplitude bound is exactly the same as in the smooth case.

\begin{lemma}
	One has $P_j = 0$ for $\abs{j} > n$.
\end{lemma}

\begin{proof}
In the algebraic case, this is due to Mauri and Migliorini \cite[Prop.~3.3]{Mauri+Migliorini}.
The two key ingredients are Kaledin's stratification of $X$ by Poisson subspaces (which is
also a Whitney stratification), and
the theorem of Matreshka; we already explained during the proof of \Cref{prop:equidimensional}
that both are still true in the analytic setting.
With this in place, the proof of \cite[Prop.~3.6]{Mauri+Migliorini} goes through; together with \cite[Prop.~3.3]{Mauri+Migliorini},
the conclusion is that $P_j = 0$ for $j > n$. We then get $P_j = 0$ for $j < -n$ by the
relative Hard Lefschetz theorem.
\end{proof}

Denote by $p \colon G \to B$ the family of meromorphic groups constructed from the
action by the cotangent bundle $T^{\ast} B$, and let $a \colon G \times_B X \to X$ be the
meromorphic action. We observe that the intersection complex $\IC_X$ is weakly equivariant.

\begin{lemma}
	The intersection complex $\IC_X$ is weakly equivariant, in the sense that there is a natural
	isomorphism $\phi \colon \au \IC_X \to \pu_2 \IC_X$ on $G \times_B X$ that satisfies
	the compatibility condition in \Cref{def:equivariant}.
\end{lemma}

\begin{proof}
	The morphism $a \colon G \times_B X \to X$ is smooth of relative dimension $n$, and so
	\[
		\au \IC_X \cong \IC_{G \times_B X}[n];
	\]
	the isomorphism is uniquely determined by the condition that, over
	$\Xreg$, it should agree with the usual isomorphism between $\au \RR_{\Xreg}$ and
	$\RR_{G \times_B \Xreg}$. Similarly, $p_2 \colon G \times_B X \to X$ is also smooth of
	relative dimension $n$, and so
	\[
		\pu_2 \IC_X \cong \IC_{G \times_B X}[n],
	\]
	where the isomorphism is again the natural one. This gives $\phi \colon \au \IC_X
	\to \pu_2 \IC_X$. The compatibility condition in \Cref{def:equivariant} holds
	because everything is determined by what happens over $\Xreg$.
\end{proof}

Let $b \in B$ be any point, and write $G_b = p^{-1}(b)$ and $T_b$ for its maximal
compact quotient. After restriction to a fiber $X_b = f^{-1}(b)$, we find that
the constructible complex $\IC_X \vert_{X_b}$ is weakly equivariant on the fiber $X_b$.
The more general version of the freeness
theorem therefore applies, with the result that
\[
	H^{\ast} \bigl( X_b, \IC_X \vert_{X_b} \bigr)
\]
is free as a comodule over $H^{\ast}(T_b)$.

The other ingredient we need for the proof of the support theorem is the compatibility
of the comodule structure with the (shifted) perverse filtration. We define this exactly
as in \Cref{def:P}, using the truncation functors $\tau_{\leq \ell}$ for the perverse
$t$-structure on $\Dbc(B, \RR)$. By the decomposition theorem, the morphism
\[
	H^k \iu \bigl( \tau_{\leq \ell} \derR \fl \IC_X \bigr)
	\to H^k \iu \bigl( \derR \fl \IC_X \bigr)
	\cong H^k(X_b, \IC_X \vert_{X_b})
\]
is injective, and we denote its image by the symbol $P_{\ell} H^k(X_b, \IC_X \vert_{X_b})$.
We again define the \define{shifted perverse filtration} as
\[
	\SP_{\ell} H^k(X_b, \IC_X \vert_{X_b}) = P_{k+\ell-2n} H^k(X_b, \IC_X \vert_{X_b}).
\]
We need to know that the shifted perverse filtration is compatible with the
comodule structure on $H^{\ast}(X_b, \IC_X \vert_{X_b})$, and for that, we need a version
of \Cref{prop:shifted-perverse} for weakly equivariant constructible complexes.
Let $f \colon X \to B$ be a proper holomorphic mapping between two complex manifolds,
let $p \colon G \to B$ be a family of meromorphic groups, and suppose that we have a
meromorphic action
\[
	a \colon G \times_B X \to X.
\]
Let $K \in \Dbc(X, \RR)$ be a weakly equivariant constructible complex in \Cref{def:equivariant}. As in \Cref{prop:shifted-perverse},
we consider the constructible complex
\[
	K_G = \derR p_! \RR_G[2n] \in \Dbc(B, \RR),
\]
which is (by \Cref{lem:K}) isomorphic to a sum of shifts of constructible sheaves. We can
now imitate the construction in \Cref{prop:shifted-perverse}.

\begin{lemma} \label{lem:morphism-K}
	There is a morphism
	\[
		K_G \tensor \derR \fl K \to \derR \fl K
	\]
	in the derived category $\Dbc(B, \RR)$.
\end{lemma}

\begin{proof}
Recall from \Cref{rem:f!} that if $f \colon X \to Y$ is a smooth morphism between
complex spaces, then the exceptional pullback functor satisfies
$f^!(\argbl) \cong \fu(\argbl)[2c]$, where $c = \dim X - \dim Y$.
Consider again the commutative diagram
\begin{tcd}
		X \dar{f} & G \times_B X \rar{a} \lar[swap]{p_2} \dar{p_1} & X \dar{f} \\
		B & G \rar{p} \lar[swap]{p} & B.
\end{tcd}
The morphism $a$ is smooth of relative dimension $n$, and so it induces
\begin{tcd}
	\derR a_! \pu_2 K \decal{2n} \rar{\derR a_! \phi^{-1}} &
	\derR a_! \au K \decal{2n} \rar{\cong} & \derR a_! a^{!} K \rar & K.
\end{tcd}
Apply the functor $\derR f_!$ to turn this into a morphism
\[
	\derR p_! \derR (p_1)_! \, \pu_2 K[2n] \to \derR f_! K.
\]
Using proper base change \cite[Prop.~2.6.7]{Kashiwara+Schapira}, we have
\[
	\derR (p_1)_! \, \pu_2 K  \cong \pu \derR f_! K,
\]
and the projection formula \cite[Prop.~2.6.6]{Kashiwara+Schapira} therefore implies that
\[
	\derR p_! \derR (p_1)_! \, \pu_2 K \cong \derR p_! \RR_G \tensor \derR f_! K.
\]
Putting everything together and remembering that $K_G = \derR p_! \RR_G[2n]$ and
$\derR f_! K = \derR \fl K$ because $f$ is proper, we finally arrive at
the desired morphism.
\end{proof}

By the same reasoning as in \Cref{prop:shifted-perverse}, we obtain a morphism
\begin{equation} \label{eq:equivariant-morphism}
	\derR \fl K \vert_U \to H^{\ast}(G_b) \tensor \derR \fl K \vert_U
\end{equation}
on a small neighborhood $U \subseteq B$ of a given point $b \in B$. We need to prove
that on stalks, we get back the comodule structure that we defined above.
Denote by $K_b = K \vert_{X_b}$ the restriction of the weakly equivariant complex $K \in \Dbc(X, \RR)$
to the fiber $X_b = f^{-1}(b)$. By proper base change,
\[
	\derR \fl K \vert_b \cong H^{\ast}(X_b, K_b),
\]
and since $K_b$ is weakly equivariant on $X_b$, the construction in the previous section shows that
$H^{\ast}(X_b, K_b)$ is a comodule over the Hopf algebra $H^{\ast}(G_b)$.

\begin{lemma} \label{lem:comodule-stalks}
	On stalks, the morphism in \eqref{eq:equivariant-morphism} recovers the comodule
	structure
	\[
			H^{\ast}(X_b, K_b) \to H^{\ast}(G_b) \tensor H^{\ast}(X_b, K_b).
	\]
\end{lemma}

\begin{proof}
The construction in \Cref{lem:morphism-K} is compatible with proper base change, and
the morphism $H_{\ast}(G_b) \tensor \RR_U \to K_G \vert_U$ is an isomorphism at the
point $b$. This reduces the problem to the case $G_b \times X_b \to X_b$ and the
weakly equivariant complex $K_b$, which is covered by \Cref{cor:comodules}.
\end{proof}

The remainder of the argument in \Cref{par:support-theorem-proof} then goes
through unchanged, and so we obtain the support theorem for the direct image
\[
	\derR \fl \IC_X \cong \bigoplus_{j=-n}^n P_j[-j]
\]
also in the singular case. We note the following technical result; the notation is
the same as in \Cref{par:support-theorem-proof}.

\begin{lemma}
At each point $b \in U$, the subspace $(\locL_{j,Z})_b \subseteq H^{j-r}(X_b,
\IC_X \vert_{X_b})$ is a mixed Hodge substructure that is pure of weight $2n+j-r$.
\end{lemma}

\begin{proof}
By \cite[Prop.~3.3]{Mauri+Migliorini}, the natural morphism $\RR_X[2n] \to \IC_X$ in
the derived category $\Dbc(X, \RR)$ induces an isomorphism
\[
	H^{2n}(X_b) \cong H^0(X_b, \IC_X \vert_{X_b}).
\]
The vector space on the right therefore has a pure Hodge structure of type $(n,n)$,
whose dimension is again the number of irreducible components of the fiber $X_b$. As before,
it follows that the subspace
\[
	(\locL_{r,Z})_b \subseteq H^0(X_b, \IC_X \vert_{X_b})
\]
is a pure Hodge structure of type $(n,n)$. Let $\mu \colon \Xt \to X$ be a functorial
resolution of singularities by a K\"ahler manifold $\Xt$.
By the decomposition theorem applied to $\derR \mu_{\ast} \RR_{\Xt}[2n]$, the
$\RR$-vector space  $H^j(X_b, \IC_X \vert_{X_b})$ is a direct summand in
$H^{2n+j}(\Xt_b)$, where $\Xt_b = (f \circ \mu)^{-1}(b)$. By functoriality, the group action
lifts to $G_b \times \Xt_b \to \Xt_b$, and so the cohomology $H^{\ast}(\Xt_b)$ is also
a comodule over $H^{\ast}(T_b)$ in a compatible way. This gives us a commutative diagram
\begin{tcd}
	(\locL_{r,Z})_b \rar{\delta_T} \dar[hook] & H^{r-j}(T_b) \tensor
	(\locL_{j,Z})_b \dar[hook] \\
	H^0(X_b, \IC_X \vert_{X_b}) \rar{\delta_T} \dar[hook] &
	H^{r-j}(T_b) \tensor H^{j-r}(X_b, \IC_X \vert_{X_b}) \dar[hook] \\
	H^{2n}(\Xt_b) \rar{\delta_T} & H^{r-j}(T_b) \tensor H^{2n+j-r}(\Xt_b)
\end{tcd}
in which the bottom arrow is a morphism of mixed Hodge structures. We can now conclude
the proof in the same way as before.
\end{proof}

\renewbibmacro{in:}{%
  \ifentrytype{article}{}{\printtext{\bibstring{in}\intitlepunct}}}
\printbibliography
\end{document}